\documentclass[twoside]{article}

\usepackage[abbrvbib, preprint]{jmlr2e}

\usepackage{marginnote}
\usepackage{xcolor}
\usepackage{float}
\usepackage[subrefformat=parens]{subcaption}
\usepackage{booktabs}
\usepackage{amssymb}
\usepackage{amstext}
\usepackage{amsmath}
\usepackage{amscd}
\usepackage{amsfonts}
\usepackage{enumerate}
\usepackage{graphicx}
\usepackage{latexsym}
\usepackage{mathrsfs}
\usepackage{mathtools}
\usepackage{cases}
\usepackage{verbatim}
\usepackage{bm}
\usepackage{caption}
\usepackage{wrapfig}
\usepackage{natbib}

\usepackage{tikz}
\usetikzlibrary{matrix}
\usepackage{hyperref} 
    \hypersetup{
    	colorlinks = true,
    	linkcolor = blue,
    	anchorcolor = blue,
    	citecolor = blue,
    	filecolor = blue,
    	urlcolor = blue
    }

\usepackage[colorinlistoftodos]{todonotes}

\NewDocumentCommand{\F}{o}{
    \IfValueT{#1}{
            \mathbb{F}_{#1}
        }
    \IfValueF{#1}{
            \mathbb{F}
        }
                    }
\renewcommand{\P}{{\mathbb{P}}}

\NewDocumentCommand{\R}{o}{
    \IfValueT{#1}{
            \mathbb{R}^{#1}
        }
    \IfValueF{#1}{
            \mathbb{R}
        }
                    }
                    
\NewDocumentCommand{\N}{o}{
    \IfValueT{#1}{
            \mathbb{N}^{#1}
        }
    \IfValueF{#1}{
            \mathbb{N}
        }
                    }

\newcommand{ \eqdef }{
    \ensuremath{\stackrel{\mbox{\upshape\tiny def.}}{=}}
}

\usepackage{cleveref}
\newcounter{termcounter}
\renewcommand{\thetermcounter}{\Roman{termcounter}}
\crefname{term}{term}{terms}
\creflabelformat{term}{#2\textup{(#1)}#3}

\makeatletter
\def\term{\@ifnextchar[\term@optarg\term@noarg}
\def\term@optarg[#1]#2{%
  \textup{#1}%
  \def\@currentlabel{#1}%
  \def\cref@currentlabel{[][2147483647][]#1}%
  \cref@label[term]{#2}}
\def\term@noarg#1{%
  \refstepcounter{termcounter}%
  \textup{(\thetermcounter)}%
  \cref@label[term]{#1}}
\makeatother

\definecolor{warmblack}{rgb}{0.0, 0.26, 0.26}

\definecolor{jade}{rgb}{0.0, 0.40, 0.26}

\NewDocumentCommand{\add}{m}{{#1}}

\definecolor{alizarin}{rgb}{0.82, 0.1, 0.26}
\definecolor{cadmiumgreen}{rgb}{0.0, 0.42, 0.24}

\usepackage[skins,breakable]{tcolorbox}
\makeatletter
\newsavebox{\bsde@headingbox}
\newcommand{\bsde@wraplongheading}{%
  \let\bsde@originalopargbegintheorem\@opargbegintheorem
  \def\@opargbegintheorem##1##2##3{%
    \sbox{\bsde@headingbox}{\bfseries ##1\ ##2\ (##3)}%
    \ifdim\wd\bsde@headingbox>\linewidth
      \trivlist
      \item[\hskip\labelsep
        {\parbox[t]{\dimexpr\linewidth-\labelsep\relax}
          {\bfseries ##1\ ##2\ (##3)}}]%
      \itshape
    \else
      \bsde@originalopargbegintheorem{##1}{##2}{##3}%
    \fi
  }%
}
\tcbset{bsdewrapheading/.code={\bsde@wraplongheading}}
\makeatother
\tcbset{canonicalgreytheorem/.style={
    enhanced, breakable,
    colback=black!5, colframe=black!75, coltext=.,
    boxrule=0.5mm, arc=1mm, boxsep=1mm,
    left=4mm, right=4mm, top=2mm, bottom=2mm,
    before skip=12pt, after skip=12pt,
    before upper={\tcbset{bsdewrapheading}}
}}
\newif\iftheoremboxes
\theoremboxestrue 

\iftheoremboxes
\fi
\definecolor{brightmaroon}{rgb}{0.76, 0.13, 0.28}

\begin{document}

\title{Polynomial Scaling is Possible For Neural Operator Approximations of Structured Families of BSDEs}

\author{\name Takashi Furuya\thanks{Equal contribution, all authors are listed in alphabetic order. } \email tfuruya@mail.doshisha.ac.jp \\
        \addr Doshisha University, Japan
        \AND
        \name Anastasis Kratsios\footnotemark[1] \thanks{Corresponding author.} \email kratsioa@mcmaster.ca \\
        \addr Department of Mathematics\\
        McMaster University and Vector Institute, Canada}

\editor{Daniel Roy}

\maketitle


\begin{abstract}
Neural operator (NO) architectures learn nonlinear maps between infinite-dimensional function spaces and are widely used to accelerate simulation and enable data-driven model discovery.
While universality results ensure expressivity, they do not address \emph{complexity}: for broad operator classes described only through regularity (e.g.\ uniform continuity or $C^r$-regularity), information-theoretic lower bounds imply that minimax-optimal NO approximation rates scale \emph{exponentially} in the reciprocal accuracy $1/\varepsilon$. This has shifted the focus of NO theory toward identifying additional problem-specific structure, beyond regularity, under which suitably tailored NO architectures can leverage to unlock polynomial scaling in $1/\varepsilon$.

We exhibit the first polynomial-scaling regime for NO approximations of solution operators in stochastic analysis; by identifying structured families of \emph{non-Markovian} BSDEs with randomized terminal condition parameterized by the Sobolev-regular terminal condition and by Sobolev-regular additive nonlinear perturbations of the generator.  We prove that their solution operator can be approximated (uniformly over the family) by a tailored NO whose number of trainable parameters grows \emph{polynomially} in $1/\varepsilon$.  We unlock this polynomial scaling regime by \emph{informing the NO's inductive bias} by factoring out the singular part of the associated semilinear elliptic PDE Green's function and by incorporating the Dol\'{e}ans--Dade exponential of the BSDE's common non-Markovian factor into the NO's decoding layers.  As a byproduct, we extend polynomial-scaling guarantees from families of linear elliptic PDEs on regular domains to the semilinear setting.
\end{abstract}




\section{Introduction}
\label{s:Intro}

Neural operators (NOs) have recently emerged as deep-learning tools for learning maps between infinite-dimensional Banach spaces, offering the prospect of amortizing the computational cost of classical numerical solvers during evaluation.
Although there is a myriad of universal approximation results for various NO architectures, e.g.~\cite{chen1993approximations,lu2019deeponet,kovachki2021universal,korolev2022two,galimberti2022designing,benth2023neural,lanthaler2023operator,furuya2023globally,acciaio2024designing,lanthaler2025nonlocality,cuchiero2026global}, each of which guarantees that any continuous nonlinear operator between suitable Banach spaces can be uniformly approximated on compact input sets, the main obstruction is information-theoretic in nature: minimax lower bounds~\cite{lanthaler2025curse} show that, in the worst case, achieving uniform approximation error at most $\varepsilon$ can require a number of trainable parameters scaling exponentially in $\varepsilon^{-1}$.  This has shifted the focus of operator-learning theory away from general guarantees of approximability towards identifying structured problem classes for which tailored NOs provably achieve \emph{polynomial scaling} in $\varepsilon^{-1}$.

\subsection{The NO Approximation Landscape: Universality vs.\ Polynomial Scaling}

\begin{wrapfigure}{r}{0.3\textwidth}
    \vspace{-0.8em}
    \centering
    \includegraphics[width=1\linewidth]{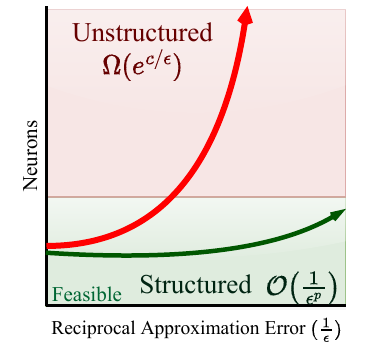}
    \vspace{-1em}
    \tiny \caption{NO Approximation Rate Landscape}\normalsize
    \label{fig:StructureInformedVsNot}
    \vspace{-0.8em}
\end{wrapfigure}

The contemporary landscape of approximation guarantees for infinite-dimensional machine learning, and especially for neural operators (NOs), can be segmented into two types of guarantees, summarized by the scaling laws in Figure~\ref{fig:FBNO}.  
In this paper, we disregard NO builds whose associated hypothesis classes have finite (local) Rademacher complexity (cf.~\cite{bartlett2002rademacher,bartlett2006local,bartl2025we}) or finite fat-shattering dimension (e.g.\ those of~\cite{alvarez2024neural}), since such models fall within the scope of classical learning-theoretic generalization theory~\cite{AlonBenDavidCesaBianchiHaussler_PDGeneralization}.  Our focus is instead on NO architectures whose learning-theoretic dimensions are infinite: for these, classical complexity-based arguments no longer yield generalization guarantees, yet certain such NOs are nevertheless known to generalize beyond their training data~\cite{benitez2024out}.

The body of machine learning, such as operator learning, over infinite-dimensional input spaces can be segmented into two classes.  
Either these models are capable of approximating \emph{general} families of solution operators under \textit{minimal structural assumptions} and must succumb to exponential scaling (the {\color{black}{red region}} in Figure~\ref{fig:StructureInformedVsNot}), or there is special structure which aligns with the inductive bias of a particular model and allows it to unlock sub-exponential scaling in the reciprocal approximation error; this latter, more challenging regime is the focus of this paper.  
We now provide visibility into what is currently known about these two regions of the operator-learning landscape.

\subsubsection{Well-Understood Regime: General Setting Implies {\color{black}{Exponential}} Scaling}

Though neural operators have proven numerically successful for solving certain parametric families of PDEs~\cite{pang2020npinns,li2020fourier,li2023fourier,raonic2024convolutional,benitez2024out,li2024physics} and admit a growing body of universal approximation guarantees~\cite{kovachki2021universal,lu2019deeponet,lanthaler2022error,lanthaler2023operator,kratsios2024mixture,cuchiero2023global,lanthaler2023nonlocal,ismailov2024universal}, recent information-theoretic results reveal a sharp \emph{generality--complexity} tradeoff.  In particular,~\citep[Theorem~2.8]{lanthaler2024operator} shows that, in the \textit{worst case}, the number of parameters required for a neural operator to uniformly approximate a generic nonlinear Lipschitz operator (such as our $\Gamma^{\star}$) is necessarily exponential, of order $\Omega(e^{c/\varepsilon})$ in the reciprocal accuracy $\varepsilon>0$, for some constant $c>0$, on spaces such as $W^{s,p}(\mathcal{D})\times W^{s + \frac{d+1}{2},2}(\partial \mathcal{D};\mathbb{R})$.  \add{Moreover, the minimax lower bounds in~\cite{lanthaler2025curse} align with available upper bounds for encoder--decoder architectures between general Fr\'{e}chet spaces~\cite{galimberti2022designing}, and with the upper bounds obtained for neural operators acting between Sobolev spaces over non-Euclidean physical domains~\cite{kratsios2024mixture}.}

We also note that~\cite{alvarez2024neural} establishes a universal neural-operator design for maps between spaces of square-integrable processes adapted to the filtration of Brownian motion, and provides explicit approximation rates for NOs with $\operatorname{ReLU}$ activations%
\footnote{They do not restrict to the super-expressive activations of~\cite{zhang2022deep}, which thus need not generalize.}%
; their rates likewise exhibit exponential dependence on $\varepsilon^{-1}$, despite universality.  Thus, if one is solely interested in an \emph{existence} guarantee, namely, that there exists a NO approximating the solution operator for a family of BSDEs under minimal structural hypotheses (e.g.\ Lipschitz generators and general terminal conditions); then~\cite[Theorem~12]{alvarez2024neural} is well-suited.  However, once one seeks \emph{sub-exponential} (and in particular polynomial) scaling laws, additional structure is unavoidable; this is precisely the second regime depicted in Figure~\ref{fig:StructureInformedVsNot}.

\subsubsection{Main Challenge: Identifying {\color{cadmiumgreen}{Polynomial}} Scaling Regimes}

{\color{black}Examples of structured approximation problems studied in} the {\color{black}literature include} elliptic PDEs~\cite{marcati2023exponential,feischl2025neural}, holomorphic dependence~\cite{adcock2024optimal,herrmann2024neural,franco2023approximation}, Hamilton--Jacobi equations~\cite{lanthaler2025curse}{\color{black} , optimization~\cite{kratsios2025generative,marwah2023deep}}, and {\color{black}finite-dimensional parametrizations}~{\color{black} \cite{liu2024deep,berner2020numerically,elbrachter2022dnn}}. {\color{black}Here we specify} the {\color{black}structure} and {\color{black}representation counts for our own construction}.

\subsubsection{The Gap: Are there ``Special Structures'' in Stochastic Analysis which are Aligned with the Inductive Bias of Specific NO Architectures?}
{\color{black}The stochastic question studied here is whether a common bounded predictable factor and a small reference PDE solution family permit polynomial upper representation bounds after applying} the stochastic {\color{black}adapter. Neural operators} for stochastic analysis {\color{black}and related applications} are {\color{black}studied} in{\color{black} ~\cite{galimberti2022designing,alvarez2024neural,firoozi2025simultaneously,cuchiero2026global}}.

{\color{black}A related setting is approximation using} signature features~\cite{lyons2014rough,kiraly2019kernels,chevyrev2022signature}{\color{black},} building on rough path theory~\cite{hambly2010uniqueness,chevyrev2016characteristic}{\color{black}. For continuous} bounded-variation paths of unit {\color{black}variation in a fixed dimension $d\geq2$}, the {\color{black}signature truncation tail in the tensor-level $\ell^1$ norm decays faster than every fixed inverse power} of {\color{black}the number} of {\color{black}retained coordinates}\footnote{See comment in Appendix~\ref{s:sig_decay}.}.

The objective of this paper is to {\color{black}answer this question} by identifying structured families of backward stochastic differential equations (BSDEs) whose associated solution operators can be approximated by NOs while maintaining polynomial scaling in the reciprocal approximation error.  We focus on BSDEs since, following their introduction in a series of papers~\cite{BismutFirstPaper_OG_1973,PardouxPend_BSDE_OG_1990,Antonelli_FBSDE_OF_1993}, they have become a core component in stochastic optimal control, see e.g.~\cite{MaYongBook_1999,touzi2012optimal}, reinforcement learning~\cite{jia2022policy,MR4209484,MR2746549,MR4329785,MR4596108}, generative modelling~\cite{park2022neural}; and thus in economics~\cite{elie2019tale,MR3738664,elie2019contracting,weston2024existence}, green finance~\cite{MR3076679}, risk-management~\cite{armenti2017central,cvitanic2017moral}, option pricing~\cite{fujii2015fbsde}, and insurance~\cite{elliott2011bsde}.

\subsection{Main Contribution: Identifying Families of BSDEs with Special Structure}
More precisely, we study NO approximations to an \textit{infinite families} of BSDEs described as follows; with technical conditions outlined within the main text.
To explain, we fix a filtered probability space $(\Omega,\mathcal{F},\mathbb{F}\eqdef (\mathcal{F}_t)_{t\ge 0},\mathbb{P})$ satisfying the usual conditions and supporting a $d(\ge 2)$-dimensional Brownian motion $W_{\cdot}\eqdef (W_t)_{t\ge 0}$.  
Fix a predictable column-vector process $\beta$ satisfying
Assumption~\ref{ass:StrongNovikov}. For the stochastic results,
write $\gamma(t,x)=\gamma(x)$ and $\mu(t,x)=\mu(x)$, and put
$\Sigma(x)=\sqrt{2\gamma(x)}$. We regard $W_t$ and $\beta_t$
as columns, and $\nabla u(x)$ and $Z_t$ as rows.
Thus $Z_t\beta_t$ and $Z_t\,dW_t$ are scalars.
Let $x\in\mathcal{D}$, where $\mathcal{D}$ is a bounded nonempty
$C^1$ domain. {\color{black} The deterministic construction uses $d\geq3$, $d<p<\infty$, and a smooth boundary, as fixed in Section~\ref{s:Prelim__ss:Notation}.} The coefficients, $\beta$, and the forward process
are fixed throughout the family indexed by the source and boundary
data $(f_0,g)$. Define the unstopped stochastic exponential by
\begin{equation}\label{eq:GirsanovMartingale}
\Upsilon_t=\exp\left(-\int_0^t\beta_s^\top\,dW_s \, -\tfrac12\int_0^t\|\beta_s\|^2\,ds\right)
.
\end{equation}
The family is
\begin{align}
\label{eq:FBSDE_ForwardProcess}
X_t
&=x+\int_0^t
       \bigl[\mu(X_s)+\Sigma(X_s)\beta_s\bigr]\,ds
     +\int_0^t\Sigma(X_s)\,dW_s,
\\
\label{eq:FBSDE}
Y_t
&=\Upsilon_\tau^{-1}g(X_\tau)
 +\int_{t\wedge\tau}^{\tau}
  \Big[
   \Upsilon_s^{-1}
     \bigl(\alpha(X_s,\Upsilon_sY_s)-f_0(X_s)\bigr)
   -2Z_s\beta_s+\|\beta_s\|^2Y_s
  \Big]\,ds
 -\int_{t\wedge\tau}^{\tau}Z_s\,dW_s,
\\
\label{eq:RandomTerminalCondition}
\tau & \eqdef \inf\{
        t\ge 0:\, X_t\not\in \mathcal{D}
    \}
,
\end{align}
Here $\Upsilon$ is the negative stochastic exponential in \eqref{eq:GirsanovMartingale}. Define $\Gamma^\star(f_0,g)$ to be the unique $\delta_{\mathrm{ref}}$-admissible solution\footnote{See Assumption~\ref{eq:Growth_alpha} for a definition of admissibility in this sense.}, whenever it exists. Proposition~\ref{prop:FK_TypeResult} constructs this solution from a small classical solution of the reference PDE.  Theorem~\ref{thrm:Main_Stochastic} transfers a uniform deterministic $C^1$ approximation of that PDE solution family to the corresponding BSDE family, including any separately justified architecture bounds.

The construction links two components: (i) convolution against the singular part of the Green function associated with the reference semilinear elliptic PDE family
\begin{align}
\label{eq:semilinear}
\mathcal{L}
u + \alpha(x,u) & = f_0 \ \mathrm{in} \ \mathcal{D},
\\
\label{eq:semilinear:BC}
u & =g \ \mathrm{on} \ \partial \mathcal{D}
\end{align}
indexed by the same boundary data $g$ and source data $f_0$; and (ii) the non-Markovianity of our specialized family of BSDEs via the Dol\'{e}ans--Dade exponential of the non-Markovian factor $\beta_{\cdot}$ in~\eqref{eq:FBSDE_ForwardProcess}.  
When it exists, the second operator which we approximate over a regular subset of its domain is the nonlinear operator sending $(f_0,g)$ to the solution of the semilinear elliptic PDE defined via~\eqref{eq:semilinear}--\eqref{eq:semilinear:BC}; we denote this operator by $\Gamma^{+}$ and refer to it as the \textit{solution operator} for this family of PDE systems.

The stochastic approximation results transfer directly from the polynomial deterministic architecture bounds whenever those bounds and the required uniform $C^1$ approximation have been independently established. The representation and transfer arguments are proved separately from the deterministic approximation construction.

\subsection{Organization of Paper}
Section~\ref{s:Prelim} contains the preliminaries needed to formulate our main results, including notation and background in PDEs and operator learning. Section~\ref{s:Main_Result} states our main approximability guarantees for solution operators of certain families of FBSDEs (Theorem~\ref{thrm:Main_Stochastic}) and the corresponding families of elliptic PDEs (Theorem~\ref{thrm:Main_EllipticPDESemilinear}).~\add{Examples elucidating our NO build and the identified ``special structure'' are provided in Section~\ref{s:Examples}.
A brief discussion of the scope of these structured examples are in real-world applications and the first-order construction are found in Section~\ref{s:Discussion}.
}
Proofs, technical lemmata, and auxiliary results (some of independent interest) are collected in Section~\ref{s:Proof}.

{\color{black}For} further background on FBSDEs we refer the reader to~\cite{MaYongBook_1999} since a deeper understanding thereof is not required for the formulation of our main results and for our proofs.

\section{Preliminaries}
\label{s:Prelim}
\add{All notation, background, and structural conditions required to formulate our main results can be found here.  The reader interested in the main result itself is encouraged to skip to the next section.}

\subsection{Notation}
\label{s:Prelim__ss:Notation}

We collect all notation used in our manuscript.  
For any $\sigma \in C(\mathbb{R})$, $N\in \mathbb{N}_+$, and any vector $x\in \mathbb{R}^N$, we denote the componentwise composition of $\sigma$ with $x$ by $\sigma\bullet x\eqdef (\sigma(x_n))_{n=1}^N$.
For each $d\in \mathbb{N}_+$ we set of $d\times d$ invertible matrices by $\operatorname{GL}_d\eqdef \{A\in \mathbb{R}^{d\times d}:\, \operatorname{det}(A)\neq 0\}$ and the set of symmetric positive-definite matrices by $P_d^+$.  
For any $p\in (1,\infty)$ we use $p'$ to denote its conjugate; i.e.\ $1/p' + 1/p=1$ and $p'=\infty$ if $p=1$.  The open ball of radius $R>0$ in any Banach space $X$, normed by $\|\cdot\|_X$, is denoted by $B_{X}(0,R)\eqdef \{x\in X:\, \|x\|_X<R\}$.  Given a (non-empty) open domain $\mathcal{D}$ in $\mathbb{R}^d$ and $1\leq p,q \leq \infty$, we will consider the mixed Lebesgue as 
$L^{p}(\mathcal{D}; L^{q}(\mathcal{D};\mathbb{R}))$, consisting of all $F:\mathcal{D}\to L^q(\mathcal{D},\mathbb{R})$ satisfying the integrability condition
\[
            \|F\|_{L^{p}(\mathcal{D}; L^{q}(\mathcal{D};\mathbb{R}))}
        \eqdef
            \Biggl(
                \int_{\mathcal{D}} 
                \biggl(
                    \int_{\mathcal{D}} |F(x,y)|^{q} dy
                \biggr)^{p/q}dx
            \Biggr)^{1/p}
    < 
        \infty
.
\]
{\color{black}For $p=\infty$ or $q=\infty$, the preceding mixed norm is interpreted as
\[
\|F\|_{L^p(\mathcal D;L^q(\mathcal D))}
=\big\|x\mapsto\|F(x,\cdot)\|_{L^q(\mathcal D)}\big\|_{L^p(\mathcal D)},
\qquad
\|h\|_{L^\infty(\mathcal D)}=\operatorname*{ess\,sup}_{x\in\mathcal D}|h(x)|.
\]}
We will use $W^{1,p}(\partial \mathcal{D})$ and $W^{1,p}(\mathcal{D})$ are the usual Sobolev spaces over $\partial \mathcal{D}$ and $\mathcal{D}$, respectively.

Let $\mathcal{D}\subset\mathbb R^d$ be bounded with smooth boundary, $d\geq3$, and $d<p<\infty$. Fix the linear differential operator $\mathcal L$ whose Green operator is specified below. Write $C^q(\overline{\mathcal D})$ for functions whose classical derivatives through order $q$ extend continuously to $\overline{\mathcal D}$. For a function on $\mathcal D$, set
\[
 \|v\|_1=\|v\|_\infty+\sum_{i=1}^d\|\partial_i v\|_\infty,
 \qquad
 \|v\|_{C^1}=\|v\|_\infty+\|\nabla v\|_{L^\infty(\mathcal D;\mathbb R^d)}.
\]
The $W^{1,\infty}$ norm below is $\|\cdot\|_1$. On $C^1(\overline{\mathcal D})$. For a kernel with weak first $x$ derivatives define 
\[
 K_Eh(x)=\int_{\mathcal D}E(x,y)h(y)\,dy,\qquad
 \|E\|_{\mathcal K}
 =\sum_{|\beta|\leq1}\operatorname*{ess\,sup}_{x\in\mathcal D}
       \int_{\mathcal D}|\partial_x^\beta E(x,y)|\,dy.
\]

Throughout this manuscript, we will fix a filtered probability space $(\Omega,\mathcal{F},\mathbb{F}\eqdef (\mathcal{F}_t)_{t\ge 0},\mathbb{P})$ such that: $\mathbb{F}$-is right-continuous, $\mathcal{N}\subset \mathcal{F}_0$ with $\mathcal{N}=\{A\subseteq \Omega:\, \exists B\in\mathcal{F}: A\subset B  \mbox{ and } \mathbb{P}(B)=0\}$, and supporting a $d$-dimensional Brownian motion $W_{\cdot}\eqdef (W_t)_{t\ge 0}$.  

\subsection{Structural Assumptions}
\label{s:Prelim__ss:SettingAssumptions}
\add{We record the main structural conditions which will be key to the polynomial scaling guarantees.}
\subsubsection{Structure for Partial Differential Equations}
\label{s:Prelim__ss:SettingAssumptions___sss:PDE}
Throughout this paper, we assume the following minimal dimension and integrability range.
\begin{assumption}
\label{ass:p-p-prime}
Let $d \geq 2$, and let $1<p<\infty$ with conjugate satisfying $1 < p'= \frac{p}{p-1} < \frac{d}{d-1}$.
\end{assumption}
In what follows we will always consider 
\add{a second order differential operator $\mathcal{L}$ on $d$ symbols}.  whose associated 
(Dirichlet) Green’s function $G_{\mathcal{L}}$ for $
\add{\mathcal{L}}
$ with the Dirichlet boundary condition is given by
\begin{equation}
\label{eq:greensfunction}
\begin{aligned}
\add{\mathcal{L}}
G_{\mathcal{L}} (\cdot, y) 
& = -\delta(\cdot,  y)\,  & \mathrm{in} \ \mathcal{D},
\\
G_{\mathcal{L}} (\cdot, y) & = 0  \, &\mathrm{on} \ \partial \mathcal{D}.
\end{aligned}
\end{equation}


\noindent
For the deterministic construction, write $G=G_{\mathcal L}$,
$\Phi=\Phi_{\mathcal L}$, and $\Psi=\Psi_{\mathcal L}$.
\begin{equation}
    \label{eq:GreensFunction}
        G_{\mathcal{L}}
    = 
        \underbrace{
            \Phi_{\mathcal{L}}
        }_{\text{Singular Part}}
    + 
        \underbrace{
            \Psi_{\mathcal{L}}
    \,
    .
        }_{\text{Regular Part}}
    \end{equation}

\begin{assumption}[Green operators]\label{ass:GreensSymmetry}
The Dirichlet Green kernel for $\mathcal L$, with
$\mathcal L_xG(x,y)=-\delta_y(x)$, has the fixed decomposition
$G=\Phi+\Psi$. The kernel $\Phi$ has continuous first $x$ derivatives
off the diagonal of $\overline{\mathcal D}^{\,2}$, and $\Psi$ has
continuous mixed derivatives there through total order four. Uniformly
off the diagonal,
\begin{equation}\label{eq:unlifted-green-bounds}
\begin{aligned}
 |\partial_x^\beta\Phi(x,y)|
 &\leq C_\Phi |x-y|^{2-d-|\beta|}, &&|\beta|\leq1,\\
 |\partial_x^\beta\partial_y^\nu\Psi(x,y)|
 &\leq C_\Psi |x-y|^{2-d-|\beta|-|\nu|},
 &&|\beta|+|\nu|\leq4.
\end{aligned}
\end{equation}
For every bounded, locally H\"older continuous $h$, the potential
$K_Gh$ is the unique solution in
$C^2(\mathcal D)\cap C^1(\overline{\mathcal D})$ of
$\mathcal L K_Gh=-h$ with zero Dirichlet boundary value.
For each $g\in H^{(d+3)/2}(\partial\mathcal D)$, its harmonic lift
$w_g=S_\gamma g\in C^2(\mathcal D)\cap C^1(\overline{\mathcal D})$
satisfies
\begin{equation}\label{eq:unlifted-lift-bound}
 \mathcal Lw_g=0,\qquad w_g|_{\partial\mathcal D}=g,\qquad
 \|w_g\|_{C^1}\leq C_S\|g\|_{H^{(d+3)/2}(\partial\mathcal D)}.
\end{equation}
\end{assumption}

\subsubsection{Structural Assumptions for BSDEs}
\label{s:Prelim__ss:SettingAssumptions___sss:BSDE}
The following sufficient conditions control the random exit time,
the stopped density, and the time supremum used in our approximation
criterion. They are stronger than a finite-horizon Novikov condition.
\begin{assumption}[Reference diffusion and bounded predictable factor]
\label{ass:StrongNovikov}
The maps $\mu:\mathbb R^d\to\mathbb R^d$ and
$\Sigma=\sqrt{2\gamma}:\mathbb R^d\to\mathbb R^{d\times d}$
are globally Lipschitz. The given forward process $X$ is a continuous
strong solution of \eqref{eq:FBSDE_ForwardProcess}.
The operator $\mathcal{L}$ appearing in \eqref{eq:semilinear} satisfies
\begin{equation}\label{eq:bsde-reference-generator}
\mathcal L v(x)=\nabla v(x)\mu(x)
 +\tfrac12\operatorname{tr}
   (\Sigma(x)\Sigma(x)^\top D^2v(x))
\end{equation}
and there is $\lambda_*>0$ such that
\[
v^\top\Sigma(x)\Sigma(x)^\top v\geq\lambda_*\|v\|^2
\quad(x\in\overline{\mathcal{D}},\ v\in\mathbb R^d).
\]
The predictable factor has a version\footnote{This version is fixed throughout.} with a deterministic bound\footnote{The coefficients, initial point, factor, and forward process are common to all source and boundary data.}
\begin{equation}\label{eq:finienss_doeals_date}
\mathbb P\bigl(\|\beta_t\|\leq B\text{ for every }t\geq0\bigr)=1,
\qquad B<\infty.
\end{equation}
\end{assumption}
For later use set
$
\mu_* \eqdef \sup_{\overline{\mathcal{D}}}\|\mu\|
$, $
\Sigma_* \eqdef \sup_{\overline{\mathcal{D}}}\|\Sigma\|_{\mathrm{op}}$, 
$a_- \eqdef \min_{\overline{\mathcal{D}}}x_1$, 
$a_+ \eqdef \max_{\overline{\mathcal{D}}}x_1$, $\kappa_0 \eqdef 2(\mu_*+1)/\lambda_*$, $
m_0 \eqdef \frac{e^{\kappa_0(a_+-a_-)}-1}{\kappa_0}$.  
Write $X_t^\tau=X_{t\wedge\tau}$ and define
\begin{equation}\label{eq:stopped_density_notation}
\begin{aligned}
M_t\eqdef\int_0^{t\wedge\tau}\beta_s^\top\,dW_s,
\,\,
 A_t\eqdef\int_0^{t\wedge\tau}\|\beta_s\|^2\,ds,
\,\,
D_t\eqdef\Upsilon_{t\wedge\tau}=e^{-M_t-A_t/2},
\,\,
 R_t\eqdef D_t^{-1}.
\end{aligned}
\end{equation}
We use $ \|v\|_{C^1} \eqdef \sup_{\overline{\mathcal{D}}}|v| +\sup_{\overline{\mathcal{D}}}\|\nabla v\| $, $\|Y\|_{\mathcal S_T^2} \eqdef (\mathbb E\sup_{t\leq T}|Y_t|^2)^{1/2}$, $\|Z\|_{\mathcal H_T^2} \eqdef (\mathbb E\int_0^T\|Z_t\|^2\,dt)^{1/2}$.
Thus $\Upsilon$ retains its unstopped definition, while $D$ and $R$ are stopped\footnote{Lemma~\ref{lem:conditions_are_met_BackwardsProces} proves the normalization at $\tau$.}.  

We simply require that the generator term $\alpha$ is polynomial in a neighbourhood compactly supported as a function of $z$%
\footnote{The map $\alpha$ cannot be polynomial and compactly supported unless it is null.}%
, at-least in a neighbourhood of the origin (but not necessarily globally), with coefficients depending on $x$ in a \textit{Sobolev} manner.  We provide a family of examples of such generators $\alpha$ in Section~\ref{s:Examples__ss:Generators}.

\begin{assumption}[Locally polynomial nonlinearity]\label{ass:semilinear-term}
There are $\delta_0>0$, an integer $H\geq3$, and fixed
$a_h\in C^4(\overline{\mathcal D})$, $2\leq h\leq H$, such that
\begin{equation}\label{eq:unlifted-polynomial}
 \alpha(x,z)=\sum_{h=2}^H a_h(x)z^h,\qquad |z|<\delta_0.
\end{equation}
\end{assumption}
Set\footnote{The finiteness of $M$ is proved in Lemma~\ref{lem:k1-potential}.} $M=\|G\|_{\mathcal K}$,
$A_0=\sum_{h=2}^H\|a_h\|_\infty$,
$A_1=\sum_{h=2}^Hh\|a_h\|_\infty$, $m=H-1$, $
 C_0=C_S+M(1+A_0)$, $E_0=1+A_0+(M+1)m$, $C_*=C_0+E_0$.  

\begin{assumption}[Small data]\label{ass:choice-delta-p}
Fix a compact set $K\subset\mathbb R^d$ and
\begin{equation}\label{eq:unlifted-smallness}
 0<\delta\leq
 \min\left\{1,\frac{\delta_0}{2},
       \frac{1}{2(1+C_*)},\frac{1}{2(1+MA_1)}\right\}.
\end{equation}
Use the prescribed family $\mathcal X_{1,\delta,K}$ of source and
boundary data, namely $f_0\in W^{1,\infty}_0(\mathcal D)$ and
$g\in H^{(d+3)/2}(\partial\mathcal D)$ with their respective norms
at most $\delta^2$ and the prescribed support condition
$\operatorname{supp}(f_0)\subseteq K$. The argument is independent of that support
condition. Set
$X_\delta=\{v\in C^1(\overline{\mathcal D}):\|v\|_{C^1}\leq\delta\}$.
\end{assumption}

The condition $\alpha(x,0)=0$ can be enforced by changing the source. For the stochastic representation, absorbing a linear term into $\mathcal L$ is permitted only if the resulting operator still has the diffusion-generator form \eqref{eq:bsde-reference-generator}. 

\begin{assumption}[Small reference solution class]
\label{eq:Growth_alpha}
The function $\alpha$ is continuous on
$\overline{\mathcal{D}}\times\mathbb R$ and is continuously
differentiable in its scalar variable near
$\overline{\mathcal{D}}\times[-\delta_{\mathrm{ref}},\delta_{\mathrm{ref}}]$.
Choose $0<\delta_{\mathrm{ref}}<\delta_0$ such that
\begin{equation}\label{eq:bsde-smallness}
\ell_{\delta_{\mathrm{ref}}}
 \eqdef \sup_{x\in\overline{\mathcal{D}},\ |y|\leq\delta_{\mathrm{ref}}}
 |\partial_y\alpha(x,y)|,
\qquad \ell_{\delta_{\mathrm{ref}}}m_0<1.
\end{equation}
We call a solution of \eqref{eq:FBSDE}
$\delta_{\mathrm{ref}}$-\textbf{admissible} if $Y$ is continuous and adapted,
$Z$ is predictable, $Y_t=Y_\tau$ and $Z_t=0$ for $t>\tau$,
\[
\int_0^{T\wedge\tau}\|Z_s\|^2\,ds<\infty
\quad\mathbb P\text{-a.s. for every }T<\infty,
\]
and
\begin{equation}\label{eq:bsde-admissible}
|D_tY_t|\leq\delta_{\mathrm{ref}}
\quad\text{for every }t\geq0,\quad\mathbb P\text{-a.s.}
\end{equation}
The equation is interpreted locally up to its specified terminal value.
Uniqueness of $Z$ is understood up to $dt\otimes d\mathbb P$-null sets.
Time-supremum statements below use the explicit predictable
representative in \eqref{eq:adapter_corrected}.
\end{assumption}
Under the local polynomial condition in
Assumption~\ref{ass:semilinear-term}, if
$c_h(x)=\partial_y^h\alpha(x,0)/h!$, then
\[
\ell_{\delta_{\mathrm{ref}}}
\leq\sum_{h=2}^H h\|c_h\|_\infty
                 \delta_{\mathrm{ref}}^{h-1}\longrightarrow0
\quad\text{as }\delta_{\mathrm{ref}}\downarrow0.
\]
Consequently \eqref{eq:bsde-smallness} can be satisfied.
The existence of a classical PDE solution with this pointwise bound
is a separate premise in Proposition~\ref{prop:FK_TypeResult} and
Theorem~\ref{thrm:Main_Stochastic}.

\subsection{Background on Standard Deep Learning Models}
\label{s:Prelim__ss:DLModels}
We now overview the required background on standard deep learning models defined with either of two activation functions.  
Specifically, we consider the rectified quadratic unit and the hard sigmoid activation functions, defined for each $t\in\mathbb{R}$ by
$
    \operatorname{ReQU}(t) \eqdef \max\{0,t\}^2
$.

\begin{definition}[Neural Network]
\label{defn:BoundedReLUMLP}
Let $d,D\in \mathbb{N}_+$, fix an activation function $\sigma\in C(\mathbb{R})$, and depth and width parameters $L,W\in \mathbb{N}_+$.  A function $f\in C(\mathbb{R}^d,\mathbb{R}^D)$ is a multilayer perceptron (MLP) with activation function $\sigma$, in the class $\mathcal{NN}_{L,W:d,D}^{\sigma}$, if it admits the iterative representation
\[
\begin{aligned}
    f(x) & = A_{L}x^{{\color{black}(L-1)}} + b_{L}\\
    x^{(l)} & = \sigma\bullet  \big(A_{l}x^{(l-1)} + b_{l}\big) \qquad {\color{black}l}=1,\dots,L-1\\
    x^{(0)} & \eqdef x
\end{aligned}
\]
where $A_l\in \mathbb{R}^{{\color{black}d_l\times d_{l-1}}}$ {\color{black}and} $b_l\in \mathbb{R}^{{\color{black}d_l}}$ {\color{black}for $l=1,\dots,L$}, $d_0=d$, {\color{black}$d_L=D$}, and $d_l\le W$ for $l=1,\dots,{\color{black}L-1}$.  
\end{definition}

\begin{figure}[th!]
    \centering
    \includegraphics[width=.98\textwidth]{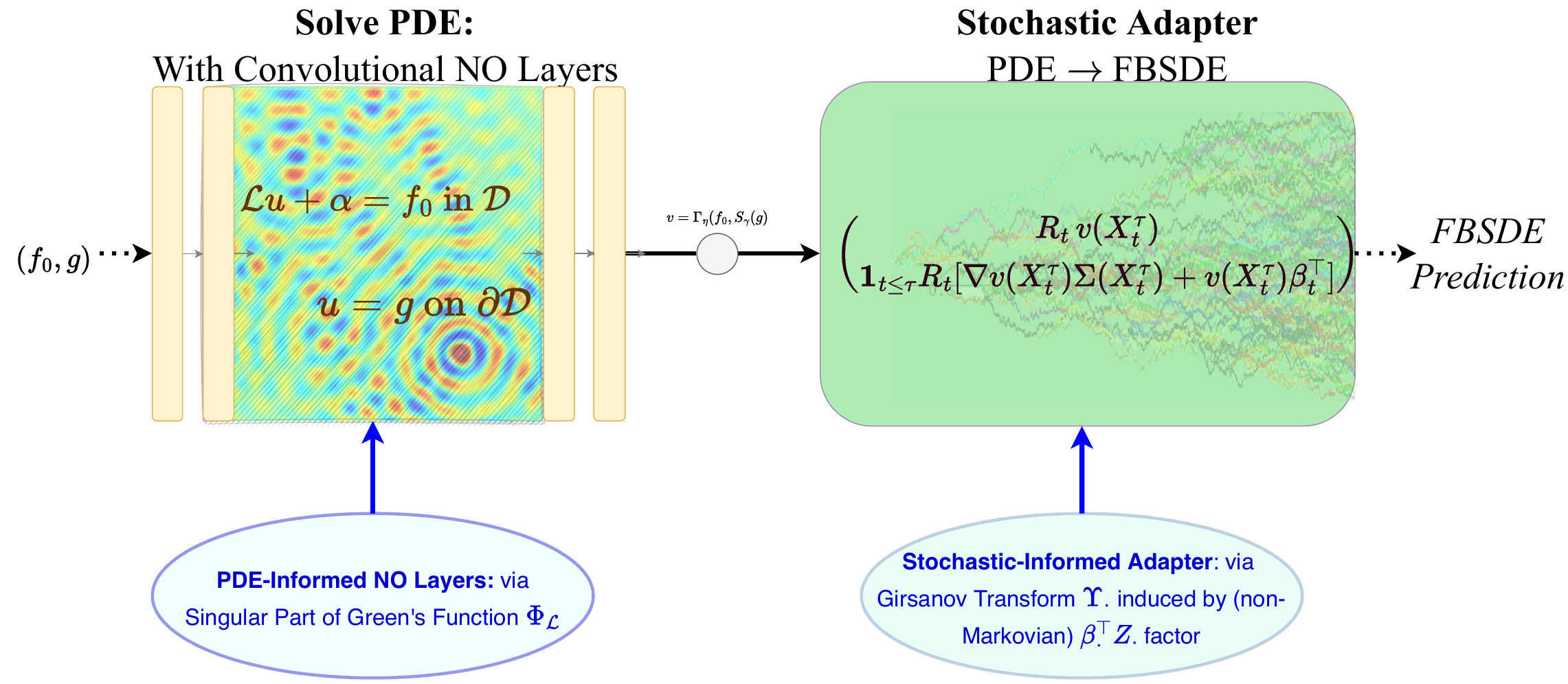} 
\caption{\textit{Neural Operator Architecture}:
\hfill\\
\textbf{The Convolutional NO (PDE-Informed):}
A \textit{convolutional NO} solves the Elliptic Dirichlet problem associated with the FBSDE in~\eqref{eq:FBSDE}, for boundary/terminal data $g$ and BSDE source $f_0$.  Each convolutional layer encodes the relevant Green's function, and their composition implements a rapidly convergent fixed-point iteration for the resulting elliptic solution $u$.
\hfill\\
\textbf{The BSDE Adapter (BSDE-Informed):}
From the NO output $u$ (with inductive bias aligned to the semilinear elliptic PDEs induced by our BSDE family), we evaluate $u$ and $\nabla u$ along forward trajectories $X_{\cdot}$ (cf.~\eqref{eq:FBSDE_ForwardProcess}) and combine them with the reciprocal Dol\'eans--Dade exponential of $\beta_{\cdot}$ to produce a structure-informed prediction of the BSDE pair $(Y_{\cdot},Z_{\cdot})$.}
    \label{fig:FBNO}
\end{figure}

\section{A Neural Operator With a Structure-Informed Inductive Bias}
\label{s:Background__ss:NOBuilds}
We recall that $1<p<\infty$, $s \in \mathbb{N}$, and $\mathcal{D} \subset \mathbb{R}^d$ is a bounded domain.
Our architecture combines the reference elliptic PDE structure with the stochastic exponential of the common predictable factor.  Proposition~\ref{prop:FK_TypeResult} gives the precise representation, and Figure~\ref{fig:FBNO} illustrates the resulting composition.

We build on two standard variants of the standard non-local neural operators (NO), considered e.g.\ in~\cite{kovachki2023neural,lanthaler2023nonlocal}.  Our first variant extends the classical NO construction by incorporating a convolution operator, once per layer.  
This is motivated by the fact that the Green's function for each elliptic PDE associated with our FBSDEs can be decomposed into a ``singular part'' and a ``regular part''.
The convolutional component of each layer is designed to encode the singular part of the Green's function while the remaining traditional affine parts of our NO layers can encode the regular part of the Green's function.

\subsection{The PDE Build: Convolutional Neural Operator with an Customized Singular Kernel Convolution}
\label{s:Background__ss:NOBuilds___sss:PDE}

Convolution operators have been employed in \cite{raonic2024convolutional}, where the motivations for their use might differ from our insight into the singularity in Green's function.  
We also incorporate a skip connection in our neural operators, since it is well-known that they improve the loss landscape~\cite{riedi2023singular} and expressivity~\cite{Bronstein2024} of classical neural networks by ensuring that information propagates forward through their hidden layers.

\begin{definition}[Neural Operator]
\label{def:neural-operator}
\begingroup
\endgroup
{Fix a \textit{lifting dimension} $k\in \mathbb{N}_+$}.
Let $\varphi = \{\varphi_n^{\uparrow}\}_{n \in \nabla}$ be sufficiently smooth and compactly supported wavelet frames of $L^p(\mathcal{D})$ and $L^p(\mathcal{D}^k)$, respectively, {where $\nabla = \{ (j, r) : j \in \mathbb{N}, \ r = 1,..., N_j \}$}%
.
Let $\mathcal{L}$ satisfy Assumption~\ref{ass:GreensSymmetry} and let $\Phi_{\mathcal{L}}$ be the singular part in~\eqref{eq:GreensFunction}.
A \textit{convolutional} neural operator is a map $
\Gamma: W^{s,p}(\mathcal{D};\mathbb{R})^2 \to W^{s,p}(\mathcal{D};\mathbb{R})$ 
\[
\Gamma : v_{0} \mapsto v_{L+1}, 
\]
admitting the following iterative
where $v_{L+1}$ is give by the following steps: 
\[
v_{\ell+1}(x) = \sigma \bullet \left( W^{(\ell)} v_{\ell}(x) + 
(K^{(\ell)}_{N}v_{\ell})(x) +
(K^{(\ell)}v_{\ell})(x) + b^{(\ell)}_N(x)  \right)
+ 
W^{(\ell)}_s v_{\ell}(x),
\]
for $\ell =0,\dots,L-1$ and $x \in \mathcal{D}$; where $\sigma : \mathbb{R} \to \mathbb{R}$ is a nonlinear activation operating componentwise, and $W^{(\ell)}, W^{(\ell)}_s \in \mathbb{R}^{d_{\ell+1}\times d_{\ell}}$ { ($d_0=2, d_{L+1}=1$)}
\[
v_{L+1}(x) = W^{(L)}v_{L} +
(K^{(L)}_{N}v_{L})(x) + 
(K^{(L)}v_{L})(x) + b^{(L)}_N(x), 
\]
and the convolutional layers%
    ~%
are 
\begin{equation}
\label{eq:NOLayers__K_fin_rank}
(K^{(\ell)}v)(x) \eqdef 
\underbrace{
    C^{(\ell)}
    \int_{\mathcal{D}} 
        \Phi_{\mathcal{L}}(x,y) v(y) dy
}_{\text{Singular Part}}
\mbox{ and }
(K^{(\ell)}_{N}v)(x) \eqdef 
\underbrace{
\sum_{n,m \in \Lambda_{N}}
    C_{n,m}^{(\ell)}
    \langle \varphi_{m}{{^{\uparrow}} \circ A, 
    }v  \rangle \varphi_{n}{^{\uparrow}}(A(x))
}_{\text{Regular Part}},
\end{equation}
where $ \Lambda_{N} \subset \nabla$ with $|\Lambda_{N}| \leq N$
each \textit{lifting channel} $A:\mathbb{R}^d\to\mathbb{R}^{dk}$ 
is injective linear map, and the bias layers are
\[
b_{N}^{(\ell)}(x) = \sum_{n \in \Lambda_{N}} b_n^{(\ell)} \varphi^{\uparrow}_n( A(x)),
\]
for $x \in \mathcal{D}$ and $v \in L^{p}(\mathcal{D};\mathbb{R})^{d_{\ell}}$; 
where $C^{(\ell)}_{m,n}, C^{(\ell)} \in \mathbb{R}^{d_{\ell+1}\times d_{\ell}}$ and $b^{(\ell)}_n \in \mathbb{R}^{d_{\ell+1}}$;
where
\[
\begin{aligned}
\langle \varphi^{\uparrow}_{m}\circ A, v \rangle
&\eqdef
\left( \langle \varphi^{\uparrow}_{m}\circ A, v_1 \rangle,..., \langle \varphi^{\uparrow}_{m}\circ A, v_{d_\ell} \rangle \right) \in \mathbb{R}^{d_{\ell}}, 
\\
\int_{\mathcal{D}} 
        \Phi_{\mathcal{L}}(x,y) v(y) dy
& \eqdef
\left( \int_{\mathcal{D}} \Phi_{\mathcal{L}}(x,y) v_1(y) dy,..., \int_{\mathcal{D}} \Phi_{\mathcal{L}}(x,y) v_{d_\ell}(y) dy \right) \in \mathbb{R}^{d_{\ell}}, 
\end{aligned}
\]
for $v \in L^{p}(\mathcal{D};\mathbb{R})^{d_{\ell}}$;  where $\langle f, g \rangle=\int f g dx$.
\medskip
\hfill\\
We denote by $\mathcal{NO}^{L, W, \sigma}_{N, \varphi}=\mathcal{NO}^{L, W, \sigma}_{N, \varphi}(W^{s,p}(\mathcal{D};\mathbb{R})
^2
, W^{s,p}(\mathcal{D};\mathbb{R})
)$ the class of neural operators defined above, with depth $L$, width $W$, {\color{black}spatial index bound} $N$, activation function $\sigma$, and wavelet frame $\varphi$. 
\end{definition}

The operators $K^{(l)}_N$ play a crucial role in capturing the non-local nature of PDEs. They are defined by truncating the basis expansion, a definition for neural operators  inspired by  \citet{lanthaler2023nonlocal}.  In this context, {\color{black}writing $e_i$ for the coordinate vectors in $\mathbb R^{d_{l+1}}$, the finite expansion gives
\[
\operatorname{ran}K^{(l)}_N
\subseteq\operatorname{span}\{(\varphi_n^{\uparrow}\circ A)e_i:
n\in\Lambda_N,\ 1\leq i\leq d_{l+1}\}.
\]
Indeed, each output is a linear combination of these functions,} with {\color{black}scalar coefficients given by the coordinates of $C_{n,m}^{(l)}\langle\varphi_m^{\uparrow}\circ A,v\rangle$}. Therefore $
\operatorname{rank}K^{(l)}_N\leq|\Lambda_N|d_{l+1}\leq Nd_{l+1}$.

\subsection{{The BSDE Build: Forward-Backwards Neural Operator}}
\label{s:Background__ss:NOBuilds___sss:FBSDE}
Under Assumption~\ref{ass:StrongNovikov} the ``stochastic adapter'' in Figure~\ref{fig:FBNO} is meaningfully well-defined as follows.  

\begin{definition}[Stochastic Adapter]
\label{defn:StochasticAdapter}
Assume Assumption~\ref{ass:StrongNovikov} and fix $T<\infty$.
For $v\in C^1(\overline{\mathcal{D}})$ define
\begin{equation}\label{eq:adapter_corrected}
\begin{aligned}
\mathcal A_{X_{\cdot},\beta_{\cdot}}(v)_t
&=(Y_t^v,Z_t^v),\\
Y_t^v&=R_tv(X_t^\tau),\\
Z_t^v&=\mathbf1_{\{t\leq\tau\}}R_t
 \bigl[\nabla v(X_t^\tau)\Sigma(X_t^\tau)
       +v(X_t^\tau)\beta_t^\top\bigr].
\end{aligned}
\end{equation}
Here $\Sigma(x)=\sqrt{2\gamma(x)}$, and $D,R,X^\tau$ are defined in
\eqref{eq:stopped_density_notation}. 
\end{definition}
The \textit{stochastic adapter} maps $C^1(\overline{\mathcal{D}})$ to
$\mathcal S_T^2\times\mathcal H_T^2$\footnote{See Lemma~\ref{lem:perturbation_bound}}.
We provide several examples of stochastic adapters, tailored to a range of BSDEs and motivated by a diverse set of applications in Subsection~\ref{s:Examples__ss:stockadapters}.

Our neural operator with structure-informed inductive bias is defined as the composition of our convolutional neural operator (Definition~\ref{def:neural-operator}) and the stochastic adapter above.
\begin{definition}[Forward-Backwards Neural Operator (FBNO)]
\label{defn:FBNO}
Fix $d,L,W,N\in\mathbb N_+$, $T>0$, and a frame $\varphi$
as in Definition~\ref{def:neural-operator}. Let the adapter be
as in Definition~\ref{defn:StochasticAdapter}.
For any $\mathcal{I}\subseteq W^{s,p}(\mathcal{D};\mathbb{R})\times W^{s+\tfrac{d+1}{2},2}(\partial \mathcal{D};\mathbb{R})
$ on which $S_\gamma(g)$ is defined,
a map $\hat\Gamma:\mathcal I\to\mathcal S_T^2\times\mathcal H_T^2$
belongs to $\mathcal{FBNO}_{N,\varphi}^{L,W}$ if it admits the prepresentation
\[
\hat\Gamma(f_0,g)
=\mathcal A_{X_{\cdot},\beta_{\cdot}}
       \bigl(\Gamma(f_0,S_\gamma(g))\bigr),
\]
where $\Gamma\in\mathcal{NO}_{\tilde N,\varphi}^{\tilde L,\tilde W,\operatorname{ReQU}}$, $\tilde L\leq L$, $\tilde W\leq W$, $\tilde N\leq N$, and $\Gamma(f_0,S_\gamma(g))\in C^1(\overline{\mathcal{D}})$ for every $(f_0,g)\in\mathcal I$.
\end{definition}
\noindent 
We are now ready to formulate our the main contributions of our manuscript.

\section{Main Results}
\label{s:Main_Result}

For each $s,\delta>0$ and compact $K\subseteq\mathbb R^d$, consider the class of terminal conditions and perturbations to the drivers to our FBSDEs belonging to the following class
\begin{equation}
\label{eq:pertubations_terminal_conditions}
        \mathcal{X}_{s,\delta,K}
    \eqdef
        \big\{
            (f_0,g)\in 
            W^{s,\infty}_0(\mathcal{D};\mathbb{R})\times W^{s + \frac{d+1}{2},2}(\partial \mathcal{D}; \mathbb{R})
            :\,
                \|f_0\|,\|g\|\le \delta^2
                ,
                \operatorname{supp}(f_0)
                \subseteq K
        \big\}
.
\end{equation}
We are now in place to present our main results.

\subsection{Simultaneously Solving Infinite Families of FBSDE}
\label{s:Main__ss:FBSDEs}
We begin with our stochastic result.
For the following theorem, fix $s,\delta>0$, a compact
$K\subseteq\mathbb R^d$, and $\mathcal I\subseteq\mathcal X_{s,\delta,K}$.
Assume that each $(f_0,g)\in\mathcal I$ has continuous
representatives on $\overline{\mathcal{D}}$ and
$\partial\mathcal{D}$, respectively, that $S_\gamma(g)$ is
well defined, that $(f_0,S_\gamma(g))\in W^{s,p}(\mathcal D)^2$,
and that the reference PDE
\eqref{eq:semilinear}--\eqref{eq:semilinear:BC} has a solution
\[
u_{f_0,g}=\Gamma^+(f_0,g)
\in C^2(\mathcal{D})\cap C^1(\overline{\mathcal{D}}),
\qquad
\|u_{f_0,g}\|_\infty\leq\delta_{\mathrm{ref}}.
\]
Fix a frame $\varphi$ as in Definition~\ref{def:neural-operator},
a lifting dimension $k\in\mathbb N_+$, and constants
$r,C_L,C_W,C_N>0$. Suppose that, for every $0<\eta<1$, there is
$\Gamma_\eta\in
\mathcal{NO}^{L_\eta,W_\eta,\operatorname{ReQU}}_{N_\eta,\varphi}$,
using this same $k$, with positive integer architecture sizes,
whose outputs on $\mathcal I$ belong to $C^1(\overline{\mathcal D})$,
and satisfying
\begin{equation}\label{eq:bsde-deterministic-premise}
\sup_{(f_0,g)\in\mathcal I}
\bigl\|u_{f_0,g}-\Gamma_\eta(f_0,S_\gamma(g))\bigr\|_{C^1(\overline{\mathcal D})}
\leq\eta
\end{equation}
and
\begin{equation}\label{eq:bsde-deterministic-complexity}
L_\eta\leq C_L(1+\log(1/\eta)),\qquad
W_\eta\leq C_W,\qquad N_\eta\leq C_N\eta^{-1/r}.
\end{equation}
The frame, $k$, $r$, and $C_L,C_W,C_N$ are independent of $\eta$
and of the individual datum $(f_0,g)\in\mathcal I$.
For $T\geq0$, set
\begin{equation}\label{eq:bsde-transfer-constant}
C_T\eqdef2e^{3B^2T/2}(1+\Sigma_*+B).
\end{equation}

\begin{theorem}[Polynomial Scaling for Structured Families of BSDEs]
\label{thrm:Main_Stochastic}
\hfill\\
Suppose that the preceding PDE hypotheses and
Assumptions~\ref{ass:StrongNovikov}
and~\ref{eq:Growth_alpha} hold.
For any approximation error $0<\varepsilon<1$ and each time
horizon $T>0$, there are $L,W,N\in\mathbb N_+$ and a FBNO
$\widehat\Gamma_\varepsilon
\in\mathcal{FBNO}^{L,W}_{N,\varphi}$ satisfying
\begin{equation}\label{eq:bsde-uniform-approximation}
\sup_{(f_0,g)\in\mathcal I}
\mathbb E\biggl[
    \sup_{0\leq t\leq T\wedge\tau}
    \Bigl(
        |\widehat Y_t^\varepsilon-Y_t|
        +\|\widehat Z_t^\varepsilon-Z_t\|
    \Bigr)
\biggr]
\leq\varepsilon,
\end{equation}
where
$(Y,Z)=\Gamma^\star(f_0,g)$ and
$(\widehat Y^\varepsilon,\widehat Z^\varepsilon)
=\widehat\Gamma_\varepsilon(f_0,g)$ use the representatives
in~\eqref{eq:adapter_corrected}.
\hfill\\
Moreover, with $\eta=\varepsilon/C_T$, one may take
$(L,W,N)=(L_\eta,W_\eta,N_\eta)$, so that
\begin{equation}\label{eq:bsde-transferred-complexity}
L\leq C_L(1+\log(C_T/\varepsilon)),\qquad
W\leq C_W,\qquad N\leq C_NC_T^{1/r}\varepsilon^{-1/r}.
\end{equation}
The lifting dimension is the same fixed $k$. Thus, for fixed $T$,
as $\varepsilon\downarrow0$,
\[
L\in\mathcal O(\log(1/\varepsilon)),\qquad
W\in\mathcal O(1),\qquad N\in\mathcal O(\varepsilon^{-1/r}),
\qquad k\in\mathcal O(1).
\]
The hidden constants may depend on the fixed problem parameters,
$r,T,B,\Sigma_*,C_L,C_W,C_N,k$, but not on $\varepsilon$
or on $(f_0,g)\in\mathcal I$.
\end{theorem}

\subsection{Simultaneously Solving Infinite Families of Semilinear Elliptic PDEs}
\label{s:Main__ss:SemilinearPDEs}

We now formulate the PDE version of our above, main, and simultaneous BSDE result.  We recall that the family of semilinear Elliptic boundary value problems associated to our family of FBSDEs are given by~\eqref{eq:semilinear}-\eqref{eq:semilinear:BC} with source and boundary conditions indexed by~\eqref{eq:pertubations_terminal_conditions}.  We recall that the deterministic variant of our NO build is the neural operator in Definition~\ref{def:neural-operator}, with $k=1$ and $A=I_d$ in the construction below.
Recall that $\Gamma^{+}$ denotes the solution operator to the family of semilinear elliptic PDEs in~\eqref{eq:semilinear}--\eqref{eq:semilinear:BC}
indexed by $(f_0,g)\in \mathcal{X}_{1,\delta,K}$\footnote{Cf.\ Corollary~\ref{cor:wellposedness} for existence on $\mathcal{X}_{1,\delta,K}$.}.
We remark the following existence guarantee:
\begin{proposition}[Existence of Classical Solutions]
\label{prop:existence}
Under Assumptions~\ref{ass:GreensSymmetry}, \ref{ass:semilinear-term}, \ref{ass:choice-delta-p}, and fix a non-empty compact $K\subset \mathcal{D}$.  
For every $f_0\in W^{1,\infty}_0(\mathcal{D})$ and $g\in H^{(d+3)/2}(\partial\mathcal{D})$ satisfying $\|f_0\|_{W^{1,\infty}}\leq\delta^2$ and $\|g\|_{H^{(d+3)/2}}\leq\delta^2$,
the reference PDE has a unique
solution $u=\Gamma^+(f_0,g)$ in the closed $C^1$ ball of radius
$\delta$, and $u\in C^2(\mathcal{D})\cap C^1(\overline{\mathcal{D}})$.
\end{proposition}
This is the well-posedness conclusion of Corollary~\ref{cor:wellposedness};
its proof is independent of the support condition.
\begin{theorem}[Polynomial Scaling for Structured Families of Semilinear Elliptic PDEs]
\label{thrm:Main_EllipticPDESemilinear}
Under Assumptions~\ref{ass:GreensSymmetry}, \ref{ass:semilinear-term}, and~\ref{ass:choice-delta-p},
the problem~\eqref{eq:semilinear}-\eqref{eq:semilinear:BC}
has a unique classical solution $u=\Gamma^+(f_0,g)$ in $X_\delta$
for each $(f_0,g)\in\mathcal X_{1,\delta,K}$
and every $0<\varepsilon<1$, there are positive integers
$L,W,N\in \mathbb{N}_+$ and a 
$\Gamma_\varepsilon
\in\mathcal{NO}^{L,W,\operatorname{ReQU}}_{N,\varphi}$ mapping $W^{1,p}(\mathcal D)^2\to W^{1,p}(\mathcal D)$
with $C^1$ outputs over $\mathcal{X}_{1,\delta,K}$ such that
\begin{equation}\label{eq:semilinear_estimate}
 \sup_{(f_0,g)\in\mathcal X_{1,\delta,K}}
 \|\Gamma^+(f_0,g)-\Gamma_\varepsilon(f_0,S_\gamma g)\|_1
 \leq\varepsilon.
\end{equation}
Writing $J_\varepsilon=\lceil\log_2(2/\varepsilon)\rceil$, one may take
\begin{equation}\label{eq:semilinear_estimate-rate}
 L=2+(5H-2)J_\varepsilon,\qquad
 W\leq4(H+5),\qquad N\leq C\varepsilon^{-d},\qquad k=1.
\end{equation}
The constant $C$ depends only on the fixed domain, Green bounds,
nonlinearity and wavelet generators, and is independent of
$\varepsilon$ and $(f_0,g)$.
\end{theorem}


We obtain the following first-order consequence without domain lifting.

\begin{corollary}[FBNO approximation without domain lifting $(k=1)$]
\label{cor:Cs_Approx}
Under the hypotheses of Theorem~\ref{thrm:Main_EllipticPDESemilinear}, assume additionally $\delta\leq\delta_{\mathrm{ref}}$ and Assumptions~\ref{ass:StrongNovikov} and~\ref{eq:Growth_alpha}. Then, for each $T>0$ and $0<\varepsilon<1$, the conclusion \eqref{eq:bsde-uniform-approximation} holds with $\mathcal I=\mathcal X_{1,\delta,K}$, and additionally
\[
L \in \mathcal{O}(\log(1/\varepsilon))
,\,\,
W\in \mathcal{O}(1)
,\,\,
N\in \mathcal{O}(\varepsilon^{-d})
,\,\, 
k=1.
\]
These bounds hold as $\varepsilon\downarrow0$ for fixed $T$;
the hidden constants are independent of $\varepsilon$ and of $(f_0,g)$.
\end{corollary}

\section{Examples: When Polynomial Scaling Holds}
\label{s:Examples}

This section provides {\color{black}examples} of {\color{black}individual} components of our setup. 

\subsection{Examples of PDE Facets}
\label{s:Examples__ss:PDE}

We now provide examples illustrating several facets of our theory and of the NO architecture, focusing on the PDE component of our problem. We begin by highlighting problem structures that satisfy our regularity assumptions.
\subsection{Examples of Generators}
\label{s:Examples__ss:Generators}
We now provide a few simple examples of generators satisfying Assumption~\ref{ass:semilinear-term}.
\begin{example}[Generators Polynomial {\color{black}Near Zero with} Vanishing {\color{black}Value and First Derivative}]
\label{ex:alpha_poly_local_compact_z}
Suppose that $\mathcal{D}=\{x\in \mathbb{R}^d:\, \|x\|_2< 1\}$.
Let $\varphi\in C_c^{\infty}(\mathcal{D})$ be a smooth cutoff in $x$.
Consider a smooth function $\eta:\mathbb{R}\to\mathbb{R}$ supported on $[-2,2]$ and with value identically $1$ on $[\tfrac{-3}{2},\tfrac{3}{2}]$; for instance take $\eta$ to be
\[
    \eta(z)
\eqdef
        I_{|z|\le \tfrac{3}{2}}
    +
        \frac{
            \exp\!\big(-\tfrac{1}{4-|z|^2}\big)
        }{\exp\!\big(-\tfrac{1}{4-|z|^2}\big)+\exp\!\big(-\tfrac{1}{|z|^2-\tfrac94}\big)}
    \,
        I_{\tfrac{3}{2}<|z|<2}
.
\]

For any $f\in C^{\infty}(\mathbb{R})$ which is constant on $(-\tfrac{3}{2},\tfrac{3}{2})$, and any polynomial $p:\mathbb{R}\to \mathbb{R}$ satisfying $p(0)=p'(0)=0$, the following map satisfies Assumption~\ref{ass:semilinear-term}
\begin{equation}
\label{eq:crude}
    \alpha(x,z) = \varphi(x)\,f(z)\,p(z)\eta(z)
\end{equation}
If, rather, $3\le H\in\mathbb{N}_+$ and $f_2,\dots,f_H\in C_c^{\infty}(\mathcal{D})$ then the following map satisfies Assumption~\ref{ass:semilinear-term}
\begin{equation}
\label{eq:refined}
    \alpha(x,z) = \eta(z)\big( f_2(x)z^2 + \dots + f_H(x)z^H\big)
.
\end{equation}
\end{example}

\subsection{Green Kernels in the Neural Operator}
\label{s:Examples__ss:GreensFunctionDirichletProblem}
The normalization $\hat Y=DY$ and the change of measure both depend on $\beta$. Together they give the reference equation \eqref{eq:reference_bsde_Q}, whose PDE coefficients and boundary data do not depend on $\beta$. For a deterministic elliptic example to be used in the stochastic theorem, its operator must also satisfy \eqref{eq:bsde-reference-generator}, including its diffusion sign and absence of a zeroth-order term.

Under Assumption~\ref{ass:GreensSymmetry}, Lemma~\ref{lem:expansion-Green} approximates $\Psi_{\mathcal L}$ in $\|\cdot\|_{\mathcal K}$ by the prescribed finite wavelet expansion.

\subsection{Examples of Stochastic Analytic Adapters}
\label{s:Examples__ss:BSDE}
We give examples of bounded predictable factors satisfying the
factor condition in Assumption~\ref{ass:StrongNovikov}, together
with explicit formulas for their stochastic exponentials and adapters.
\subsubsection{Examples of Non-Markovian Factors \texorpdfstring{$\beta_{\cdot}$ in Equation~\ref{eq:FBSDE}}{}}
\label{s:Examples__ss:BSDE___sss:NonMarkFact}
The following examples satisfy the bounded predictable-factor
condition of Assumption~\ref{ass:StrongNovikov}:
\begin{example}[General Bounded Predictable Factor]
\label{ex:StrongNovikov__bounded}
Let $\beta$ be predictable and satisfy
\[
\mathbb P(\|\beta_t\|\leq B\text{ for every }t\geq0)=1.
\]
Then the factor condition of Assumption~\ref{ass:StrongNovikov}
holds. For every $T<\infty$ and $q\geq1$,
\[
\mathbb E\exp\left(\tfrac12(q^2+q)
              \int_0^T\|\beta_t\|^2\,dt\right)
\leq e^{(q^2+q)B^2T/2}<\infty.
\]
The reference coefficients must also satisfy that assumption.
\end{example}

A neural-network instance of
Example~\ref{ex:StrongNovikov__bounded} uses fixed predictable
features; state feedback is a special case.
\begin{example}[MLP-Parameterized Predictable Factor]
\label{ex:StrongNovikov__HSigMLP}
Let $\xi$ be a predictable $\mathbb R^m$-valued process, fixed
throughout the data family. Let
$f_\theta:\mathbb R^{1+m}\to\mathbb R^d$ be a fixed
$\operatorname{ReQU}$-MLP, and apply the sigmoid
$\operatorname{Sig}(a)=(1+e^{-a})^{-1}$ componentwise to set
\[
\beta_t=\operatorname{Sig}\bullet f_\theta(t,\xi_t).
\]
This process is predictable and $\|\beta_t\|\leq\sqrt d$
for every $t$, so its factor bound holds with $B=\sqrt d$.
For state feedback one may take a continuous adapted state process
as $\xi$; general predictable features may also encode path history.
\end{example}

\subsubsection{Examples and Motivation for Stochastic Adapters}
\label{s:Examples__ss:stockadapters}
Before providing a range of examples of stochastic adapters, we motivate its construction by a slightly non-standard Feynman-Kac-type representation of the BSDE considered herein.

The stochastic adapter combines a change of measure with the
normalization $\hat Y=DY$. The physical driver in
\eqref{eq:FBSDE} is chosen so that the normalized equation under
$\mathbb Q$ has driver $\alpha(X,\hat Y)-f_0(X)$ and terminal
value $g(X_\tau)$. The next proposition proves this representation
and uniqueness in the specified small normalized solution class.
\begin{proposition}[A Non-Markovian Feynman-Kac-Type Representation]
\label{prop:FK_TypeResult}
Assume Assumptions~\ref{ass:StrongNovikov} and~\ref{eq:Growth_alpha}.
Under $\mathbb P$, let the common forward process solve
\begin{equation}\label{eq:SDE_X__originalgeneral}
X_t=x+\int_0^t[\mu(X_s)+\Sigma(X_s)\beta_s]\,ds
       +\int_0^t\Sigma(X_s)\,dW_s,
\qquad \Sigma=\sqrt{2\gamma}
\end{equation}
and let $\Upsilon$ be as in \eqref{eq:GirsanovMartingale}.
Then $\tau<\infty$ almost surely,
$(D_t)=(\Upsilon_{t\wedge\tau})$ is uniformly integrable, and
$
\frac{d\mathbb Q}{d\mathbb P}=\Upsilon_\tau
\quad\text{on }{\color{black}\mathcal F}$
defines a probability measure equivalent to $\mathbb P$.
The process
$W_t^{\mathbb Q}=W_t+\int_0^{t\wedge\tau}\beta_s\,ds$
is Brownian under $\mathbb Q$.
Let $f_0\in C(\overline{\mathcal{D}})$ and
$g\in C(\partial\mathcal{D})$, and suppose
$u\in C^2(\mathcal{D})\cap C^1(\overline{\mathcal{D}})$ satisfies
\begin{equation}\label{eq:bsde-reference-pde}
\mathcal Lu+\alpha(x,u)=f_0\text{ in }\mathcal{D},
\quad u=g\text{ on }\partial\mathcal{D},
\quad \|u\|_\infty\leq\delta_{\mathrm{ref}}.
\end{equation}
The system
\begin{align}
\label{eq:BSDE_perturbed__X}
X_t&=x+\int_0^t[\mu(X_s)+\Sigma(X_s)\beta_s]\,ds
         +\int_0^t\Sigma(X_s)\,dW_s,\\
\label{eq:BSDE_perturbed__Y}
Y_t&=\Upsilon_\tau^{-1}g(X_\tau)
 +\int_{t\wedge\tau}^{\tau}
 \Big[\Upsilon_s^{-1}
       \bigl(\alpha(X_s,\Upsilon_sY_s)-f_0(X_s)\bigr)
       -2Z_s\beta_s+\|\beta_s\|^2Y_s\Big]\,ds
 -\int_{t\wedge\tau}^{\tau}Z_s\,dW_s
\end{align}
has a unique $\delta_{\mathrm{ref}}$-admissible\footnote{See Assumption~\ref{eq:Growth_alpha} for a definition of admissibility in this sense.} solution.
Its indicated predictable representative is, for $0\leq t\leq\tau$,
\begin{align}
\label{eq:Y_COM}
Y_t&=\Upsilon_t^{-1}u(X_t),\\
\label{eq:Z_COM}
Z_t&=\Upsilon_t^{-1}
 \bigl[(\nabla u)(X_t)\Sigma(X_t)
       +u(X_t)\beta_t^\top\bigr],\\
\label{eq:Gamma_COM}
\Upsilon_t&=\exp\left(-\int_0^t\beta_s^\top\,dW_s
                         -\tfrac12\int_0^t\|\beta_s\|^2\,ds\right).
\end{align}
For $t>\tau$ set $Y_t=Y_\tau$ and $Z_t=0$.
\end{proposition}
When $\beta\equiv0$, one has $\Upsilon\equiv1$ and the
representation reduces to $Y_t=u(X_t)$,
$Z_t=\nabla u(X_t)\Sigma(X_t)$ before exit, for the
reference PDE \eqref{eq:semilinear}--\eqref{eq:semilinear:BC}.
{\color{black}Stochastic-integral-free} expressions for selected factors {\color{black} are given in Appendix~\ref{s:ClosedForm__StochasticAdapted}}.
Fix $\lambda\geq0$, $C\neq0$, and $h\in C^2(\mathbb R^d)$ with
\begin{equation}\label{eq:gen_armonic}
\Delta h=2\lambda h.
\end{equation}
Define the continuous raw factor
$b_t=Ce^{-\lambda t}(\nabla h(W_t))^\top$.
Choose $B>\|b_0\|$, and let
\begin{equation}\label{eq:control__beta}
\eta_B \eqdef \inf\{t\geq0:\|b_t\|\geq B\},\qquad
\beta_t \eqdef \mathbf1_{\{t\leq\eta_B\}}b_t.
\end{equation}
Continuity makes $\eta_B$ a stopping time. The indicator is
left-continuous and adapted, so $\beta$ is predictable, with
$\|\beta_t\|\leq B$ on every path. The forward process uses this
bounded factor.

\section{Discussion}
\label{s:Discussion}
Before turning to the technical proofs, we briefly indicate (i) the first-order construction and (ii) how, and in which settings, the BSDE structure we study arises in applications. 
\subsection{The First-Order Construction}
Theorem~\ref{thrm:Main_EllipticPDESemilinear} uses $s=1$, $k=1$, and
$A=I_d$, with the first-order error \eqref{eq:semilinear_estimate}
and $N\leq C\varepsilon^{-d}$ in \eqref{eq:semilinear_estimate-rate}.
The inequality $\|w\|_{C^1}\leq\|w\|_{W^{1,\infty}}$ supplies the
norm required by the stochastic transfer, giving the rates in
Corollary~\ref{cor:Cs_Approx}.
The deterministic construction seeks to implement the fixed-point
iteration for the reference elliptic equation. Proposition~\ref{prop:FK_TypeResult}
then reconstructs the physical BSDE solution from a classical small
reference solution, and Lemma~\ref{lem:perturbation_bound} transfers
a $C^1$ approximation through the stochastic adapter.

\subsection{The Structure of BSDEs in Applications}
The connection to an application requires checking that its physical
driver and terminal value have the forms in \eqref{eq:FBSDE} and
that the reference-diffusion and small-solution assumptions hold.
Proposition~\ref{prop:FK_TypeResult} represents this structured family
through a deterministic elliptic PDE followed by the stochastic adapter.
An application with a different driver is not covered solely because
its driver also contains a term linear in $Z$.

For a constant factor $\beta_t=\bar\beta\in\mathbb R^d$,
\[
\Upsilon_t=\exp(-\bar\beta^\top W_t
                      -\tfrac12\|\bar\beta\|^2t),\qquad
\frac{g(X_\tau)}{\Upsilon_\tau}
=g(X_\tau)\exp(\bar\beta^\top W_\tau
                      +\tfrac12\|\bar\beta\|^2\tau).
\]
At a deterministic time $t$, the logarithm of $\Upsilon_t^{-1}$
is Gaussian (possibly degenerate). No such conclusion is asserted
at the exit time $\tau$.
When $\beta=0$, the physical and reference pairs coincide and the
driver is $\alpha(X_t,Y_t)-f_0(X_t)$ with terminal value $g(X_\tau)$.
Applications of this subcase must still satisfy the stated assumptions.

\section{Proofs}
\label{s:Proof}

The deterministic construction and stochastic transfer are treated
separately. The deterministic part studies neural approximations to
the reference elliptic solution operator. The stochastic part proves
Proposition~\ref{prop:FK_TypeResult}, the adapter estimate in
Lemma~\ref{lem:perturbation_bound}, and the conditional transfer in
Theorem~\ref{thrm:Main_Stochastic}. Its complexity conclusion uses
the deterministic architecture at tolerance $\varepsilon/C_T$;
logarithmic depth alone does not determine the total parameter count.

\subsection{
Approximation of the Solution operator to the semilinear elliptic equation}
\label{s:Proof__ss:SO_to_SemilinearElliptic}
\label{s:Proof__ss:ApproximationResult}

\subsubsection{Green potentials and the exact Picard map}

\begin{lemma}[Singular potentials]\label{lem:k1-potential}
The kernels $\Phi$, $\Psi$, and $G$ have weak first $x$ derivatives,
finite $\mathcal{K}$ norm, and map bounded functions to $C^1(\overline{\mathcal{D}})$.
For these kernels, for finite $C^1$ kernels, and for their differences,
\begin{equation}\label{eq:k1-operator-bound}
 \left\|K_Eh\right\|_1\leq\left\|E\right\|_{\mathcal{K}}\left\|h\right\|_\infty,
 \qquad
 \partial_{x_i}K_Eh(x)=
       \int_\mathcal{D}\partial_{x_i}E(x,y)h(y)\,dy.
\end{equation}
The derivative identity holds in the interior and extends
continuously to the boundary.
\end{lemma}
\begin{proof}
It suffices first to consider a kernel satisfying
$|E(x,y)|\leq C|x-y|^{2-d}$ and
$|\nabla_xE(x,y)|\leq C|x-y|^{1-d}$, with the stated off-diagonal
continuity. Radial integration gives, uniformly in $x$,
\begin{equation}\label{eq:k1-radial}
 \int_{\mathcal{D}\cap B(x,r)}|E(x,y)|\,dy\leq Cr^2,
 \qquad
 \int_{\mathcal{D}\cap B(x,r)}|\nabla_xE(x,y)|\,dy\leq Cr.
\end{equation}
Integrating up to a fixed radius larger than $\operatorname{diam}\mathcal{D}$
proves the required integrability.

Choose $\chi\in C^\infty([0,\infty))$ with $0\leq\chi\leq1$,
$\chi=0$ on $[0,1]$, and $\chi=1$ on $[2,\infty)$. Set
$E_t(x,y)=\chi(|x-y|/t)E(x,y)$. This kernel and its first $x$
derivatives are continuous on $\overline{\mathcal{D}}^2$. Differentiation of
the cutoff, together with \eqref{eq:k1-radial}, gives, for $0<t\leq1$,
\begin{equation}\label{eq:k1-cutoff-bound}
 \sup_x\int_\mathcal{D}|E-E_t|(x,y)\,dy\leq Ct^2,
 \quad
 \sup_x\int_\mathcal{D}|\partial_{x_i}E-\partial_{x_i}E_t|(x,y)\,dy
 \leq Ct.
\end{equation}
In the second estimate, the cutoff derivative contributes at most
$Ct^{-1}\int_{|x-y|<2t}|E(x,y)|\,dy\leq Ct$.
Write $E_i$ for the off-diagonal derivative $\partial_{x_i}E$,
assigning any value on the measure-zero diagonal. For
$\varphi\in C_c^\infty(\mathcal D\times\mathcal D)$, integration
by parts for $E_t$ and \eqref{eq:k1-cutoff-bound} give
\[
\begin{aligned}
 \left|\int_{\mathcal D^2}(E\partial_{x_i}\varphi+E_i\varphi)\right|
 &\leq \|E-E_t\|_{L^1(\mathcal D^2)}\|\partial_{x_i}\varphi\|_\infty
       +\|E_i-\partial_{x_i}E_t\|_{L^1(\mathcal D^2)}\|\varphi\|_\infty\\
 &\leq C|\mathcal D|
       \bigl(t^2\|\partial_{x_i}\varphi\|_\infty+t\|\varphi\|_\infty\bigr).
\end{aligned}
\]
Letting $t\downarrow0$ proves $\partial_{x_i}E=E_i$ weakly.
All integrands are absolutely integrable by the displayed estimates,
so the iterated integrations are justified by
{\color{black} Lemma~\ref{lem:product-integration}}.

For bounded $h$, put $F_t=K_{E_t}h$, $F=K_Eh$, and
$F_i(x)=\int_{\mathcal D}E_i(x,y)h(y)\,dy$. The continuity and
boundedness of $E_t$ and its derivatives on the compact product imply
$\partial_{x_i}F_t=K_{\partial_{x_i}E_t}h$. Moreover,
\[
 \|F_t-F\|_\infty\leq Ct^2\|h\|_\infty,\qquad
 \|\partial_{x_i}F_t-F_i\|_\infty\leq Ct\|h\|_\infty.
\]
Thus $F$ and each $F_i$ are continuous on $\overline{\mathcal D}$.
If the segment $\{x+re_i:0\leq r\leq s\}$ lies in an interior
ball, the fundamental theorem of calculus and these uniform bounds give
\[
 F(x+se_i)-F(x)
 =\lim_{t\downarrow0}\int_0^s\partial_{x_i}F_t(x+re_i)\,dr
 =\int_0^sF_i(x+re_i)\,dr.
\]
The same identity holds for negative $s$ with the oriented integral.
Continuity of $F_i$ proves $\partial_{x_i}F=F_i$ in the interior,
with continuous boundary extensions. Integrating the absolute values
of the kernels gives \eqref{eq:k1-operator-bound}.
The bounds for $\Phi$, $\Psi$, and $G$ follow from Assumption~\ref{ass:GreensSymmetry}.
For a finite $C^1$ kernel the argument is ordinary differentiation
under the integral; differences are handled by linearity and the
triangle inequality.
\end{proof}

For each admissible $(f_0,g)$, define
\begin{equation}\label{eq:integral-semilinear}
 T(v)=w_g+K_G R(v),\qquad
 R(v)=\sum_{h=2}^H a_hv^h-f_0,\qquad v\in X_\delta.
\end{equation}

\begin{corollary}[Exact contraction and classical solution]
\label{cor:wellposedness}
The map $T$ sends $X_\delta$ into the $C^1$ ball of radius
$\delta/2$ and has Lipschitz constant at most $1/2$ there.
Its unique fixed point is the unique classical solution of
\eqref{eq:semilinear}-\eqref{eq:semilinear:BC} in $X_\delta$.
\end{corollary}
\begin{proof}
For $v,z\in X_\delta$, the polynomial identity
$v^h-z^h=(v-z)\sum_{i=0}^{h-1}v^{h-1-i}z^i$ gives
\begin{equation}\label{eq:k1-forcing-bounds}
 \|R(v)\|_\infty\leq(1+A_0)\delta^2,\qquad
 \|R(v)-R(z)\|_\infty\leq A_1\delta\|v-z\|_\infty.
\end{equation}
Consequently, Lemma~\ref{lem:k1-potential} gives
\begin{align}
 \|T(v)\|_{C^1}
 &\leq C_S\delta^2+M(1+A_0)\delta^2
   =C_0\delta^2\leq\delta/2,\label{eq:k1-exact-invariance}\\
 \|T(v)-T(z)\|_1
 &\leq MA_1\delta\|v-z\|_\infty
   \leq\tfrac12\|v-z\|_\infty.\label{eq:k1-exact-lipschitz}
\end{align}
The second inequality also bounds the $C^1$ norm. Starting from
$u_0=0$, the iterates $u_{j+1}=T(u_j)$ satisfy
$\|u_{j+1}-u_j\|_{C^1}\leq2^{-j}\|u_1-u_0\|_{C^1}$.
The sum of these bounds is finite, so completeness of the closed
ball $X_\delta$ gives a limit $u\in X_\delta$. Continuity of $T$
gives $T(u)=u$. If $T(u)=u$ and $T(v)=v$, then
$\|u-v\|_{C^1}\leq\frac12\|u-v\|_{C^1}$, whence $u=v$.

We verify the classical equation with the current sign convention.
{\color{black} Apply Lemma~\ref{lem:analytic-approximation} with $U=\mathcal D$
and $u=f_0\in W^{1,\infty}(\mathcal D)$.}
Every $W^{1,\infty}$ source has a locally Lipschitz representative:
on a ball compactly contained in $\mathcal D$, its mollifications
satisfy $|f_t(x)-f_t(y)|\leq\|\nabla f_0\|_\infty|x-y|$ on each
smaller ball. Passing to the limit at pairs of Lebesgue points
gives the same inequality on a dense set; its unique continuous
extension is a locally Lipschitz representative. These representatives
agree on overlaps because they agree almost everywhere.
For $v\in X_\delta$, $R(v)$ is therefore bounded and locally Lipschitz.
Assumption~\ref{ass:GreensSymmetry} now gives
\[
 K_GR(v)\in C^2(\mathcal D)\cap C^1(\overline{\mathcal D}),\qquad
 \mathcal L K_GR(v)=-R(v),\qquad K_GR(v)|_{\partial\mathcal D}=0.
\]
Since $u=T(u)$ and $\delta<\delta_0$, we obtain
$\mathcal Lu=-R(u)=f_0-\alpha(\cdot,u)$ and $u=g$ on the boundary.
Conversely, for a classical solution $u\in X_\delta$, the difference
$u-w_g$ solves $\mathcal L(u-w_g)=-R(u)$ with zero boundary value.
The uniqueness assertion in Assumption~\ref{ass:GreensSymmetry}
gives $u-w_g=K_GR(u)$, so it is the already constructed fixed point.
\end{proof}

\subsubsection{Native wavelet approximation}\label{sec:native}

\begin{lemma}[A fixed extension]\label{lem:k1-extension}
There is a fixed linear extension $\mathcal E$ from
$C^4(\overline{\mathcal{D}})$ to $C^4_c(\mathbb{R}^d)$, with support in a fixed
compact set, satisfying
\begin{equation}\label{eq:k1-extension-bound}
 \left\|\mathcal E f\right\|_{C^q(\mathbb{R}^d)}
 \leq C_q\left\|f\right\|_{C^q(\overline{\mathcal{D}})},\qquad 0\leq q\leq4.
\end{equation}
The same construction is bounded on $W^{1,p}$.
For $F\in C^4(\overline{\mathcal D}^{\,2})$, the extension applied
separately in $x$ and $y$ satisfies
\begin{equation}\label{eq:k1-product-extension}
 \left\|\partial_x^\beta\partial_y^\nu
                 \mathcal E_x\mathcal E_yF\right\|_\infty
 \leq C_{\beta,\nu}
 \sum_{\substack{|\beta'|\leq|\beta|\\|\nu'|\leq|\nu|}}
       \left\|\partial_x^{\beta'}\partial_y^{\nu'}F\right\|_\infty
 \quad (|\beta|+|\nu|\leq4).
\end{equation}
\end{lemma}

\begin{proof}
Choose finitely many smooth boundary charts $\kappa_a$ from
$(-3,3)^d$ onto neighbourhoods $U_a$, taking $t>0$ to
$\mathcal D\cap U_a$ and $t=0$ to its boundary. After shrinking and
rescaling charts, choose a smooth partition
$\eta_0+\sum_{a=1}^A\eta_a=1$ on $\overline{\mathcal D}$, with
$\eta_0$ compactly supported in $\mathcal D$ and
$\operatorname{supp}\eta_a\subset\kappa_a((-1,1)^d)$.
All charts and inverse charts have bounded derivatives through
order four on the compact neighbourhoods used below; their
Jacobian determinants are bounded above and bounded away from zero.

For a boundary piece put $f_a=(\eta_a f)\circ\kappa_a$ on $t\geq0$,
extended by zero at the lateral and upper edges of the chart.
For $t<0$ set
\begin{equation}\label{eq:k1-reflection}
 (\mathcal R f_a)(x',t)=\sum_{i=1}^5\lambda_i f_a(x',-it),
 \qquad(\lambda_1,\ldots,\lambda_5)=(15,-40,45,-24,5);
\end{equation}
for $t\geq0$ set $\mathcal Rf_a=f_a$. Choose a fixed smooth cutoff $\zeta_a$ supported
in $(-2,2)^{d-1}\times(-1/5,2)$ and equal to one on
$[-1,1]^{d-1}\times[0,1]$. On its exterior support,
$-1/5<t<0$, so every reflected normal coordinate $-it$ belongs
to $(0,1)$ and the sampled points lie in the larger chart. Define
\[
 \mathcal Ef=(\eta_0f)^0+
       \sum_{a=1}^A
       \bigl[(\zeta_a\mathcal R f_a)\circ\kappa_a^{-1}\bigr]^0,
\]
where the superscript $0$ denotes extension by zero outside the
interior patch or the chart, respectively. The cutoffs make each
such extension zero near the edge of its patch. On $\mathcal D$
the sum is $(\eta_0+\sum_a\eta_a)f=f$.

For a tangential multi-index $\alpha$ and an integer
$0\leq b\leq4-|\alpha|$, differentiation on $t<0$ gives
\begin{equation}\label{eq:k1-reflection-derivative}
 \partial_{x'}^\alpha\partial_t^b\mathcal Rf_a(x',t)
 =\sum_{i=1}^5\lambda_i(-i)^b
       (\partial_{x'}^\alpha\partial_t^bf_a)(x',-it).
\end{equation}
The five identities
\begin{equation}\label{eq:k1-matching}
 \sum_{i=1}^5\lambda_i(-i)^b=1\qquad(0\leq b\leq4)
\end{equation}
give equal limits of these derivatives from both sides at $t=0$.
Also,
\[
 \|\partial_{x'}^\alpha\partial_t^b\mathcal Rf_a\|_\infty
 \leq \max\left\{1,\sum_{i=1}^5|\lambda_i|i^b\right\}
             \|\partial_{x'}^\alpha\partial_t^bf_a\|_\infty.
\]
To account explicitly for charts and cutoffs, the iterated product
and chain rules have the form
\[
 \partial^\alpha((\eta_a f)\circ\kappa_a)
 =\sum_{|\beta|\leq|\alpha|}B_{a,\alpha,\beta}
                    (\partial^\beta f)\circ\kappa_a,\qquad
 \max_{a,\alpha,\beta}\|B_{a,\alpha,\beta}\|_\infty<\infty
       \quad(|\alpha|\leq4).
\]
Here the coefficients are finite sums of products of derivatives
of the fixed chart and cutoff. The inverse-chart and exterior-cutoff
formula has the same form. Combining these two formulas with
\eqref{eq:k1-reflection-derivative}, and summing over the finite
cover, proves \eqref{eq:k1-extension-bound}. This construction
is linear and has support in one fixed compact set.
It acts only on the extended variable. Applying the displayed
formulas in $y$ and then in $x$ yields
\[
 \partial_x^\beta\partial_y^\nu\mathcal E_x\mathcal E_y F
 =\sum_{\substack{|\beta'|\leq|\beta|\\|\nu'|\leq|\nu|}}
        \mathcal T_{\beta,\nu,\beta',\nu'}
                (\partial_x^{\beta'}\partial_y^{\nu'}F),
\]
where each $\mathcal T$ is a finite sum of fixed pullbacks
multiplied by bounded functions. This proves
\eqref{eq:k1-product-extension}.

For the Sobolev estimate, put $S=\sum_i|\lambda_i|$. Weighted
H\"older and the substitution $u=-it$ give, for smooth localized
data and $|\alpha|+b\leq1$,
\[
\begin{aligned}
 \int_{t<0}|\partial_{x'}^\alpha\partial_t^b\mathcal Rf_a|^p\,dx'\,dt
 &\leq S^{p-1}\sum_{i=1}^5|\lambda_i|i^{bp}
       \int_{t<0}|\partial_{x'}^\alpha\partial_t^bf_a(x',-it)|^p\,dx'\,dt\\
 &=S^{p-1}\sum_{i=1}^5|\lambda_i|i^{bp-1}
       \int_{u>0}|\partial_{x'}^\alpha\partial_u^bf_a(x',u)|^p\,dx'\,du.
\end{aligned}
\]
The derivative and Jacobian bounds for the charts, and the product
rule for their cutoffs, therefore give
$\|\mathcal Ef\|_{W^{1,p}(\mathbb R^d)}
\leq C\|f\|_{W^{1,p}(\mathcal D)}$ for smooth $f$.
For $\varphi\in C_c^\infty(\mathbb R^d)$, integration by parts
on the two half-spaces gives
\[
 \int_{\mathbb R^d}\mathcal Rf_a\,\partial_t\varphi
 =-\int_{t\ne0}\partial_t\mathcal Rf_a\,\varphi
   +\int_{\mathbb R^{d-1}}
       \bigl(\mathcal Rf_a(x',0-)-f_a(x',0+)\bigr)\varphi(x',0)\,dx'.
\]
The boundary term is zero by \eqref{eq:k1-matching} with $b=0$.
Tangential integration introduces no interface term. Hence the
piecewise derivatives used in the estimate are the weak derivatives.

We next justify changes of coordinates for arbitrary Sobolev input.
Let $\kappa:V\to U$ be a smooth chart restricted to one of the fixed
neighbourhoods above, with
$J=|\det D\kappa|$, $c_\kappa=\inf_VJ>0$, and
$M_\kappa=\sup_V\|D\kappa\|_{\mathrm{op}}<\infty$.
For $1\leq i\leq d$ define the smooth vector field
$
 B_i(x)=\frac{D\kappa(\kappa^{-1}(x))e_i}
                    {J(\kappa^{-1}(x))}$ for each $x\in U$.
For $\xi\in C_c^\infty(U)$, change of variables and the classical
chain rule for $\xi$ give
\[
 \int_U B_i\cdot\nabla\xi\,dx
 =\int_V\partial_i(\xi\circ\kappa)\,dy=0.
\]
Thus $\operatorname{div}B_i=0$. If $q\in W^{1,p}(U)$ and
$\varphi\in C_c^\infty(V)$, the weak derivative identity applied
to the vector test function $B_i(\varphi\circ\kappa^{-1})$ gives
\[
\begin{aligned}
 \int_V(q\circ\kappa)\partial_i\varphi\,dy
 =\int_Uq B_i\cdot\nabla(\varphi\circ\kappa^{-1})\,dx
 =-\int_U\nabla q\cdot B_i\,(\varphi\circ\kappa^{-1})\,dx
 =-\int_V\sum_{r=1}^d
       (\partial_rq)\circ\kappa\,\partial_i\kappa_r\,
       \varphi\,dy.
\end{aligned}
\]
All test functions have compact support inside the indicated open
sets, so no boundary value of $q$ is used. Consequently
\begin{equation}\label{eq:k1-weak-chart}
 \partial_i(q\circ\kappa)
 =\sum_{r=1}^d(\partial_rq)\circ\kappa\,\partial_i\kappa_r
 \quad\hbox{weakly on }V.
\end{equation}
Changing variables once more yields
\begin{equation}\label{eq:k1-chart-bounds}
\begin{aligned}
 \|q\circ\kappa\|_{L^p(V)}
 &\leq c_\kappa^{-1/p}\|q\|_{L^p(U)},\\
 \|\nabla(q\circ\kappa)\|_{L^p(V)}
 &\leq M_\kappa c_\kappa^{-1/p}\|\nabla q\|_{L^p(U)}.
\end{aligned}
\end{equation}
Apply \eqref{eq:k1-chart-bounds} to $\kappa^{-1}:U\to V$.
Its Jacobian is $(J\circ\kappa^{-1})^{-1}$. Thus, with
$C_\kappa=\sup_VJ$ and
$\widehat M_\kappa=\sup_U\|D\kappa^{-1}\|_{\mathrm{op}}$, for
$v\in W^{1,p}(V)$ we obtain
\begin{equation}\label{eq:k1-inverse-chart-bounds}
\begin{aligned}
 \|v\circ\kappa^{-1}\|_{L^p(U)}
 &\leq C_\kappa^{1/p}\|v\|_{L^p(V)},\\
 \|\nabla(v\circ\kappa^{-1})\|_{L^p(U)}
 &\leq \widehat M_\kappa C_\kappa^{1/p}\|\nabla v\|_{L^p(V)}.
\end{aligned}
\end{equation}
The constants are finite on the fixed compact chart neighbourhoods.
The argument uses only test functions in the open domains and therefore
applies to the positive half-charts as well.
For these applications restrict $\kappa_a$ to $(-5/2,5/2)^d$
and replace $U_a$ by its image; all cutoffs below are supported
strictly inside this smaller chart.

Now fix $f\in W^{1,p}(\mathcal D)$. Testing its weak derivative
against $\eta_a\xi$ proves
$\partial_r(\eta_af)=\eta_a\partial_rf+f\partial_r\eta_a$.
Apply \eqref{eq:k1-weak-chart}--\eqref{eq:k1-chart-bounds} to
$q=\eta_af$ on $\mathcal D\cap U_a$. Its pullback vanishes near
the lateral and upper chart edges. Extension by zero across those
edges therefore has the same weak derivatives extended by zero:
this follows by inserting a smooth cutoff equal to one on its
support into each test function. With
$\mathbb H=\mathbb R^{d-1}\times(0,\infty)$, this proves
\[
 f_a\in W^{1,p}(\mathbb H),\qquad
 \|f_a\|_{W^{1,p}(\mathbb H)}
       \leq C_a\|f\|_{W^{1,p}(\mathcal D)}.
\]
Smooth approximation on each half-chart can now be performed
without extending weak derivatives across its physical boundary.
Let {\color{black}$\rho\geq0$} be a smooth unit-mass mollifier supported in the unit ball.
For $t\geq0$ use
\[
 f_{a,\varepsilon}(x',t)
   =\int\rho_{\varepsilon/2}(s)\,
          f_a(x'-s',t+2\varepsilon-s_d)\,ds.
\]
Every sampled normal coordinate is at least $3\varepsilon/2$.
For $\varphi\in C_c^\infty(\mathbb H)$ and $|s|<\varepsilon/2$,
the function $y\mapsto\varphi(y-2\varepsilon e_d+s)$, extended by
zero, has compact support inside $\mathbb H$. The weak derivative
identity for $f_a$, followed by integration in $s$, therefore gives
\[
 \int_{\mathbb H}f_{a,\varepsilon}\partial_i\varphi
 =-\int_{\mathbb H}\!\int\rho_{\varepsilon/2}(s)
       (\partial_if_a)(z+2\varepsilon e_d-s)\varphi(z)\,ds\,dz.
\]
{\color{black} The interchange uses the Fubini theorem for the
zero extension $g$ of $f_a$ or $\partial_i f_a$ and for
$\psi=\partial_i\varphi$ or $\varphi$, respectively,
\[
\int\!\int_{\mathbb H}\rho_{\varepsilon/2}(s)
 |g(z+2\varepsilon e_d-s)\psi(z)|\,dz\,ds
\leq\|g\|_{L^p(\mathbb R^d)}\|\psi\|_{L^{p'}(\mathbb H)}<\infty.
\]}
If $g_\beta$ denotes the zero extension of $\partial^\beta f_a$
as an $L^p$ function, this proves
$\partial^\beta f_{a,\varepsilon}
= (\rho_{\varepsilon/2}*g_\beta)(\,\cdot+2\varepsilon e_d)$
on $\mathbb H$ for $|\beta|\leq1$. {\color{black} Since $g_\beta\in L^p(\mathbb R^d)$
and $d<p<\infty$, Lemma~\ref{lem:analytic-approximation}, with
convolution radius $\varepsilon/2$} and {\color{black}translation $2\varepsilon e_d$,
gives}
\[
 \|\partial^\beta f_{a,\varepsilon}-\partial^\beta f_a\|_{L^p(t>0)}
 \leq
 \|(\rho_{\varepsilon/2}*g_\beta)(\,\cdot+2\varepsilon e_d)-g_\beta\|_{L^p(\mathbb R^d)},\qquad
 \lim_{\varepsilon\downarrow0}
 \|(\rho_{\varepsilon/2}*g_\beta)(\,\cdot+2\varepsilon e_d)-g_\beta\|_{L^p}=0.
\]
These $L^p$ convergence statements follow first for compactly
supported continuous functions from uniform continuity and then
for $L^p$ functions by density and the norm-one bounds for
translations and convolution.

To give the assembly estimate, choose fixed smooth cutoffs
$\omega_a\in C_c^\infty((-2,2)^d)$ equal to one
on a neighbourhood of $[-1,1]^{d-1}\times[0,1]$. They vanish near
the lateral and upper chart edges. Since $\omega_af_a=f_a$,
the product rule gives
\[
 \|\omega_af_{a,\varepsilon}-f_a\|_{W^{1,p}(\mathbb H)}
 =\|\omega_a(f_{a,\varepsilon}-f_a)\|_{W^{1,p}(\mathbb H)}
 \leq C_a\|f_{a,\varepsilon}-f_a\|_{W^{1,p}(\mathbb H)}.
\]
Put
$u_{a,\varepsilon}=[(\omega_af_{a,\varepsilon})\circ\kappa_a^{-1}]^0$
on $\mathcal D$, where zero extension is only across the artificial
chart edges. The cutoff makes that extension smooth there. The
inverse-coordinate bound and the same test-function argument give
\[
 \|u_{a,\varepsilon}-\eta_af\|_{W^{1,p}(\mathcal D)}
 \leq C_a\|f_{a,\varepsilon}-f_a\|_{W^{1,p}(\mathbb H)}.
\]
For the interior piece, $h_0=(\eta_0f)^0\in W^{1,p}(\mathbb R^d)$
by that argument, and
$u_{0,\varepsilon}=(\rho_\varepsilon*h_0)|_{\mathcal D}$ is supported
inside $\mathcal D$ for small $\varepsilon$. Its weak derivatives are
$\rho_\varepsilon*\partial_rh_0$, so
$\lim_{\varepsilon\downarrow0}
\|u_{0,\varepsilon}-\eta_0f\|_{W^{1,p}(\mathcal D)}=0$.
Define $f_\varepsilon=u_{0,\varepsilon}+\sum_{a=1}^Au_{a,\varepsilon}$.
The inward translation makes each $f_{a,\varepsilon}$ smooth in a
neighbourhood of $t\geq0$, hence
$f_\varepsilon\in C^\infty(\overline{\mathcal D})$. The partition
identity and the preceding bounds yield
\[
 \|f_\varepsilon-f\|_{W^{1,p}(\mathcal D)}
 \leq\|u_{0,\varepsilon}-\eta_0f\|_{W^{1,p}(\mathcal D)}
     +\sum_{a=1}^AC_a
             \|f_{a,\varepsilon}-f_a\|_{W^{1,p}(\mathbb H)},\qquad
 \lim_{\varepsilon\downarrow0}\|f_\varepsilon-f\|_{W^{1,p}(\mathcal D)}=0.
\]
Choose $\varepsilon_n\downarrow0$ and set $f_n=f_{\varepsilon_n}$.
The extension estimate makes $(\mathcal Ef_n)_n$ Cauchy in
$W^{1,p}(\mathbb R^d)$; its limit is independent of the approximating
sequence and restricts to $f$. For $C^4$ input the same local
approximations converge in $C^4$ on the closed half-patches,
so the limit agrees with the extension already defined there.
Linearity passes to the limit, and the smooth estimate yields
\[
 \|\mathcal Ef\|_{W^{1,p}(\mathbb R^d)}
 =\lim_{n\uparrow\infty}\|\mathcal Ef_n\|_{W^{1,p}(\mathbb R^d)}
 \leq C\lim_{n\uparrow\infty}\|f_n\|_{W^{1,p}(\mathcal D)}
 =C\|f\|_{W^{1,p}(\mathcal D)}
\]
and we are done.
\end{proof}

\begin{lemma}[Local projection and fourth-order accuracy]
\label{lem:k1-projection}
For $q\in\{1,d,2d\}$, let $P_j^{(q)}$ be the orthogonal projection
onto the scaling space with basis $\{\Theta^{(q)}_{j,m}\}_m$.
Its kernel
\begin{equation}\label{eq:k1-projection-kernel}
 p_j^{(q)}(z,w)=\sum_{m\in\mathbb{Z}^q}
            \Theta^{(q)}_{j,m}(z)\Theta^{(q)}_{j,m}(w)
\end{equation}
vanishes if $|z-w|>C2^{-j}$ and satisfies, for $|\beta|\leq1$,
\begin{equation}\label{eq:k1-projection-row}
 \sup_z\int_{\mathbb{R}^q}|\partial_z^\beta p_j^{(q)}(z,w)|\,dw
 \leq C2^{j|\beta|},
 \qquad
 \sup_w\int_{\mathbb{R}^q}|p_j^{(q)}(z,w)|\,dz\leq C.
\end{equation}
It reproduces every polynomial of total degree at most three.
Consequently, for $F\in C^4_c(\mathbb{R}^q)$ and $b=2^{-j}$,
\begin{equation}\label{eq:k1-projection-error}
 \left\|\partial^\beta(P_j^{(q)}F-F)\right\|_\infty
 \leq C b^{4-|\beta|}
             \max_{|\gamma|=4}\left\|\partial^\gamma F\right\|_\infty,
 \qquad |\beta|\leq1.
\end{equation}
Moreover, $
 \operatorname{supp}(P_j^{(q)}F)\subset
       \{z:\operatorname{dist}(z,\operatorname{supp} F)\leq Cb\}
$.
\end{lemma}
\begin{proof}
The support of $\theta$ is fixed and compact. For any fixed $z$,
only a bounded number of the functions
$\Theta^{(q)}_{j,m}$ or their first derivatives are nonzero at $z$.
Their sup norms are bounded by $C2^{jq/2}$ and
$C2^{j(q/2+1)}$, respectively, while their $L^1$ norms are bounded
by $C2^{-jq/2}$. Summing over those indices proves the row estimate.
The identity $p_j^{(q)}(z,w)=p_j^{(q)}(w,z)$ turns its zeroth-order
case into the column estimate. If one summand of the kernel is nonzero,
$z$ and $w$ belong to the same translated, dilated compact support;
hence $|z-w|\leq Cb$. This also proves
the support containing claim.

We verify constant reproduction directly. In one dimension set
$A(t)=\sum_{m\in\mathbb{Z}}\theta(t-m)$, a continuous one-periodic
function. Let $F\in C^\infty_c(\mathbb{R})$ equal one on a neighbourhood
of $[0,1]$. For all sufficiently large $j$ and $x\in[0,1]$,
the normalization $\int\theta=1$ gives
$P_j^{(1)}F(x)=A(2^jx)$. The multiresolution density property gives
$\lim_{j\uparrow\infty}\left\|P_j^{(1)}F-F\right\|_{L^2}=0$. But periodicity
and the integer $2^j$ imply
\[
 \int_0^1|A(2^jx)-1|^2\,dx
       =\int_0^1|A(t)-1|^2\,dt.
\]
The right side is therefore zero, and continuity gives $A=1$.
Tensorization proves $P_j^{(q)}1=1$.

For a monomial $Q(w)=w^\alpha$ with $|\alpha|\leq3$, change
variables $t_i=2^jw_i-a_i$ in its coefficient against
\begin{equation}\label{eq:k1-detail-basis}
 \Upsilon_{j,a,e}^{(q)}(z)
 =2^{jq/2}\prod_{i=1}^q\gamma_{e_i}(2^jz_i-a_i),
 \quad a\in\mathbb Z^q,\quad e\in\{0,1\}^q\setminus\{0\}.
\end{equation}
yields
$
 \int_{\mathbb R^q}w^\alpha\Upsilon_{j,a,e}^{(q)}(w)\,dw
 =2^{-j(q/2+|\alpha|)}
   \prod_{i=1}^q\int_{\mathbb R}(t_i+a_i)^{\alpha_i}
                          \gamma_{e_i}(t_i)\,dt_i=0
$.  Now, linearity gives the same zero coefficient for every polynomial
$Q$ of total degree at most three.

Fix $j$ and $z$. Only finitely many level-$j$ detail functions and
level-$j$ or level-$(j+1)$ scaling functions are nonzero at $z$.
Choose $\vartheta\in C_c^\infty(\mathbb R^q)$ equal to one on the
union of their supports. Apply the orthogonal decomposition
$V_{j+1}^{(q)}=V_j^{(q)}\oplus W_j^{(q)}$ to
$\vartheta Q\in L^2(\mathbb R^q)$. Its finite-level outputs are
continuous, so their $L^2$ identity holds at every point. At $z$,
the choice of $\vartheta$ gives
\[
 (P_{j+1}^{(q)}Q)(z)-(P_j^{(q)}Q)(z)
 =\sum_{a,e}\Upsilon_{j,a,e}^{(q)}(z)
       \int_{\mathbb R^q}\Upsilon_{j,a,e}^{(q)}(w)Q(w)\,dw=0.
\]
Every sum at this fixed evaluation point is finite; no $L^2$
membership of $Q$ is assumed.
On the other hand, the locality and row bound give
$
 |P_j^{(q)}Q(z)-Q(z)|
 \leq C\sup_{|w-z|\leq C2^{-j}}|Q(w)-Q(z)|$, 
whose limit is zero. Since $P_j^{(q)}Q(z)$ is independent of $j$,
it equals $Q(z)$ for every $j$.

For a fixed $z$, let $Q_z$ be the Taylor polynomial of degree three of $F$ at $z$. Keep this polynomial fixed when differentiating in the evaluation variable. Polynomial reproduction and the Taylor
remainder estimate imply
$
 |\partial^\beta(P_j^{(q)}F-F)(z)|
 \leq\int|\partial_z^\beta p_j^{(q)}(z,w)|
                         |F(w)-Q_z(w)|\,dw
 \leq Cb^{4-|\beta|}\max_{|\gamma|=4}
                              \left\|\partial^\gamma F\right\|_\infty$.
\end{proof}

\begin{lemma}[Native finite spans]\label{lem:k1-frame}
For $j\geq1$,
\begin{equation}\label{eq:k1-count-and-constant}
 |\Lambda_j|\leq C2^{jd},\qquad
 1=\sum_{n\in\Lambda_j}\kappa_n\varphi_n,\qquad
 \kappa_n=
 \begin{cases}1,&n\text{ is a coarse index},\\0,&n\text{ is a wavelet index}.
 \end{cases}
\end{equation}
For compactly supported $F\in L^2(\mathbb R^d)$ and
$\widetilde F\in L^2(\mathbb R^{2d})$,
\begin{align}
 (P_j^{(d)}F)|_{\mathcal D}
 &=\sum_{n\in\Lambda_j}\langle F,\Xi_n\rangle\,\varphi_n,
 \label{eq:k1-native-one}\\
 (P_j^{(2d)}\widetilde F)|_{\mathcal D^2}
 &=\sum_{n,m\in\Lambda_j}
     \langle\widetilde F,\Xi_n\otimes\Xi_m\rangle\,
       \varphi_n\otimes\varphi_m.\label{eq:k1-native-two}
\end{align}
The inner products in these identities are over the corresponding
ambient Euclidean spaces.
\end{lemma}
\begin{proof}
Choose $R,S\geq1$ such that $\mathcal D\subset[-R,R]^d$ and
the support of every tensor generator is contained in $[-S,S]^d$.
If a translate at level $\ell$ meets $\mathcal D$, each of its
translation coordinates has absolute value at most $2^\ell R+S$.
There are finitely many wavelet orientations, hence
\[
 |\Lambda_j|\leq C+C\sum_{\ell=0}^{j-1}(2^\ell R+S+1)^d
 \leq C+C'\frac{2^{jd}-1}{2^d-1}\leq C''2^{jd}.
\]
The level-zero partition of unity proved in
Lemma~\ref{lem:k1-projection} gives the constant expansion.
Indices omitted from $\nabla$ vanish on $\mathcal D$ and,
by continuity, on its closure.
The orthogonal identity
$V_j^{(d)}=V_0^{(d)}\oplus\bigoplus_{\ell=0}^{j-1}W_\ell^{(d)}$
gives \eqref{eq:k1-native-one}. Compact support leaves finitely
many nonzero coefficients at each of the finitely many levels.
Since $P_j^{(2d)}=P_j^{(d)}\otimes P_j^{(d)}$, expansion in the
tensor product of these finite-level orthonormal bases gives
\eqref{eq:k1-native-two}. Restricting removes exactly the terms
whose indices are outside $\nabla$.
\end{proof}

\begin{lemma}[Native kernel and coefficient approximation]\label{lem:expansion-Green}
For every $0<\eta\leq1$ there are $j\geq1$, $N=|\Lambda_j|$,
and real coefficients $c_{nm},d_{hn},\kappa_n$ such that
\begin{equation}\label{eq:k1-finite-expansions}
 \Psi_\eta(x,y)=\sum_{n,m\in\Lambda_j}c_{nm}\varphi_n(x)\varphi_m(y),
 \quad
 a_{h,\eta}(x)=\sum_{n\in\Lambda_j}d_{hn}\varphi_n(x),
 \quad
 1=\sum_{n\in\Lambda_j}\kappa_n\varphi_n(x),
\end{equation}
with
\begin{equation}\label{eq:k1-approximation-rate}
 \left\|\Psi-\Psi_\eta\right\|_{\mathcal{K}}\leq\eta,
 \qquad \max_{2\leq h\leq H}\left\|a_h-a_{h,\eta}\right\|_\infty\leq\eta,
 \qquad N\leq C\eta^{-d}.
\end{equation}
The set and coefficients are independent of $(f_0,g)$ and of the
Picard iteration number. All finite approximants are $C^4$ on their
closed domains. The constants depend only on the fixed hypotheses
and the fixed wavelet generators.
\end{lemma}

\begin{proof}
Fix the cutoff $\chi$ used in Lemma~\ref{lem:k1-potential} and let
$r_0\geq\max\{1,2\operatorname{diam}\mathcal{D}\}$,
$r_\ell=r_0 2^{-\ell}$. For $\ell\geq1$, define on $\mathcal{D}^2$
\begin{equation}\label{eq:k1-shells}
 F_\ell(x,y)=
 \left[\chi\left(\frac{|x-y|}{r_\ell}\right)
       -\chi\left(\frac{|x-y|}{r_{\ell-1}}\right)\right]\Psi(x,y).
\end{equation}
These functions vanish near the diagonal and belong to
$C^4(\overline{\mathcal{D}}^2)$. Their support is contained in
$\{r_\ell\leq|x-y|\leq4r_\ell\}$. By the product rule and
\eqref{eq:unlifted-green-bounds},
\begin{equation}\label{eq:k1-shell-derivatives}
 \left\|\partial_x^\beta\partial_y^\nu F_\ell\right\|_\infty
 \leq C r_\ell^{2-d-|\beta|-|\nu|},
 \qquad |\beta|+|\nu|\leq4.
\end{equation}
The telescoping identity is
\begin{equation}\label{eq:k1-shell-telescope}
 \sum_{\ell=1}^jF_\ell(x,y)
   =\chi\left(\frac{|x-y|}{r_j}\right)\Psi(x,y).
\end{equation}
For $E_j=(1-\chi(|x-y|/r_j))\Psi$, radial integration gives, uniformly in $x$,
\[
 \int_{\mathcal D}|E_j(x,y)|\,dy
 \leq C\int_0^{2r_j}t\,dt=2Cr_j^2,
\]
\[
 \sum_{i=1}^d\int_{\mathcal D}|\partial_{x_i}E_j(x,y)|\,dy
 \leq C\int_0^{2r_j}dt+\frac{C}{r_j}\int_{r_j}^{2r_j}t\,dt
 \leq C'r_j.
\]
The second term is the derivative of the cutoff. Since
$r_j\leq r_0$, these inequalities prove
\begin{equation}\label{eq:k1-shell-tail}
 \left\|\Psi-\sum_{\ell=1}^jF_\ell\right\|_{\mathcal{K}}\leq Cr_j.
\end{equation}

Extend $F_\ell$ by $\mathcal E_x\mathcal E_y$.
Choose a fixed smooth $\zeta$ equal to one on $[0,4]$ and zero
on $[8,\infty)$, and set
\begin{equation}\label{eq:k1-shell-extension}
 \widetilde F_\ell(x,y)=
 \zeta\left(\frac{|x-y|}{r_\ell}\right)
                 (\mathcal E_x\mathcal E_yF_\ell)(x,y).
\end{equation}
This is a $C^4_c(\mathbb{R}^{2d})$ extension of $F_\ell$, with support in
a fixed compact set and in the band $|x-y|\leq8r_\ell$.
Equations \eqref{eq:k1-product-extension} and
\eqref{eq:k1-shell-derivatives}, followed by the product rule, give
\begin{equation}\label{eq:k1-extended-shell-derivatives}
 \left\|\partial_x^\beta\partial_y^\nu\widetilde F_\ell\right\|_\infty
 \leq C r_\ell^{2-d-|\beta|-|\nu|},
 \qquad |\beta|+|\nu|\leq4.
\end{equation}
Lower derivative orders produced by the charts obey the same
bound after increasing the constant, since $r_\ell\leq r_0$.
Use one common dyadic scale $b=2^{-j}$ for every shell and put
\begin{equation}\label{eq:k1-dyadic-approximants}
 \Psi_b=\left.\sum_{\ell=1}^jP_j^{(2d)}\widetilde F_\ell
                   \right|_{\overline{\mathcal{D}}^2},
 \qquad
 a_{h,b}=(P_j^{(d)}\mathcal E a_h)|_{\overline{\mathcal{D}}}.
\end{equation}
For $|\beta|\leq1$, Lemma~\ref{lem:k1-projection} gives
\begin{equation}\label{eq:k1-shell-projection-error}
 \left\|\partial_x^\beta
       (P_j^{(2d)}\widetilde F_\ell-\widetilde F_\ell)\right\|_\infty
 \leq Cb^{4-|\beta|}r_\ell^{-d-2}.
\end{equation}
Both terms in this error, and their derivatives, are supported in
$|x-y|\leq8r_\ell+Cb$. Since $b\leq r_j\leq r_\ell$,
the $y$ section of that support has volume at most $Cr_\ell^d$,
uniformly in $x$. Restricting to $\mathcal{D}^2$, where
$\widetilde F_\ell=F_\ell$, therefore yields
\begin{equation}\label{eq:k1-shell-mixed-error}
 \left\|\left.P_j^{(2d)}\widetilde F_\ell\right|_{\mathcal{D}^2}
                          -F_\ell\right\|_{\mathcal{K}}
 \leq C(b^4+b^3)r_\ell^{-2}
 \leq Cb^3r_\ell^{-2}.
\end{equation}
The geometric sum satisfies
$
 \sum_{\ell=1}^j r_\ell^{-2}
 =\frac{4(4^j-1)}{3r_0^2}\leq\frac{4}{3}r_j^{-2}$.  
Consequently,
\begin{equation}\label{eq:k1-key-estimate}
 \left\|\Psi-\Psi_b\right\|_{\mathcal{K}}
 \leq Cr_j+Cb^3\sum_{\ell=1}^jr_\ell^{-2}
 \leq Cr_j+Cb^3r_j^{-2}
 \leq Cb,
 \qquad r_j=r_0b.
\end{equation}
The shell errors form a geometric sum dominated by its finest
scale. The fourth-order derivatives in
\eqref{eq:k1-extended-shell-derivatives} are used in both variables.
For the coefficient functions, the same projection lemma gives
\begin{equation}\label{eq:k1-coefficient-error}
 \max_{2\leq h\leq H}\left\|a_h-a_{h,b}\right\|_\infty\leq Cb^4.
\end{equation}
Choose $C_{\rm app}\geq1$ such that the bounds in
\eqref{eq:k1-key-estimate} and \eqref{eq:k1-coefficient-error} are both
at most $C_{\rm app}b$, and set
\begin{equation}\label{eq:k1-choice-j}
 j=\left\lceil\log_2\frac{2C_{\rm app}}{\eta}\right\rceil,
 \qquad b=2^{-j},\qquad
 \Psi_\eta=\Psi_b,\quad a_{h,\eta}=a_{h,b}.
\end{equation}
Then the errors are at most $\eta$. Lemma~\ref{lem:k1-frame} supplies
the common index set $\Lambda_j$, the exact constant expansion,
and the explicit coefficients
\begin{equation}\label{eq:k1-coefficients}
 c_{nm}=\sum_{\ell=1}^j
       \left\langle \widetilde F_\ell,\Xi_n\otimes\Xi_m\right\rangle_{L^2(\mathbb{R}^{2d})},
 \qquad
 d_{hn}=\left\langle \mathcal E a_h,\Xi_n\right\rangle_{L^2(\mathbb{R}^d)}.
\end{equation}
Finally, $2^j\leq4C_{\rm app}/\eta$, so
$N=|\Lambda_j|\leq C2^{jd}\leq C\eta^{-d}$.
Every choice depends only on the fixed reference problem and
$\eta$. This proves the lemma.
\end{proof}

\subsubsection{Approximate Picard iteration}
The following perturbation argument is the counterpart of
Lemmas A.5--A.7 in the supplied proof. Its kernel error is
$G_\eta-G=\Psi_\eta-\Psi$; its additional coefficient error is
included explicitly.
\begin{proposition}[Perturbed Picard iteration]\label{prop:k1-perturbed}
For the approximants in Lemma~\ref{lem:expansion-Green}, set
\begin{equation}\label{eq:k1-approx-picard}
 G_\eta=\Phi+\Psi_\eta,\qquad
 T_\eta(v)=w_g+K_{G_\eta}
                  \left(\sum_{h=2}^H a_{h,\eta}v^h-f_0\right).
\end{equation}
For $0<\eta\leq1$, this map sends $X_\delta$ into its ball of
radius $\delta/2$, and
\begin{equation}\label{eq:k1-map-error}
 \sup_{v\in X_\delta}\left\|T_\eta(v)-T(v)\right\|_1
 \leq E_0\delta^2\eta.
\end{equation}
If $v_0=0$ and $v_{j+1}=T_\eta(v_j)$, then, for every $J\geq1$,
$
 \left\|u-v_J\right\|_1\leq2^{-J}\delta+2E_0\delta^2\eta$.  
In particular, with
\begin{equation}\label{eq:k1-accuracy-allocation}
 \eta=\frac{\varepsilon}{4(1+E_0)},\qquad
 J_\varepsilon=\left\lceil\log_2\frac2\varepsilon\right\rceil,
 \qquad 0<\varepsilon<1,
\end{equation}
one has $\left\|u-v_{J_\varepsilon}\right\|_{W^{1,\infty}}\leq\varepsilon$, uniformly
over $\mathcal X_{1,\delta,K}$.
\end{proposition}
\begin{proof}
Let $R(v)=\sum a_hv^h-f_0$ and
$D_\eta(v)=\sum(a_{h,\eta}-a_h)v^h$. Then
\begin{equation}\label{eq:k1-perturbation-identity}
 T_\eta(v)-T(v)
 =K_{\Psi_\eta-\Psi}R(v)+K_{G_\eta}D_\eta(v).
\end{equation}
Equations \eqref{eq:k1-forcing-bounds} and
\eqref{eq:k1-approximation-rate} give
$\left\|R(v)\right\|_\infty\leq(1+A_0)\delta^2$,
$\left\|D_\eta(v)\right\|_\infty\leq m\eta\delta^2$, and
$\left\|G_\eta\right\|_{\mathcal{K}}\leq M+1$.
Lemma~\ref{lem:k1-potential} proves \eqref{eq:k1-map-error} with exactly
$E_0$ from Assumption~\ref{ass:choice-delta-p}. It also gives $C^1$ regularity
of $T_\eta(v)$. Equations \eqref{eq:k1-exact-invariance},
\eqref{eq:k1-map-error}, and \eqref{eq:unlifted-smallness} imply
$
 \left\|T_\eta(v)\right\|_{C^1}
 \leq(C_0+E_0)\delta^2\leq\delta/2$.
Thus all iterates belong to $X_\delta$. We do not need to claim
that $T_\eta$ itself is a contraction.
Write $e_j=\left\|u-v_j\right\|_1$ and $B=E_0\delta^2\eta$.
For $j\geq1$, \eqref{eq:k1-exact-lipschitz} and
\eqref{eq:k1-map-error} imply
\begin{equation}\label{eq:k1-error-recurrence}
 e_{j+1}\leq\tfrac12 e_j+B.
\end{equation}
To respect the distinction between the two norms, i.e.\ $\|\cdot\|_{C^1}\leq\|\cdot\|_1\leq d\|\cdot\|_{C^1}$ on $C^1(\overline{\mathcal D})$, use $u=T(u)$ and $v_1=T_\eta(0)$
to obtain
\begin{equation}\label{eq:k1-first-error}
 e_1\leq\left\|T(u)-T(0)\right\|_1+B
 \leq MA_0\delta^2+B.
\end{equation}
For $J\geq1$, induction in \eqref{eq:k1-error-recurrence} gives
\[
\begin{aligned}
 e_J
 \leq2^{-(J-1)}e_1+B\sum_{r=0}^{J-2}2^{-r}\leq2^{-(J-1)}MA_0\delta^2
                +\bigl(2-2^{-(J-1)}\bigr)B
 \leq2^{-(J-1)}MA_0\delta^2+2B.
\end{aligned}
\]
The sum is empty for $J=1$. This is the base case, and multiplication
by $1/2$ followed by addition of $B$ gives the next case.
Since $MA_0\leq C_0$ and $C_0\delta^2\leq\delta/2$, this is
the estimate $\left\|u-v_J\right\|_1\leq2^{-J}\delta+2E_0\delta^2\eta$. With \eqref{eq:k1-accuracy-allocation}
and $\delta\leq1$,
$
 e_{J_\varepsilon}\leq\tfrac\varepsilon2
       +\tfrac{E_0\delta^2}{2(1+E_0)}\varepsilon\leq\varepsilon
$.
All constants and choices are uniform over the data class.
\end{proof}

\subsubsection{Exact realization in the prescribed network}

\begin{lemma}[Well-defined pointwise operations]\label{lem:unlifted-extension}
For $p>d$, the space $W^{1,p}(\mathcal{D})$ embeds continuously into
$C(\overline{\mathcal{D}})$. Products and composition with $\operatorname{ReQU}$ preserve
$W^{1,p}(\mathcal{D})$. The layers used below therefore define maps on all
of $W^{1,p}(\mathcal{D})^2$, not only on the admissible data.
\end{lemma}
\begin{proof}
Fix a nonnegative $\rho\in C_c^\infty(B(0,1))$ of integral one,
and put $\rho_r(y)=r^{-d}\rho(y/r)$. For $w\in C_c^\infty(\mathbb R^d)$,
the segment formula and $z=ty$ give
\[
\begin{aligned}
 |w(x)-\rho_r*w(x)|
 &\leq\|\rho\|_\infty r^{-d}
     \int_{|y|<r}\int_0^1 |y|\,|\nabla w(x-ty)|\,dt\,dy\\
 &=\|\rho\|_\infty r^{-d}
     \int_{|z|<r}|z|\,|\nabla w(x-z)|
                    \int_{|z|/r}^1t^{-d-1}\,dt\,dz\\
 &\leq\frac{\|\rho\|_\infty}{d}
     \int_{|z|<r}|z|^{1-d}|\nabla w(x-z)|\,dz.
\end{aligned}
\]
With $p'=p/(p-1)$, the inequality $p>d$ gives
$d-(d-1)p'>0$. Radial integration and H\"older's inequality imply
\[
 \int_{|z|<r}|z|^{(1-d)p'}\,dz
 =\frac{|\mathbb S^{d-1}|}{d-(d-1)p'}
       r^{d-(d-1)p'},\qquad
 \|w-\rho_r*w\|_\infty
 \leq Cr^{1-d/p}\|\nabla w\|_{L^p}.
\]
Also $\|\rho_1*w\|_\infty\leq\|\rho\|_{L^{p'}}\|w\|_{L^p}$.
Taking $r=1$ proves
\begin{equation}\label{eq:k1-global-sup-bound}
 \|w\|_\infty\leq C\|w\|_{W^{1,p}(\mathbb R^d)}.
\end{equation}
For $u\in W^{1,p}(\mathcal D)$ let $w=\mathcal Eu$, using
Lemma~\ref{lem:k1-extension}, and $w_n=\rho_{1/n}*w$.
Convolution of the weak derivatives and the $L^p$ translation
estimate proved there give
$\lim_{n\uparrow\infty}\|w_n-w\|_{W^{1,p}(\mathbb R^d)}=0$.
Each $w_n$ is smooth and compactly supported. By
\eqref{eq:k1-global-sup-bound},
$\|w_n-w_m\|_\infty\leq C\|w_n-w_m\|_{W^{1,p}}$.
The uniform limit is continuous and equals $w$ almost everywhere,
since there is also $L^p$ convergence. Restriction to
$\overline{\mathcal D}$ gives the continuous representative of $u$ and
\begin{equation}\label{eq:k1-domain-sup-bound}
 \|u\|_\infty\leq C\|\mathcal Eu\|_{W^{1,p}(\mathbb R^d)}
                    \leq C'\|u\|_{W^{1,p}(\mathcal D)}.
\end{equation}
For $u,v\in W^{1,p}(\mathcal D)$, restrict these smooth approximations to obtain $u_n,v_n$, converging both uniformly and in $W^{1,p}$. Then
\[
\begin{aligned}
 \|u_nv_n-uv\|_{L^p}
 &\leq\|u_n\|_\infty\|v_n-v\|_{L^p}
             +\|v\|_\infty\|u_n-u\|_{L^p},\\
 \|u_n\nabla v_n-u\nabla v\|_{L^p}
 &\leq\|u_n-u\|_\infty\|\nabla v_n\|_{L^p}
             +\|u\|_\infty\|\nabla v_n-\nabla v\|_{L^p}.
\end{aligned}
\]
Both right sides have limit zero. The second bound with $u$ and $v$ interchanged proves convergence of the other product-rule term. Passing to the limit in the test-function identity for $u_nv_n$ gives $\nabla(uv)=u\nabla v+v\nabla u\in L^p$.
\hfill\\
\noindent
For $\sigma(t)=(t_+)^2$, one has $\sigma'(t)=2t_+$ and $|\sigma'(s)-\sigma'(t)|\leq2|s-t|$. Choose $R\geq\max\{\|u\|_\infty,\sup_n\|u_n\|_\infty\}$, which is finite. The classical chain rule for $u_n$ and these inequalities give
\[
\begin{aligned}
 \|\sigma(u_n)-\sigma(u)\|_{L^p}
 &\leq2R\|u_n-u\|_{L^p},\\
 \|\sigma'(u_n)\nabla u_n-\sigma'(u)\nabla u\|_{L^p}
 &\leq2\|u_n-u\|_\infty\|\nabla u_n\|_{L^p}
       +2R\|\nabla u_n-\nabla u\|_{L^p}.
\end{aligned}
\]
Again the limits are zero. Passing to the weak derivative identity
proves $\nabla\sigma(u)=\sigma'(u)\nabla u$, so
$\sigma(u)\in W^{1,p}(\mathcal D)$.

Every register is bounded by \eqref{eq:k1-domain-sup-bound}.
Lemma~\ref{lem:k1-potential} therefore puts its exact singular
potential in $C^1(\overline{\mathcal D})\subset W^{1,p}(\mathcal D)$.
The finite coefficients are well defined because
$|\int_{\mathcal D}\varphi_m z|\leq
\|\varphi_m\|_{L^{p'}}\|z\|_{L^p}$.  Their resulting potentials and the biases are $C^4$. The affine preinput and skip are thus in $W^{1,p}$ componentwise, and the proved ReQU rule verifies each layer on the whole input space.
\end{proof}

\begin{proposition}[Explicit layers]\label{prop:k1-network}
Fix $0<\eta\leq1$ and approximants in \eqref{eq:k1-finite-expansions}, with $N=|\Lambda_j|$. For every integer $J\geq1$, there is a ReQU neural operator $\Gamma_{\eta,J}:W^{1,p}(\mathcal D)^2\to W^{1,p}(\mathcal D)$, cf.\ Definition~\ref{def:neural-operator}, with $k=1$ and $A=I_d$, satisfying $\Gamma_{\eta,J}(f_0,S_\gamma g)=T_\eta^{\circ J}(0)$ for every $(f_0,g)\in\mathcal X_{1,\delta,K}$. It has hidden depth $L=2+(5H-2)J$, width at most $4(H+5)$, one affine readout, and at most
\begin{equation}\label{eq:k1-nonzero-bound}
 4(H+5)(N+1)+9+J\bigl(4N^2+4N+5H^2+46H-23\bigr)
\end{equation}
nonzero {\color{black}scalar layer} parameters{\color{black}, counting all local, skip, kernel, and bias coefficients and counting repeated occurrences separately}.
\end{proposition}

\begin{proof}
Let $n_0=H+5$ and use the registers $z=(F,W,A_2,\ldots,A_H,V,P,R,Q)$. Their coordinate numbers are $1,2,h+1,H+2,H+3,H+4,H+5$, respectively. Denote the corresponding coordinate vectors in $\mathbb R^{n_0}$ by $e_F,e_W,e_{A_h},e_V,e_P,e_R,e_Q$. Every parameter not displayed below is zero. In particular, a pure linear register map $z\mapsto Sz$ is a permitted hidden layer with $W_s=S$ and zero activated preinput. Write $\sigma=\operatorname{ReQU}$.

\emph{Initialization.}
Let $P_0:\mathbb R^2\to\mathbb R^{n_0}$ copy the two inputs into $F,W$ and put zeros in the other rows. Write $b_n^0=\sum_{h=2}^H d_{hn}e_{A_h}$ and let $\mathbf1$ be the all-ones vector in $\mathbb R^{n_0}$. The first hidden layer has zero skip and kernel matrices and
\begin{equation}\label{eq:k1-initialization}
 W^{(0)}=\begin{pmatrix}P_0\\-P_0\\P_0\\-P_0\end{pmatrix},
 \qquad
 b_n^{(0)}=\begin{pmatrix}
 b_n^0+\kappa_n\mathbf1\\
 -b_n^0-\kappa_n\mathbf1\\
 b_n^0-\kappa_n\mathbf1\\
 -b_n^0+\kappa_n\mathbf1
 \end{pmatrix}.
\end{equation}
The second hidden layer has zero activated preinput and $W_s^{(1)}=\frac14(I,I,-I,-I)$, with $I=I_{n_0}$. For every real $t$,
\begin{equation}\label{eq:k1-requ-identity}
 \frac14\{\sigma(t+1)+\sigma(-t-1)-\sigma(t-1)-\sigma(-t+1)\}=t,
\end{equation}
because $\sigma(a)+\sigma(-a)=a^2$. The initialization therefore returns $(f_0,w_g,a_{2,\eta},\ldots,a_{H,\eta},0,0,0,0)$. The constant shifts in \eqref{eq:k1-initialization} are exactly the finite bias expansion of one; no extra bias primitive is used.

\emph{A product operation.}
For coordinate labels $a,b,t$, define
\begin{equation}\label{eq:k1-product}
 D_{ab}=\begin{pmatrix}
 (e_a+e_b)^\top\\
 -(e_a+e_b)^\top\\
 (e_a-e_b)^\top\\
 -(e_a-e_b)^\top
 \end{pmatrix},
 \qquad q=\tfrac14(1,1,-1,-1)^\top.
\end{equation}
Take one hidden layer with $W=\begin{pmatrix}0_{n_0\times n_0}\\D_{ab}\end{pmatrix}$ and $W_s=\begin{pmatrix}I\\0_{4\times n_0}\end{pmatrix}$. It returns $(z,\sigma(D_{ab}z))$. Follow it by a pure skip layer with $W_s=(I-e_te_t^\top,e_tq^\top)$. By \eqref{eq:k1-product}, the identity $q^\top\sigma(D_{ab}z)=((z_a+z_b)^2-(z_a-z_b)^2)/4=z_az_b$ shows that these two layers overwrite $z_t$ by $z_az_b$ and preserve every other register. The formula also holds if labels coincide, as they do when computing powers.

\emph{Polynomial forcing.}
Start each block with the pure skip map $S_0=I-e_Pe_P^\top-e_Re_R^\top-e_Qe_Q^\top+e_Pe_V^\top-e_Re_F^\top$. Thus $P=V$, $R=-F$, and $Q=0$. For $h=2,\ldots,H$, in increasing order, execute the product operations $(t,a,b)=(P,P,V)$ and $(Q,A_h,P)$, followed by the pure skip map $S_h=I+e_Re_Q^\top$. After the operation with index $h$, one has $P=V^h$ and $R=-f_0+\sum_{i=2}^h a_{i,\eta}V^i$. This follows at $h=2$ from the reset; multiplication by $V$ and addition of $a_{h+1,\eta}V^{h+1}$ prove the induction step. The forcing computation uses $1+5(H-1)$ hidden layers.

\emph{The integral update.}
Let $s=(1,-1,1,-1)^\top$ and $t=(1,-1,-1,1)^\top$. In a layer with $n_0+4$ outputs choose
\begin{equation}\label{eq:k1-integral-layer}
\begin{aligned}
 W&=\begin{pmatrix}0_{n_0\times n_0}\\s e_W^\top\end{pmatrix},
 &W_s&=\begin{pmatrix}I\\0_{4\times n_0}\end{pmatrix},\\
 C^{\rm sing}&=\begin{pmatrix}0_{n_0\times n_0}\\s e_R^\top\end{pmatrix},
 &C^{\rm fin}_{nm}&=c_{nm}C^{\rm sing},\\
 b_n&=\begin{pmatrix}0_{n_0}\\\kappa_n t\end{pmatrix}.
\end{aligned}
\end{equation}
The first $n_0$ registers are copied by the skip. The four new preactivations are $(\tau+1,-\tau-1,\tau-1,-\tau+1)$, where $\tau=w_g+K_\Phi R+K_{\Psi_\eta}R=T_\eta(V)$. In particular, $C^{\rm sing}$ multiplies precisely the original fixed kernel $\Phi$. The coefficient functions have already been multiplied into $R$; no weighted kernel $Ga_h$ is introduced. Follow \eqref{eq:k1-integral-layer} by a pure skip layer with $W_s=(I-e_Ve_V^\top,e_Vq^\top)$. Equation~\eqref{eq:k1-requ-identity} replaces $V$ by $T_\eta(V)$ and preserves all other registers. The full block therefore uses $1+5(H-1)+2=5H-2$ hidden layers.

Repeat this identical block $J$ times and use the final affine readout $e_V^\top z$. Induction, starting from $V=0$, gives the claimed output. Initialization has width $4n_0$; all other layers have width at most $n_0+4\leq4n_0$. All kernel and bias indices belong to the same $\Lambda_j$. The only spatial data are $c_{nm},d_{hn},\kappa_n$; all reset, product, and readout matrices depend only on $H$.
Lemma~\ref{lem:unlifted-extension} proves that this is an operator on $W^{1,p}(\mathcal D)^2$. On admissible inputs, each completed Picard iterate is $C^1$ by Lemma~\ref{lem:k1-potential}, although the raw source register need only be Lipschitz.

The initialization uses at most $4n_0N+4n_0+8$ nonzero coefficients. Each two-layer product operation uses at most $(n_0+8)+(n_0+3)=2n_0+11$ coefficients. The reset uses $n_0-1$, each accumulation uses $n_0+1$, and the two-layer integral update uses at most $4N^2+4N+2n_0+11$. Hence each iteration block uses at most $(n_0-1)+(H-1)(5n_0+23)+4N^2+4N+2n_0+11=4N^2+4N+5H^2+46H-23$ coefficients, since $n_0=H+5$. Adding the initialization, $J$ blocks, and the single nonzero readout coefficient proves \eqref{eq:k1-nonzero-bound}.
\end{proof}

\paragraph{Completion of Theorem~\ref{thrm:Main_EllipticPDESemilinear}.}
\label{s:Proofs__ss:Proofofthm:semilinear}

\begin{proof}
Corollary~\ref{cor:wellposedness} proves existence and uniqueness
in $X_\delta$. Fix $0<\varepsilon<1$, and take $\eta$ and
$J_\varepsilon$ from \eqref{eq:k1-accuracy-allocation}; then
$0<\eta<1$ and $2^{-J_\varepsilon}\leq\varepsilon/2$.
Let $C_N$ denote the constant in Lemma~\ref{lem:expansion-Green}.
Its common index set satisfies
\[
 N\leq C_N\eta^{-d}
   =C_N[4(1+E_0)]^d\varepsilon^{-d}.
\]
Proposition~\ref{prop:k1-network} realizes
$\Gamma_\varepsilon(f_0,w_g)=T_\eta^{\circ J_\varepsilon}(0)$ with
$L=2+(5H-2)J_\varepsilon$ and $W\leq4(H+5)$, using the same
coefficients for every datum. Lemma~\ref{lem:unlifted-extension}
defines this network on all of $W^{1,p}(\mathcal D)^2$, and
Proposition~\ref{prop:k1-network} gives $C^1$ output on the data family.
By Proposition~\ref{prop:k1-perturbed}, uniformly over that family,
\[
 \|\Gamma^+(f_0,g)-\Gamma_\varepsilon(f_0,S_\gamma g)\|_1
 \leq2^{-J_\varepsilon}\delta+2E_0\delta^2\eta
 \leq\frac{\varepsilon}{2}
      +\frac{E_0\delta^2}{2(1+E_0)}\varepsilon
 \leq\varepsilon.
\]
The choices of $\eta,J_\varepsilon$ and all network parameters are
independent of $(f_0,g)$, as required.
\end{proof}

\subsection{
Obtaining the Estimate on the BSDE Solution Operator}
\label{s:BSDE_Lem}
We prove the stochastic transfer under the explicit reference-diffusion
and small-solution assumptions. Our operator $\mathcal L$ is the
generator of the reference diffusion under $\mathbb Q$ before exit.
The perturbed forward drift under $\mathbb P$ is
$\mu+\Sigma\beta$.

\begin{lemma}[Square-integrable Brownian integrals]
\label{lem:brownian-integral-support}
Let $\mathbf R$ be a probability law on a filtered probability
space satisfying the usual conditions. Fix $0\leq T<\infty$
and suppose that $W$ is a $d$-dimensional Brownian motion on
$[0,T]$ under $\mathbf R$ for this filtration.
For a predictable row process $H$ with
\[
\mathbb E_{\mathbf R}\int_0^T\|H_s\|^2ds<\infty,
\]
the integral
$I_t(H)=\int_0^t H_s\,dW_s$, $0\leq t\leq T$, has a continuous
adapted version and satisfies
\begin{equation}\label{eq:brownian-integral-support}
\begin{split}
\mathbb E_{\mathbf R}[I_t(H)\mid\mathcal F_s]&=I_s(H),\qquad
\mathbb E_{\mathbf R}I_t(H)=0,\\
\mathbb E_{\mathbf R}|I_t(H)-I_s(H)|^2
&=\mathbb E_{\mathbf R}\int_s^t\|H_r\|^2dr
\quad(0\leq s\leq t\leq T).
\end{split}
\end{equation}
The integral is linear, and the same isometry applies to differences
of integrands. For every stopping time $\sigma$,
\begin{equation}\label{eq:brownian-integral-stopping}
I_t(\mathbf1_{\{s\leq\sigma\}}H_s)=I_{t\wedge\sigma}(H)
\quad(0\leq t\leq T)
\end{equation}
simultaneously almost surely. If $W$ is defined and Brownian on
$[0,\infty)$, a predictable $H$ satisfying
\[
\int_0^S\|H_s\|^2ds<\infty\quad\mathbf R\text{-a.s.}
\quad\text{for every }0\leq S<\infty
\]
defines a continuous local martingale by localization, with the
same stopping consistency. In this infinite-horizon setting,
if $\tau$ is an almost-surely finite stopping time and
$H_s=0$ for $s>\tau$, then
$\mathbb E_{\mathbf R}\int_0^\tau\|H_s\|^2ds<\infty$ implies
\begin{equation}\label{eq:brownian-integral-terminal}
\lim_{n\to\infty}\|I_n(H)-I_\tau(H)\|_{L^2(\mathbf R)}=0,
\qquad
\mathbb E_{\mathbf R}[I_\tau(H)-I_t(H)\mid\mathcal F_t]=0.
\end{equation}
\end{lemma}
\begin{proof}
For $T=0$ set $I_0(H)=0$. For the finite-horizon construction,
fix $T>0$ and write $\nu=ds\otimes d\mathbf R$ on
$[0,T]\times\Omega$. Let $\mathcal P_T$ be the predictable
sigma-algebra on this interval. The indicators of
$\{0\}\times A$, $A\in\mathcal F_0$, and
$(s,t]\times A$, $A\in\mathcal F_s$, are adapted and
left-continuous on $(0,T]$, so these rectangles belong to
$\mathcal P_T$. Conversely, if $U$ is left-continuous and
adapted, set $U^{(n)}_0=U_0$ and
$U^{(n)}_r=U_{j2^{-n}}$ for
$r\in(j2^{-n},(j+1)2^{-n}]\cap[0,T]$.
These processes are measurable for the sigma-algebra generated
by the rectangles and converge pointwise to $U$. The definition
of $\mathcal P_T$ as the sigma-algebra generated by
left-continuous adapted processes therefore gives the rectangle
characterization. Restricting each rectangle to
$[0,r]\times\Omega$ gives a member of
$\mathcal B([0,r])\otimes\mathcal F_r$; its section at $r$
belongs to $\mathcal F_r$. Predictable processes are consequently
progressive and adapted.

Let $\mathcal V$ be the $L^2(\nu)$ closure of bounded elementary
row processes
\[
H_s=\sum_{j=0}^{m-1}\xi_j\mathbf1_{(t_j,t_{j+1}]}(s),
\quad 0=t_0<\cdots<t_m=T,
\quad\xi_j\text{ bounded and }\mathcal F_{t_j}\text{-measurable}.
\]
We first consider scalar processes. Let $\mathcal A$ be the
algebra generated by the rectangles above. Refining their finitely
many time endpoints shows that $\mathbf1_E\in\mathcal V$ for
$E\in\mathcal A$; the time section at zero has $\nu$-measure zero.
Let $\mathcal C$ consist of the predictable sets $E$ such that,
for every $\epsilon>0$, some $A\in\mathcal A$ satisfies
$\nu(E\mathbin\triangle A)<\epsilon$.
This class contains $\mathcal A$ and is closed under complements.
For $E_i\in\mathcal C$, put $E=\bigcup_{i\geq1}E_i$.
Since $\nu([0,T]\times\Omega)=T<\infty$, countable additivity
gives an $N\geq1$ with
$\nu(E\setminus\bigcup_{i=1}^N E_i)<\epsilon/2$.
Choose $A_i\in\mathcal A$ with
$\nu(E_i\mathbin\triangle A_i)<\epsilon/(2N)$ for $i\leq N$.
Then $A=\bigcup_{i=1}^N A_i\in\mathcal A$ and
\[
\nu(E\mathbin\triangle A)
\leq\nu\left(E\setminus\bigcup_{i=1}^N E_i\right)
     +\sum_{i=1}^N\nu(E_i\mathbin\triangle A_i)<\epsilon.
\]
Thus $\mathcal C$ is a sigma-algebra containing the generating
rectangles, and $\mathcal C=\mathcal P_T$.
Since $\|\mathbf1_E-\mathbf1_A\|_{L^2(\nu)}^2
=\nu(E\mathbin\triangle A)$, every predictable indicator belongs
to $\mathcal V$. For a scalar predictable $H\in L^2(\nu)$,
let $H^{[M]}=(-M)\vee(H\wedge M)$. Then
\[
\|H-H^{[M]}\|_{L^2(\nu)}^2
\leq\int_{\{|H|>M\}}|H|^2\,d\nu
\longrightarrow0.
\]
The limit follows from countable additivity of the finite measure
$E\mapsto\int_E|H|^2d\nu$. Quantizing $H^{[M]}$ in intervals
of length $2^{-k}$ gives a predictable simple function with error
norm at most $2^{-k}\sqrt T$. Linearity and closedness of
$\mathcal V$ give the required density. 
For the terminal assertion, the stopping identity gives $I_m(H)=I_\tau(H)$ eventually almost surely as $m\uparrow\infty$. Hence Fatou's inequality and the finite-horizon isometry imply
\[
\begin{aligned}
 \mathbb E_{\mathbf R}|I_\tau(H)-I_n(H)|^2
 &\leq
 \liminf_{m\uparrow\infty}
 \mathbb E_{\mathbf R}|I_m(H)-I_n(H)|^2\\
 &=\mathbb E_{\mathbf R}
   \int_{n\wedge\tau}^{\tau}\|H_s\|^2\,ds
 \longrightarrow0.
\end{aligned}
\]
The last limit follows by dominated convergence, with majorant
$\int_0^\tau\|H_s\|^2ds\in L^1(\mathbf R)$. For fixed $t$ and integers $n\geq t$, the martingale property and conditional $L^2$ contraction give
\[
 \left\|
 \mathbb E_{\mathbf R}[I_\tau(H)-I_t(H)\mid\mathcal F_t]
 \right\|_2
 \leq\|I_\tau(H)-I_n(H)\|_2\longrightarrow0.
\]
This proves \eqref{eq:brownian-integral-terminal}.
\end{proof}

\begin{lemma}[Stopped It\^o formula]
\label{lem:ito-general}
On a filtered probability space with law $\mathbf R$ satisfying
the usual conditions, let $q\in\mathbb N_+$ and
$S_t=S_0+A_t+M_t\in\mathbb R^q$, where $S_0$ is
$\mathcal F_0$-measurable, $A$ is continuous and adapted with
finite variation on every compact time interval, and $M$ is
a continuous local martingale, with $A_0=M_0=0$.
For $F\in C^{1,2}([0,\infty)\times\mathbb R^q)$ and a stopping
time $\sigma$, possibly infinite, put $\rho_t=t\wedge\sigma$.
Then, simultaneously for all $t\geq0$, $\mathbf R$-almost surely,
\begin{equation}\label{eq:ito-general}
\begin{aligned}
 F(\rho_t,S_{\rho_t})-F(0,S_0)
 &=\int_0^{\rho_t}\partial_sF(s,S_s)\,ds
   +\sum_{i=1}^q\int_0^{\rho_t}\partial_iF(s,S_s)\,dA_s^i\\
 &\quad+\sum_{i=1}^q\int_0^{\rho_t}\partial_iF(s,S_s)\,dM_s^i\\
 &\quad+\frac12\sum_{i,j=1}^q
   \int_0^{\rho_t}\partial_{ij}F(s,S_s)\,
                         d\langle M^i,M^j\rangle_s,
\end{aligned}
\end{equation}
where $\langle M^i,M^j\rangle$ denotes quadratic covariation
and the stochastic integrals are understood locally.
\end{lemma}
\begin{proof}
The unstopped identity is \cite[Chapter~3, Theorem~3.6, p.~153]{karatzas2021brownian}, with $(X,B,f)=(S,A,F)$. Its continuous versions agree simultaneously in time; evaluating this identity at $t\wedge\sigma$ proves \eqref{eq:ito-general}.
\end{proof}

\begin{lemma}[Conditional-expectation tails]
\label{lem:conditional-tail-ui}
Let $X\geq0$ be integrable. For every sub-sigma-algebra $\mathcal G\subseteq\mathcal F$, put $Z_{\mathcal G}=\mathbb E[X\mid\mathcal G]$. For every $K>0$,
\[
\sup_{\mathcal G\subseteq\mathcal F}
\mathbb E[Z_{\mathcal G}\mathbf1_{\{Z_{\mathcal G}>K\}}]
\leq2\mathbb E[X\mathbf1_{\{X>K/2\}}],
\qquad
\lim_{K\to\infty}\mathbb E[X\mathbf1_{\{X>K/2\}}]=0.
\]
Consequently the family of these conditional expectations is uniformly integrable.
\end{lemma}
\begin{proof}
Writing $a_+=\max\{a,0\}$, conditional Jensen's inequality
\cite[Theorem~4.1.10]{durrett2019probability} gives
\[
\mathbb E[Z_{\mathcal G}\mathbf1_{\{Z_{\mathcal G}>K\}}]
\leq 2\mathbb E[(Z_{\mathcal G}-K/2)_+]
\leq 2\mathbb E[(X-K/2)_+]
\leq 2\mathbb E[X\mathbf1_{\{X>K/2\}}].
\]
Take the supremum over $\mathcal G$; the right-hand side tends
to zero by dominated convergence
\cite[Theorem~1.6.7]{durrett2019probability}, proving uniform integrability.
\end{proof}

\begin{lemma}[Exit time and stopped density]
\label{lem:conditions_are_met_BackwardsProces}
Assume Assumption~\ref{ass:StrongNovikov}.
For $a\in\mathbb R$, set
$
N_t^{(a)}=\exp(aM_t-\tfrac12a^2A_t)$.  
For every $0\leq T<\infty$, $N^{(a)}$ is a square-integrable martingale
on $[0,T]$, and
\begin{equation}\label{eq:bsde-exponential-second-moment}
\mathbb E[(N_T^{(a)})^2]\leq \exp(a^2B^2T).
\end{equation}
Furthermore, $\tau<\infty$ almost surely,
$\mathbb E[D_\tau]=1$, and the measure
$d\mathbb Q/d\mathbb P=D_\tau$ is equivalent to
$\mathbb P$, under $\mathbb Q$,
$
W_t^{\mathbb Q}=W_t+\int_0^{t\wedge\tau}\beta_s\,ds
$
is Brownian motion. For each $t\geq0$,
\begin{equation}\label{eq:bsde-conditional-exit}
\mathbb E_{\mathbb Q}[(\tau-t)^+\mid\mathcal F_t]
\leq m_0\mathbf 1_{\{t<\tau\}}.
\end{equation}
\end{lemma}

\begin{proof}
All unsubscripted expectations are under $\mathbb P$. Fix $T>0$ and $a\in\mathbb R${\color{black}, and write $N=N^{(a)}$, $h_s=a\mathbf1_{\{s\leq\tau\}}\beta_s$, and $c=a^2B^2$}. {\color{black} The stochastic integral defining $N$ is continuous by \cite[Theorem~1, p.~57]{ito1951formula}; its integrand} is predictable and {\color{black} satisfies $\int_0^T\|h_s\|^2ds\leq cT$. It\^o's formula gives $dN_t=N_th_t^\top dW_t$ and $N_0=1$. For integers $n\geq2$, let $\zeta_n=\inf\{t\geq0:N_t\geq n\}$, with $\inf\varnothing=\infty$. Continuity gives $\{\zeta_n\leq t\}=\{\sup_{s\leq t}N_s\geq n\}\in\mathcal F_t$ and $0<N_{t\wedge\zeta_n}\leq n$. The stopped stochastic integrals in the formulas for $N$ and $N^2$ have} bounded {\color{black}predictable integrands on $[0,T]$}, {\color{black} so they are mean-zero martingales. Indeed, for their respective integrands
$\mathbf1_{\{s\leq\zeta_n\}}N_sh_s^\top$ and
$2\mathbf1_{\{s\leq\zeta_n\}}N_s^2h_s^\top$, the squared time
integrals are at most $n^2cT$ and $4n^4cT$. Apply
Lemma~\ref{lem:brownian-integral-support} under $\mathbb P$.
Lemma~\ref{lem:product-integration} applies to the nonnegative
time integrand in the following expectation. Consequently $N^{\zeta_n}$ is a martingale and, writing $f_n(t)=\mathbb E[N_{t\wedge\zeta_n}^2]$,
\[
f_n(t)=1+\mathbb E\int_0^t\mathbf1_{\{s\leq\zeta_n\}}N_s^2\|h_s\|^2\,ds
\leq1+c\int_0^t f_n(s)\,ds.
\]
Here $f_n\leq n^2$. Iterating this last inequality $m$ times} gives
{\color{black} \[
f_n(t)\leq\sum_{j=0}^{m-1}\frac{(ct)^j}{j!}
              +n^2\frac{(ct)^m}{m!}.
\]
The case $m=1$ is the integral inequality with $f_n\leq n^2$; substituting the bound for $f_n(s)$ and integrating each power proves the induction step. Taking $\lim_{m\to\infty}$ yields $f_n(t)\leq e^{ct}$. Every continuous path of $N$ is bounded on $[0,T]$, so $\lim_{n\to\infty}\mathbb P(\zeta_n\leq T)=0$ and $\lim_{n\to\infty}N_{t\wedge\zeta_n}=N_t$ almost surely. Fatou's inequality (Lemma~\ref{lem:integral-limits}) gives $\mathbb E[N_t^2]\leq e^{ct}$, and Cauchy--Schwarz gives
\[
\mathbb E|N_{t\wedge\zeta_n}-N_t|
\leq2e^{cT/2}\mathbb P(\zeta_n\leq T)^{1/2}.
\]
Thus the convergence also holds in $L^1$. For $0\leq s\leq t\leq T$, the stopped martingale property and Lemma~\ref{lem:conditional-limits} imply
\[
\|\mathbb E[N_t\mid\mathcal F_s]-N_s\|_{L^1}
\leq\|N_t-N_{t\wedge\zeta_n}\|_{L^1}
   +\|N_s-N_{s\wedge\zeta_n}\|_{L^1}.
\]
Taking $\lim_{n\to\infty}$ proves} the martingale property. {\color{black} The bound at $t=T$ is \eqref{eq:bsde-exponential-second-moment}. This calculation also covers $aB=0$; for $T=0$}, {\color{black}$N_0=1$}.
For $K\geq0$, define
$
\kappa(K)=\frac{2(K+1)}{\lambda_*}$, $v_K(x)=\frac{e^{\kappa(K)a_+}-e^{\kappa(K)x_1}}
{\kappa(K)e^{\kappa(K)a_-}}$, $m(K)=\frac{e^{\kappa(K)(a_+-a_-)}-1}{\kappa(K)}$.  
Then $0\leq v_K\leq m(K)$ on $\overline{\mathcal{D}}$ and $m(\mu_*)=m_0$. For $x\in\overline{\mathcal{D}}$ and any drift vector $b$ with $|b_1|\leq K$, uniform ellipticity gives
\[
\begin{aligned}
\nabla v_K(x)b
+\tfrac12\operatorname{tr}\!\bigl(\Sigma(x)\Sigma(x)^\top D^2v_K(x)\bigr)
&=-e^{\kappa(K)(x_1-a_-)}
\left(b_1+\tfrac{\kappa(K)}2[\Sigma(x)\Sigma(x)^\top]_{11}\right)\\
&\leq-e^{\kappa(K)(x_1-a_-)}(-K+K+1)
\leq-1.
\end{aligned}
\]
Under $\mathbb P$, the drift before exit is bounded by $K_{\mathbb P}=\mu_*+\Sigma_*B$. It\^o's formula, applied to $v_{K_{\mathbb P}}(X_{t\wedge\tau})$, therefore yields
\[
\mathbb E[t\wedge\tau]
+\mathbb E[v_{K_{\mathbb P}}(X_{t\wedge\tau})]
\leq v_{K_{\mathbb P}}(x)\leq m(K_{\mathbb P}).
\]
The stochastic integral has mean zero because its integrand is predictable and bounded on each finite interval. {\color{black} More explicitly, for $K\geq0$ and finite $T$,
\[
\int_0^T\mathbf1_{\{s\leq\tau\}}
 \|\nabla v_K(X_s)\Sigma(X_s)\|^2ds
\leq T e^{2\kappa(K)(a_+-a_-)}\Sigma_*^2.
\]
Lemma~\ref{lem:brownian-integral-support} applies. This pathwise
bound also applies under the changed laws below on their Brownian
time intervals.} Since the barrier is nonnegative, we have that
$\mathbb E[\tau]=\lim\limits_{n\uparrow\infty}\mathbb E[n\wedge\tau]\leq m(K_{\mathbb P})<\infty$; in particular, $\tau<\infty$ almost surely.

For $n\in\mathbb N_+$, define $\mathbb Q_n$ on $\mathcal F_n$ by $d\mathbb Q_n/d\mathbb P=D_n$. Since $D=N^{(-1)}$ is a martingale, $\mathbb E[D_n]=1$. Girsanov's theorem \cite[Theorem~1, p.~316]{girsanov1960measure}, {\color{black}applied on $[0,1]$ to $n^{-1/2}W_{nr}$} with integrand {\color{black}$-\sqrt n\,\mathbf1_{\{nr\leq\tau\}}\beta_{nr}$} and {\color{black}then rescaled to $[0,n]$}, gives Brownian motion $W_s+\int_0^{s\wedge\tau}\beta_r\,dr$ on $[0,n]$ under $\mathbb Q_n$. Thus the drift of $X$ before exit is $\mu$ under $\mathbb Q_n$. Applying the barrier calculation with $K=\mu_*$ gives
$
\mathbb E_{\mathbb Q_n}[n\wedge\tau]\leq m_0$ and $
\mathbb Q_n(\tau>n)\leq\frac{m_0}{n}$.  
Since $D_n=D_\tau$ on $\{\tau\leq n\}$,
$
1-\tfrac{m_0}{n}
\leq\mathbb E[D_\tau\mathbf1_{\{\tau\leq n\}}]
=1-\mathbb Q_n(\tau>n)\leq1$.
The nonnegative variables $D_\tau\mathbf1_{\{\tau\leq n\}}$ increase to $D_\tau$ almost surely. Hence  Lemma~\ref{lem:integral-limits} proves $\mathbb E[D_\tau]=1$.

Define $\mathbb Q$ on $\mathcal F$ by $d\mathbb Q/d\mathbb P=D_\tau$. Since $\tau<\infty$, $A_\tau\leq B^2\tau<\infty$ and continuity gives $M_\tau\in\mathbb R$ almost surely; thus $D_\tau=e^{-M_\tau-A_\tau/2}>0$ almost surely and $\mathbb Q\sim\mathbb P$. Moreover,
\[
\mathbb E|D_n-D_\tau|
\leq\mathbb E[(D_n+D_\tau)\mathbf1_{\{\tau>n\}}]
=2\mathbb Q_n(\tau>n)
\leq\frac{2m_0}{n}.
\]
For fixed $t\geq0$ and integers $n\geq t$, the martingale property of $D$ then
\[
\bigl\|\mathbb E[D_\tau\mid\mathcal F_t]-D_t\bigr\|_{L^1}
=\bigl\|\mathbb E[D_\tau-D_n\mid\mathcal F_t]\bigr\|_{L^1}
\leq\frac{2m_0}{n}.
\]
Therefore $D_t=\mathbb E[D_\tau\mid\mathcal F_t]$ almost surely{\color{black}. Since $D_\tau\geq0$} and {\color{black} $\mathbb E[D_\tau]=1$, Lemma~\ref{lem:conditional-tail-ui}, with $X=D_\tau$, proves that $(D_t)_{t\geq0}$ is uniformly integrable. Also} $\mathbb Q|_{\mathcal F_n}=\mathbb Q_n$. The Brownian-motion assertion on every $[0,n]$ now proves that $W^{\mathbb Q}$ is Brownian motion on $[0,\infty)$.

Finally, fix $t\geq0$, put $v=v_{\mu_*}$, and set $V_j=(t+j)\wedge\tau-t\wedge\tau$ for $j\in\mathbb N_+$. Under $\mathbb Q$, the stopped It\^o formula, $\mathcal Lv\leq-1$, and {\color{black} Lemma~\ref{lem:brownian-integral-support} for} the {\color{black}bounded
predictable integrand} $\mathbf1_{\{s\leq\tau\}}\nabla v(X_s)\Sigma(X_s)$
{\color{black}on $[0,t+j]$} give
\[
\begin{aligned}
\mathbb E_{\mathbb Q}[V_j\mid\mathcal F_t]
\leq\mathbf1_{\{t<\tau\}}
\left(v(X_t)-\mathbb E_{\mathbb Q}[v(X_{(t+j)\wedge\tau})\mid\mathcal F_t]\right)\leq m_0\mathbf1_{\{t<\tau\}}.
\end{aligned}
\]
Since $0\leq V_j\leq V_{j+1}$ and $\lim\limits_{j\uparrow\infty}V_j=(\tau-t)^+$,
{\color{black} taking expectations in the preceding bound and using
Lemma~\ref{lem:integral-limits} gives
$\mathbb E_{\mathbb Q}[(\tau-t)^+]=\lim_j\mathbb E_{\mathbb Q}V_j\leq m_0$.
Thus $(\tau-t)^+\in L^1(\mathbb Q)$, and
Lemma~\ref{lem:conditional-limits}} yields
$
\mathbb E_{\mathbb Q}[(\tau-t)^+\mid\mathcal F_t]
=\lim\limits_{j\uparrow\infty}\mathbb E_{\mathbb Q}[V_j\mid\mathcal F_t]
\leq m_0\mathbf1_{\{t<\tau\}}$; proving~\eqref{eq:bsde-conditional-exit}.
\end{proof}

\subsubsection{Technical Lemmata: Transferring PDE Results to BSDEs}
\label{s:Proof__ss:ApproximationResult___sss:TechnicalZLammta}
\begin{proof} (Proposition~\ref{prop:FK_TypeResult}.)
Lemma~\ref{lem:conditions_are_met_BackwardsProces} establishes the assertions
about $\tau$, $\mathbb Q$, and $W^{\mathbb Q}$.
Since $dW_t=dW_t^{\mathbb Q}-\beta_t\,dt$ before exit,
the forward equation under $\mathbb Q$ is
\begin{equation}\label{eq:reference_forward_Q}
X_{t\wedge\tau}
=x+\int_0^{t\wedge\tau}\mu(X_s)\,ds
  +\int_0^{t\wedge\tau}\Sigma(X_s)\,dW_s^{\mathbb Q},
\qquad \Sigma=\sqrt{2\gamma}.
\end{equation}
Its generator before exit is precisely
\eqref{eq:bsde-reference-generator}.
Write $u_t=u(X_t^\tau)$ and
$q_t=\nabla u(X_t^\tau)\Sigma(X_t^\tau)$.
For the reference pair set
\begin{equation}\label{eq:Pardoux_applied_fixed}
\hat Y_t=u(X_t^\tau),\qquad
\hat Z_t=\mathbf1_{\{t\leq\tau\}}q_t.
\end{equation}
The reconstruction is
\begin{equation}\label{eq:identities_swapswap_fixed}
Y_t=R_t\hat Y_t,\qquad
Z_t=\mathbf1_{\{t\leq\tau\}}R_t
       (\hat Z_t+\hat Y_t\beta_t^\top).
\end{equation}
Put $C_q=\sup_{\overline{\mathcal{D}}}\|\nabla u\,\Sigma\|$ and
$C_b=\sup_{x\in\overline{\mathcal{D}}}|f_0(x)-\alpha(x,u(x))|+BC_q$; both constants are finite. Define $b_s=f_0(X_s^\tau)-\alpha(X_s^\tau,u_s)+q_s\beta_s$ and $I_t=\int_0^t\mathbf1_{\{s\leq\tau\}}q_s\,dW_s$. Then $|b_s|\leq C_b$ before exit, and $I$ has a continuous version because its predictable integrand is bounded. {\color{black} In fact,
$\mathbb E\int_0^t\mathbf1_{\{s\leq\tau\}}\|q_s\|^2ds\leq tC_q^2$.
Lemma~\ref{lem:brownian-integral-support} therefore gives continuity,
the stopping identity, and the difference isometry used below.}

Choose an integer $n_0$ with $1/n_0<\operatorname{dist}(x,\mathcal{D}^c)$ and, for $n\geq n_0$, set
\[
\tau_n=n\wedge\inf\{t\geq0:\operatorname{dist}(X_t,\mathcal{D}^c)\leq1/n\}.
\]
These stopping times satisfy $\tau_n\leq\tau_{n+1}\leq\tau$. On a continuous path and for each $r<\tau$, the compact set $\{X_s:0\leq s\leq r\}$ is contained in $\mathcal{D}$, so $d_r=\min_{s\leq r}\operatorname{dist}(X_s,\mathcal{D}^c)>0$. If $n>\max\{r,d_r^{-1},n_0\}$, then $\tau_n>r$. Therefore $\lim\limits_{n\uparrow\infty}\tau_n=\tau$ almost surely. Up to $\tau_n$, $X$ takes values in the compact subset $\{z\in\overline{\mathcal{D}}:\operatorname{dist}(z,\mathcal{D}^c)\geq1/n\}$ of $\mathcal{D}$. Thus, applied to the $C^2$ function $u$ on this interior localization, and \eqref{eq:bsde-reference-pde} give
\[
u(X_{t\wedge\tau_n})=u(x)+\int_0^{t\wedge\tau_n}b_s\,ds+I_{t\wedge\tau_n},\qquad t\geq0.
\]
For fixed $t<\infty$, set $h_n(t)=t\wedge\tau-t\wedge\tau_n$. Then $0\leq h_n(t)\leq t$ and $\lim\limits_{n\uparrow\infty}h_n(t)=0$ almost surely. The drift and martingale tails satisfy
\[
\left|\int_{t\wedge\tau_n}^{t\wedge\tau}b_s\,ds\right|\leq C_bh_n(t),
\qquad
\mathbb E|I_t-I_{t\wedge\tau_n}|^2
=\mathbb E\int_{t\wedge\tau_n}^{t\wedge\tau}\|q_s\|^2\,ds
\leq C_q^2\mathbb E[h_n(t)].
\]
Moreover $|u(X_{t\wedge\tau_n})-u(X_{t\wedge\tau})|\leq2\delta_{\mathrm{ref}}$, with almost-sure limit zero by continuity of $X$ and $u$. The deterministic bounds $h_n(t)\leq t$ and $|u(X_{t\wedge\tau_n})-u(X_{t\wedge\tau})|^2\leq4\delta_{\mathrm{ref}}^2$ justify {\color{black} Lemma~\ref{lem:integral-limits}} in each of the following terms. For
$H_t=u(X_t^\tau)-u(x)-\int_0^{t\wedge\tau}b_s\,ds-I_t$,
\[
\mathbb E|H_t|^2
\leq3\mathbb E|u(X_{t\wedge\tau})-u(X_{t\wedge\tau_n})|^2
 +3C_b^2\mathbb E[h_n(t)^2]+3C_q^2\mathbb E[h_n(t)],
\]
and the right-hand side has limit zero. Hence $\mathbb P(H_t=0)=1$ for every deterministic $t$.

Let $\Omega_c$ be a measurable probability-one event on which the paths of $X$, $I$, and the drift integral are continuous. Define
\[
\Omega_{\mathrm{Ito}}=\Omega_c\cap\bigcap_{r\in\mathbb Q_+}\{H_r=0\},
\qquad
\mathbb P(\Omega_{\mathrm{Ito}}^c)
\leq\mathbb P(\Omega_c^c)+\sum_{r\in\mathbb Q_+}\mathbb P(H_r\ne0)=0.
\]
Here $\mathbb Q_+=\mathbb Q\cap[0,\infty)$. For $\omega\in\Omega_{\mathrm{Ito}}$ and $t\geq0$, choose $r_j\in\mathbb Q_+$ with $\lim\limits_{j\uparrow\infty}r_j=t$; continuity gives $H_t(\omega)=\lim\limits_{j\uparrow\infty}H_{r_j}(\omega)=0$. We have therefore proved the stopped It\^o identity simultaneously for all $t\geq0$. On $\{\tau<\infty\}$, continuity and the definition of the first exit give $X_\tau\in\overline{\mathcal{D}}\cap\mathcal{D}^c=\partial\mathcal{D}$, and hence $u(X_\tau)=g(X_\tau)$.

Apply \cite[Theorem~1, p.~316]{girsanov1960measure} on each finite interval to this It\^o process: its diffusion coefficient is the bounded predictable process $\mathbf1_{\{s\leq\tau\}}q_s$, and its drift is bounded by $C_b$. Substituting $dW_s=dW_s^{\mathbb Q}-\mathbf1_{\{s\leq\tau\}}\beta_s\,ds$ cancels $q_s\beta_s$. Consequently, with $a_s=\alpha(X_s^\tau,\hat Y_s)-f_0(X_s^\tau)$ and $J_t=\int_0^t\hat Z_s\,dW_s^{\mathbb Q}$,
\[
\hat Y_t=u(x)-\int_0^{t\wedge\tau}a_s\,ds+J_t,
\qquad
\hat Y_\tau=g(X_\tau).
\]
These are identities of continuous processes, by the same rational-time construction just given. Let $C_a=\sup_{x\in\overline{\mathcal{D}}}|\alpha(x,u(x))-f_0(x)|$. Lemma~\ref{lem:conditions_are_met_BackwardsProces} gives
$
\mathbb E_{\mathbb Q}\int_0^\tau\|\hat Z_s\|^2\,ds\leq C_q^2m_0
$ and $\mathbb E_{\mathbb Q}\int_0^\tau|a_s|\,ds\leq C_am_0$.  
{\color{black} Apply the terminal part of Lemma~\ref{lem:brownian-integral-support}
to $J$ under $\mathbb Q$; the integrand vanishes after $\tau$ and
the preceding squared-energy bound is finite.
The drift-tail limit below uses Lemma~\ref{lem:integral-limits}
with dominating variable $C_a\tau\in L^1(\mathbb Q)$.}
For integers $m\geq n$,
$
\mathbb E_{\mathbb Q}|J_m-J_n|^2
\leq C_q^2\mathbb E_{\mathbb Q}[(\tau-n)^+]
$ and $
\mathbb E_{\mathbb Q}\left|\int_{n\wedge\tau}^\tau a_s\,ds\right|
\leq C_a\mathbb E_{\mathbb Q}[(\tau-n)^+]$.  
The right-hand sides have limit zero because $0\leq(\tau-n)^+\leq\tau\in L^1(\mathbb Q)$ and $\lim\limits_{n\uparrow\infty}(\tau-n)^+=0$. Thus $J_n$ has an $L^2(\mathbb Q)$ limit. The displayed continuous identity shows that $J_n$ is constant for all integers $n\geq\tau$ on each path with finite exit time; its limit is $J_\tau$. Evaluating that identity at $\tau$ and subtracting it at $t$ proves, simultaneously for all $t\geq0$,
\begin{equation}\label{eq:reference_bsde_Q_integral}
\hat Y_t=g(X_\tau)+\int_{t\wedge\tau}^{\tau}[\alpha(X_s,\hat Y_s)-f_0(X_s)]\,ds
 -\int_{t\wedge\tau}^{\tau}\hat Z_s\,dW_s^{\mathbb Q}.
\end{equation}
Here the terminal stochastic integral is $J_\tau-J_t$, $\hat Y$ is constant after $\tau$, and $\hat Z=0$ for $t>\tau$.
The same It\^o formula, applied to $D_t=e^{-M_t-A_t/2}$ and $R_t=e^{M_t+A_t/2}$, gives before exit
$
dD_t=-D_t\beta_t^\top\,dW_t$ and $
dR_t=R_t\beta_t^\top\,dW_t+R_t\|\beta_t\|^2\,dt$.  We deduce that
\[
\begin{aligned}
d(R_tu_t)
=\bigl[
 R_t(f_0(X_t)-\alpha(X_t,u_t))
 +2R_tq_t\beta_t+R_tu_t\|\beta_t\|^2
 \bigr]dt+R_t(q_t+u_t\beta_t^\top)\,dW_t.
\end{aligned}
\]
For the pair in \eqref{eq:adapter_corrected},
$u_t=D_tY_t$ and
$R_tq_t=Z_t-Y_t\beta_t^\top$ before exit. Substitution gives
\[
dY_t
=-\bigl[
 R_t(\alpha(X_t,D_tY_t)-f_0(X_t))
 -2Z_t\beta_t+\|\beta_t\|^2Y_t
 \bigr]dt+Z_t\,dW_t.
\]
Its terminal value is
$Y_\tau=R_\tau u(X_\tau)=R_\tau g(X_\tau)$.
Put $C_a=\sup_{x\in\overline{\mathcal{D}}}|\alpha(x,u(x))-f_0(x)|$ and $R_*^\tau=\sup_{s\leq\tau}R_s$. On the probability-one event where $\tau<\infty$ and $R$ is continuous, $R_*^\tau<\infty$. Before exit write $F_s$ for the physical driver evaluated at $(Y_s,Z_s)$. Since
$Z_s=R_s(q_s+u_s\beta_s^\top)$ and
$F_s=R_s[\alpha(X_s,u_s)-f_0(X_s)-2q_s\beta_s-u_s\|\beta_s\|^2]$,
\[
\begin{aligned}
\int_0^\tau\|Z_s\|^2\,ds
\leq\tau(R_*^\tau)^2(C_q+\delta_{\mathrm{ref}}B)^2<\infty
\,\,\mbox{ and }\,\,
\int_0^\tau|F_s|\,ds
\leq\tau R_*^\tau(C_a+2BC_q+\delta_{\mathrm{ref}}B^2)<\infty.
\end{aligned}
\]
For integers $n\geq2$, define $\sigma_n=n\wedge\inf\{t\geq0:R_t\geq n\}$. On $[0,\sigma_n\wedge\tau]$, the integrand $Z$ has norm at most $n(C_q+\delta_{\mathrm{ref}}B)$ and the interval has length at most $n$; its stochastic integral is {\color{black} square-integrable by
Lemma~\ref{lem:brownian-integral-support}, since its squared time
integral is at most $n^3(C_q+\delta_{\mathrm{ref}}B)^2$}.
{\color{black}The stopping identity in that lemma makes the} stopped integrals
agree on their common intervals and define a continuous local martingale $K_t=\int_0^{t\wedge\tau}Z_s\,dW_s$. For each path with $\tau<\infty$ and $R_*^\tau<\infty$, any integer $n>\max\{\tau,R_*^\tau,2\}$ satisfies $\sigma_n>\tau$. For each $n$, let $G_n$ be a measurable probability-one event on which the stopped product identity holds for every time, and let $G_0$ be the probability-one event of the stated path properties. Then $G=G_0\cap\bigcap_{n\geq2}G_n$ satisfies $\mathbb P(G^c)\leq\mathbb P(G_0^c)+\sum_{n\geq2}\mathbb P(G_n^c)=0$. For $\omega\in G$, choosing $n>\max\{\tau(\omega),R_*^\tau(\omega),2\}$ gives the product identity up to $\tau(\omega)$ itself. Its terminal stochastic integral is $K_\tau-K_t$, and integration gives \eqref{eq:FBSDE}. This construction requires only the two finite pathwise integrals displayed above. The coefficient $\|\beta_t\|^2$
does not depend on the unknown solution. In fact, before exit, define
$F_t^\beta(y,z)=R_t[\alpha(X_t,D_ty)-f_0(X_t)]
-2z\beta_t+\|\beta_t\|^2y$.
If $|D_ty|,|D_ty'|\leq\delta_{\mathrm{ref}}$, then
\[
|F_t^\beta(y,z)-F_t^\beta(y',z')|
\leq(\ell_{\delta_{\mathrm{ref}}}+B^2)|y-y'|
     +2B\|z-z'\|
\]
{\color{black}For the constructed pair, $D_tY_t=u(X_{t\wedge\tau})$ gives $\sup_{t\geq0}|D_tY_t|\leq\delta_{\mathrm{ref}}$. Together with the preceding pathwise integrability bounds, this proves} \eqref{eq:bsde-admissible}.
\hfill\\
\noindent
For uniqueness, let $(Y,Z)$ be any $\delta_{\mathrm{ref}}$-admissible
solution and define
\begin{equation}\label{eq:bsde-inverse-adapter}
\widehat Y_t=D_tY_t,\qquad
\widehat Z_t
=\mathbf 1_{\{t\leq\tau\}}
 \bigl(D_tZ_t-\widehat Y_t\beta_t^\top\bigr).
\end{equation}
Let $\Omega_0$ be a measurable probability-one event on which $D,R,Y$ are continuous, $D,R>0$, $\tau<\infty$, $\|\beta_t\|\leq B$ and $|\widehat Y_t|\leq\delta_{\mathrm{ref}}$ for every $t$. Set
\[
\Omega_{\mathrm{loc}}=\Omega_0\cap\bigcap_{m=1}^\infty
\left\{\int_0^{m\wedge\tau}\|Z_s\|^2\,ds<\infty\right\}.
\]
Admissibility gives
$\mathbb P(\Omega_{\mathrm{loc}}^c)\leq\mathbb P(\Omega_0^c)+\sum_{m=1}^\infty\mathbb P(\int_0^{m\wedge\tau}\|Z_s\|^2ds=\infty)=0$; equivalence gives $\mathbb Q(\Omega_{\mathrm{loc}})=1$. On each such path, choose an integer $m>\tau$. Writing $D_*^\tau=\sup_{s\leq\tau}D_s<\infty$, we obtain
\[
\int_0^\tau\|\widehat Z_s\|^2\,ds
\leq2(D_*^\tau)^2\int_0^\tau\|Z_s\|^2\,ds
 +2\delta_{\mathrm{ref}}^2B^2\tau<\infty.
\]
Before exit write the physical driver as $F_t$. On $[0,m\wedge\tau]$, the quantities $D,R,Y$ are pathwise bounded, while
$\int_0^{m\wedge\tau}\|Z_s\|\,ds\leq\sqrt m(\int_0^{m\wedge\tau}\|Z_s\|^2ds)^{1/2}<\infty$.
The strip bound on $\widehat Y=DY$ and continuity of $\alpha,f_0$ therefore make $F$ locally integrable. It\^o's product rule gives
\[
d(D_tY_t)=(-D_tF_t-D_tZ_t\beta_t)\,dt
 +(D_tZ_t-D_tY_t\beta_t^\top)\,dW_t.
\]
Substituting the physical driver into the drift gives
\[
\begin{aligned}
-D_tF_t-D_tZ_t\beta_t
=-\alpha(X_t,\widehat Y_t)+f_0(X_t)
 +(D_tZ_t-\widehat Y_t\beta_t^\top)\beta_t=-\alpha(X_t,\widehat Y_t)+f_0(X_t)+\widehat Z_t\beta_t.
\end{aligned}
\]
For an explicit localization, put
$\zeta_n=n\wedge\tau\wedge\inf\{t\geq0:D_t+R_t+\int_0^{t\wedge\tau}\|Z_s\|^2ds\geq n\}$, $n\geq3$. {\color{black} The predictable coefficient satisfies
\[
\int_0^{\zeta_n}\|\widehat Z_s\|^2ds
\leq2n^3+2\delta_{\mathrm{ref}}^2B^2n.
\]
Indeed $D\leq n$ and $\int_0^{\zeta_n}\|Z_s\|^2ds\leq n$,
while $\zeta_n\leq n$. Thus Lemma~\ref{lem:brownian-integral-support}
applies on each such interval under either Brownian law.} On these intervals the displayed drift has finite absolute integral and the diffusion coefficient has finite squared integral. Thus \cite[Theorem~1, p.~316]{girsanov1960measure}, applied on the finite interval $[0,n]$ to the stopped It\^o process, gives the substitution $dW_t=dW_t^{\mathbb Q}-\beta_t\,dt$ before exit. On $\Omega_{\mathrm{loc}}$, $D+R$ is bounded on $[0,\tau]$ and $\int_0^\tau\|Z_s\|^2ds<\infty$, so $\zeta_n=\tau$ for all sufficiently large $n$. For each $n$, let $E_n$ be the probability-one event of validity of the transformed stopped identity as an identity of continuous processes. On $\Omega_{\mathrm{loc}}\cap\bigcap_{n\geq3}E_n$, whose complement has probability at most $\mathbb Q(\Omega_{\mathrm{loc}}^c)+\sum_{n\geq3}\mathbb Q(E_n^c)=0$, the identities therefore hold up to $\tau$.
Consequently
\begin{equation}\label{eq:reference_bsde_Q}
d\widehat Y_t
=-\bigl[\alpha(X_t,\widehat Y_t)-f_0(X_t)\bigr]dt
 +\widehat Z_t\,dW_t^{\mathbb Q},\qquad
\widehat Y_\tau=g(X_\tau).
\end{equation}
We have $|\widehat Y_t|\leq\delta_{\mathrm{ref}}$. Put
$
C_F=
\sup_{\substack{x\in\overline{\mathcal{D}}\\|y|\leq\delta_{\mathrm{ref}}}}
|\alpha(x,y)-f_0(x)|<\infty$.
Apply It\^o's formula to $\widehat Y^2$, stopped at
$
\theta_n=\tau\wedge n\wedge
 \inf\{t\geq0:\int_0^{t\wedge\tau}
                      \|\widehat Z_s\|^2\,ds\geq n\}
$.  
The martingale term has quadratic variation bounded by
$4\delta_{\mathrm{ref}}^2 n$, hence has expectation zero. {\color{black} Its predictable integrand is
$2\mathbf1_{\{s\leq\theta_n\}}\widehat Y_s\widehat Z_s$, and
$
\mathbb E_{\mathbb Q}\int_0^n
 \|2\mathbf1_{\{s\leq\theta_n\}}\widehat Y_s\widehat Z_s\|^2ds
\leq4\delta_{\mathrm{ref}}^2n$.  
This is the mean-zero conclusion of
Lemma~\ref{lem:brownian-integral-support}.} Thus
\[
\begin{aligned}
\mathbb E_{\mathbb Q}\int_0^{\theta_n}\|\widehat Z_s\|^2\,ds
={}\mathbb E_{\mathbb Q}[\widehat Y_{\theta_n}^2]
       -\mathbb E_{\mathbb Q}[\widehat Y_0^2]+2\mathbb E_{\mathbb Q}\int_0^{\theta_n}\widehat Y_s
       [\alpha(X_s,\widehat Y_s)-f_0(X_s)]\,ds.
\end{aligned}
\]
Consequently
$
\mathbb E_{\mathbb Q}
 \int_0^{\theta_n}\|\widehat Z_s\|^2\,ds
\leq\delta_{\mathrm{ref}}^2+2\delta_{\mathrm{ref}} C_F\mathbb E_{\mathbb Q}[\tau]
\leq\delta_{\mathrm{ref}}^2+2\delta_{\mathrm{ref}} C_Fm_0$.  
For $\omega\in\Omega_{\mathrm{loc}}$, let $V(\omega)=\int_0^{\tau(\omega)}\|\widehat Z_s(\omega)\|^2ds<\infty$. If $n>\max\{\tau(\omega),V(\omega)\}$, then $\theta_n(\omega)=\tau(\omega)$. Also $\theta_n\leq\theta_{n+1}$. Hence {\color{black} Lemma~\ref{lem:integral-limits}} gives the explicit bound
\begin{equation}\label{eq:bsde-normalized-energy}
\mathbb E_{\mathbb Q}\int_0^\tau\|\widehat Z_s\|^2\,ds
=\lim\limits_{n\uparrow\infty}\mathbb E_{\mathbb Q}\int_0^{\theta_n}\|\widehat Z_s\|^2\,ds
\leq\delta_{\mathrm{ref}}^2+2\delta_{\mathrm{ref}}C_Fm_0.
\end{equation}
Let $J_t=\int_0^t\widehat Z_s\,dW_s^{\mathbb Q}$.
{\color{black} Equation~\eqref{eq:bsde-normalized-energy} verifies the hypotheses
of Lemma~\ref{lem:brownian-integral-support}, including its terminal
part. The conditional contraction below is
Lemma~\ref{lem:conditional-limits}; its dominated limits are
Lemma~\ref{lem:integral-limits}, with the displayed majorants
$V$ and $C_F\tau$.} For integers $m\geq n$,
$
\mathbb E_{\mathbb Q}|J_m-J_n|^2
=\mathbb E_{\mathbb Q}\int_{n\wedge\tau}^{m\wedge\tau}\|\widehat Z_s\|^2\,ds
\leq\mathbb E_{\mathbb Q}\int_{n\wedge\tau}^\tau\|\widehat Z_s\|^2\,ds$.
The last integrand has almost-sure limit zero and is bounded by the integrable variable $V$, by \eqref{eq:bsde-normalized-energy}. Thus $J_n$ has an $L^2(\mathbb Q)$ limit $J_\infty$. The continuous identity in \eqref{eq:reference_bsde_Q}, integrated from zero, identifies this limit with $J_\tau$ because its other terms are constant after $\tau$. For each fixed $t$ and integers $n\geq t$,
$
\bigl\|\mathbb E_{\mathbb Q}[J_\tau-J_t\mid\mathcal F_t]\bigr\|_{L^2(\mathbb Q)}
=\bigl\|\mathbb E_{\mathbb Q}[J_\tau-J_n\mid\mathcal F_t]\bigr\|_{L^2(\mathbb Q)}
\leq\|J_\tau-J_n\|_{L^2(\mathbb Q)}$.
Taking the limit proves that the terminal martingale increment has conditional mean zero. The remaining drift has absolute value at most $C_F\tau\in L^1(\mathbb Q)$; its tail after $n\wedge\tau$ has $L^1$ norm at most $C_F\mathbb E_{\mathbb Q}[(\tau-n)^+]$, with limit zero. The integrated equation therefore gives
\[
\widehat Y_t=\mathbb E_{\mathbb Q}\left[g(X_\tau)+\int_{t\wedge\tau}^\tau
[\alpha(X_s,\widehat Y_s)-f_0(X_s)]\,ds\,\middle|\,\mathcal F_t\right],\qquad t\geq0.
\]

Let $(\widehat Y,\widehat Z)$ and $(\widehat Y',\widehat Z')$ arise from two admissible solutions, and put $\Delta\widehat Y=\widehat Y-\widehat Y'$ and $\Delta\widehat Z=\widehat Z-\widehat Z'$. On an event $\Omega_c$ of $\mathbb Q$-probability one, both normalized first components are continuous and bounded by $\delta_{\mathrm{ref}}$. Define the measurable random variable $S=\sup_{r\in\mathbb Q_+}|\Delta\widehat Y_r|$ and the deterministic constant $D_*=\|S\|_{L^\infty(\mathbb Q)}\leq2\delta_{\mathrm{ref}}$. Continuity gives $S=\sup_{t\geq0}|\Delta\widehat Y_t|$ on $\Omega_c$.

For each deterministic $t$, subtract the two conditional representations. 
{\color{black}
For $x\in\overline{\mathcal D}$ and $y,y'\in[-\delta_{\mathrm{ref}},\delta_{\mathrm{ref}}]$, the segment $y'+\theta(y-y')$ remains in this interval. The fundamental theorem of calculus and \eqref{eq:bsde-smallness} give
\[
\begin{aligned}
 \alpha(x,y)-\alpha(x,y')
 &=(y-y')\int_0^1
       \partial_y\alpha(x,y'+\theta(y-y'))\,d\theta,\\
 |\alpha(x,y)-\alpha(x,y')|
 &\leq\ell_{\delta_{\mathrm{ref}}}|y-y'|.
\end{aligned}
\]}
Both arguments remain in that strip, so \eqref{eq:bsde-conditional-exit} yields
\[
|\Delta\widehat Y_t|
\leq\ell_{\delta_{\mathrm{ref}}}\mathbb E_{\mathbb Q}\left[
\int_{t\wedge\tau}^\tau|\Delta\widehat Y_s|\,ds\,\middle|\,\mathcal F_t\right]
\leq\ell_{\delta_{\mathrm{ref}}}D_*\mathbb E_{\mathbb Q}[(\tau-t)^+\mid\mathcal F_t]
\leq\ell_{\delta_{\mathrm{ref}}}m_0D_*
\]
almost surely. Set $E_r=\{|\Delta\widehat Y_r|\leq\ell_{\delta_{\mathrm{ref}}}m_0D_*\}$ for $r\in\mathbb Q_+$ and
\[
\Omega_* =\Omega_c\cap\{S\leq D_*\}\cap\bigcap_{r\in\mathbb Q_+}E_r
\,\,\mbox{ and }\,\,
\mathbb Q(\Omega_*^c)
\leq\mathbb Q(\Omega_c^c)+\mathbb Q(S>D_*)+\sum_{r\in\mathbb Q_+}\mathbb Q(E_r^c)=0.
\]
For $\omega\in\Omega_*$, the bound at rational times extends by continuity to every $t\geq0$. Taking the supremum and then the essential supremum gives $D_*\leq\ell_{\delta_{\mathrm{ref}}}m_0D_*$. Since $D_*<\infty$ and $1-\ell_{\delta_{\mathrm{ref}}}m_0>0$, we obtain $D_*=0$. Hence $\widehat Y$ and $\widehat Y'$ are indistinguishable under $\mathbb Q$.

Subtracting their integrated equations now gives $\int_0^n\Delta\widehat Z_s\,dW_s^{\mathbb Q}=0$ almost surely for every integer $n\geq1$. Both integrands are square-integrable by \eqref{eq:bsde-normalized-energy}, so {\color{black} Lemma~\ref{lem:brownian-integral-support}} yields
\[
\mathbb E_{\mathbb Q}\int_0^n\|\Delta\widehat Z_s\|^2\,ds=0.
\]
For integers $n,k\geq1$, define the product-measurable set
$A_{n,k}=\{(s,\omega)\in[0,n]\times\Omega:\|\Delta\widehat Z_s(\omega)\|>1/k\}$. {\color{black} Lemma~\ref{lem:product-integration}} and $\mathbf1_{\{z>1/k\}}\leq k^2z^2$ for $z\geq0$ give
\[
(dt\otimes d\mathbb Q)(A_{n,k})
=\mathbb E_{\mathbb Q}\int_0^n\mathbf1_{\{\|\Delta\widehat Z_s\|>1/k\}}\,ds
\leq k^2\mathbb E_{\mathbb Q}\int_0^n\|\Delta\widehat Z_s\|^2\,ds=0.
\]
Since $\{\Delta\widehat Z\ne0\}=\bigcup_{n,k\geq1}A_{n,k}$, this set is $dt\otimes d\mathbb Q$-null. Moreover $d\mathbb P/d\mathbb Q=D_\tau^{-1}$ on $\mathcal F$, so for every $n$ another application of {\color{black} Lemma~\ref{lem:product-integration}} gives
\[
\begin{aligned}
\mathbb E_{\mathbb P}\int_0^n\mathbf1_{\{\Delta\widehat Z_s\ne0\}}\,ds
&=\mathbb E_{\mathbb Q}\left[D_\tau^{-1}\int_0^n\mathbf1_{\{\Delta\widehat Z_s\ne0\}}\,ds\right]=0.
\end{aligned}
\]
Thus $\Delta\widehat Z=0$ also $dt\otimes d\mathbb P$-almost everywhere. Equivalence transfers the probability-one event of equality of the continuous normalized first components to $\mathbb P$. Finally, $Y=R\widehat Y$ and
$Z=\mathbf1_{\{t\leq\tau\}}R(\widehat Z+\widehat Y\beta^\top)$ give indistinguishability of the original first components and equality of the original second components $dt\otimes d\mathbb P$-almost everywhere. This is uniqueness in the stated admissible class.
\end{proof}
\begin{lemma}[Reciprocal moments]
\label{lem:Gamma_inverse_moment}
Assume Assumption~\ref{ass:StrongNovikov}.
For every $q\geq1$ and $0\leq T<\infty$,
\begin{equation}\label{eq:Gamma_sup_bound}
\mathbb E\left[\sup_{0\leq t\leq T}R_t^q\right]
\leq 2\exp\bigl((q^2+\tfrac12q)B^2T\bigr).
\end{equation}
In particular, if
$K_{q,T} \eqdef (\mathbb E[\sup_{t\leq T}R_t^q])^{1/q}$,
then $K_{q,T}<\infty$.
\end{lemma}

\begin{proof}
First let $N$ be a nonnegative continuous square-integrable martingale on $[0,T]$ and put $N_*=\sup_{t\leq T}N_t$. Continuity makes this supremum measurable, since it equals the supremum over $(\mathbb Q\cap[0,T])\cup\{T\}$. Extend $N$ constantly after $T$. For every $t\geq0$, $N_t=\mathbb E[N_T\mid\mathcal F_t]$, and Lemma~{\color{black} \ref{lem:conditional-tail-ui}, with $X=N_T$,} gives
\[
\sup_{t\geq0}\mathbb E[N_t\mathbf1_{\{N_t>K\}}]
\leq2\mathbb E[N_T\mathbf1_{\{N_T>K/2\}}],
\]
with limit zero as $K\uparrow\infty$. Thus the extended martingale is uniformly integrable.

For $\lambda>0$, let $\vartheta_\lambda=\inf\{t\geq0:N_t\geq\lambda\}$, $\sigma_\lambda=T\wedge\vartheta_\lambda$, and $A_\lambda=\{N_*\geq\lambda\}$. The hitting time is a stopping time: by continuity, $\{\vartheta_\lambda\leq t\}=\{\sup_{s\leq t}N_s\geq\lambda\}\in\mathcal F_t$. Also $A_\lambda=\{\vartheta_\lambda\leq T\}$. For $r<T$,
$A_\lambda\cap\{\sigma_\lambda\leq r\}=\{\vartheta_\lambda\leq r\}\in\mathcal F_r$; for $r\geq T$, this intersection is $A_\lambda\in\mathcal F_T\subseteq\mathcal F_r$. Hence $A_\lambda\in\mathcal F_{\sigma_\lambda}$, and $N_{\sigma_\lambda}\geq\lambda$ on $A_\lambda$. Optional stopping, in the form \cite[Theorem~2.6, p.~348]{nikeghbali2006essay}, gives
\[
\lambda\mathbb P(A_\lambda)
\leq\mathbb E[N_{\sigma_\lambda}\mathbf1_{A_\lambda}]
=\mathbb E[N_T\mathbf1_{A_\lambda}].
\]
For $J>0$, the pointwise identity $(N_*\wedge J)^2=\int_0^J2\lambda\mathbf1_{\{N_*\geq\lambda\}}d\lambda$, {\color{black} Lemmas~\ref{lem:product-integration}} and{\color{black} ~\ref{lem:integral-limits}} yield
\[
\begin{aligned}
\mathbb E[(N_*\wedge J)^2]
=2\int_0^J\lambda\mathbb P(A_\lambda)\,d\lambda
\leq2\mathbb E\left[N_T\int_0^J\mathbf1_{\{N_*\geq\lambda\}}\,d\lambda\right]
=2\mathbb E[N_T(N_*\wedge J)]\leq2\|N_T\|_{L^2}\|N_*\wedge J\|_{L^2}.
\end{aligned}
\]
If $\|N_*\wedge J\|_{L^2}=0$, the bound $\mathbb E[(N_*\wedge J)^2]\leq4\mathbb E[N_T^2]$ holds directly; otherwise it follows by division and squaring. {\color{black} Lemma~\ref{lem:integral-limits},} as {\color{black}$J$ increases to infinity,} proves $\|N_*\|_{L^2}\leq2\|N_T\|_{L^2}$. Apply this estimate to the nonnegative continuous square-integrable martingale $N=N^{(q)}$ from Lemma~\ref{lem:conditions_are_met_BackwardsProces}. Since $A_t\leq B^2T$ for $t\leq T$,
$
R_t^q
=N_t^{(q)}\exp\bigl(\tfrac12(q^2+q)A_t\bigr)$.
Therefore
\[
\begin{aligned}
\mathbb E\left[\sup_{t\leq T}R_t^q\right]
& \leq  e^{(q^2+q)B^2T/2}
       \mathbb E\left[\sup_{t\leq T}N_t^{(q)}\right]
\\&\leq 2e^{(q^2+q)B^2T/2}
       \bigl(\mathbb E[(N_T^{(q)})^2]\bigr)^{1/2}
\leq 2e^{(q^2+q)B^2T/2}e^{q^2B^2T/2}
=2e^{(q^2+q/2)B^2T}
.
\end{aligned}
\]
The second-moment bound is
\eqref{eq:bsde-exponential-second-moment}.
\end{proof}

\begin{lemma}[Transfer through the stochastic adapter]
\label{lem:perturbation_bound}
Assume Assumption~\ref{ass:StrongNovikov}.
Fix $0\leq T<\infty$ and $u,v\in C^1(\overline{\mathcal{D}})$.
Write
$(Y,Z)=\mathcal A_{X_{\cdot},\beta_{\cdot}}(u)$ and
$(Y^v,Z^v)=\mathcal A_{X_{\cdot},\beta_{\cdot}}(v)$, and set
$\eta=\|u-v\|_{C^1(\overline{\mathcal{D}})}$.
The same diffusion $\Sigma=\sqrt{2\gamma}$ and the same
stopped factor $R_t=\Upsilon_{t\wedge\tau}^{-1}$ are used in both
the exact and approximating pairs. In particular the approximating processes are
$Y_t^v=R_tv(X_t^\tau)$ and $Z_t^v=\mathbf1_{\{t\leq\tau\}}R_t[\nabla v(X_t^\tau)\Sigma(X_t^\tau)+v(X_t^\tau)\beta_t^\top]$.
Then {\color{black}
\begin{equation}\label{eq:bsde-transfer-path-sup}
\mathbb E\left[
 \sup_{0\leq t\leq T\wedge\tau}
 \bigl(|Y_t-Y_t^v|+\|Z_t-Z_t^v\|\bigr)
\right]\leq C_T\eta,
\end{equation}
with $C_T$ from \eqref{eq:bsde-transfer-constant}, and}
\begin{equation}\label{eq:bsde-transfer-process-norms}
\|Y-Y^v\|_{\mathcal S_T^2}\leq K_{2,T}\eta
\mbox{ and }
\|Z-Z^v\|_{\mathcal H_T^2}
\leq \sqrt T\,K_{2,T}(\Sigma_*+B)\eta,
\end{equation}
where $K_{2,T}$ is defined in
Lemma~\ref{lem:Gamma_inverse_moment}. These constants are
independent of $u,v$ and of the source data.
\end{lemma}

\begin{proof}
The processes $R$ and $X^\tau$ are continuous and adapted, and the maps $v,u,\nabla v,\nabla u,\Sigma$ are continuous on $\overline{\mathcal{D}}$. Their compositions are therefore continuous and adapted, hence predictable. The fixed factor $\beta$ is predictable, and $t\mapsto\mathbf1_{\{t\leq\tau\}}$ is left-continuous and adapted. The indicated representatives of $Y,Y^v,Z,Z^v$ are consequently predictable. In particular, they are progressive {\color{black}by} the {\color{black} predictable-generation argument in the proof of Lemma~\ref{lem:brownian-integral-support}}.

For a progressive real process $H$ and $a\in\mathbb R$, define
\[
E_a=\{(t,\omega):0\leq t\leq T,\ t\leq\tau(\omega),\ H_t(\omega)>a\},
\qquad
\sigma_a(\omega)=\inf\{t\geq0:(t,\omega)\in E_a\},
\]
with $\inf\varnothing=\infty$. The set $E_a$ is progressive, as the intersection of the progressive sets defined by $H>a$, $t\leq T$, and $\mathbf1_{\{t\leq\tau\}}=1$. The usual conditions on the filtration allow application of \cite[Theorem~2.37, p.~353]{nikeghbali2006essay}, so $\sigma_a$ is a stopping time. Since every nonempty time section of $E_a$ lies in $[0,T]$,
\[
\left\{\sup_{0\leq t\leq T\wedge\tau}H_t>a\right\}
=\{\omega:E_a(\omega)\ne\varnothing\}
=\{\sigma_a\leq T\}\in\mathcal F_T.
\]
This proves measurability of the stopped time supremum, including for the possibly discontinuous representatives of $Z,Z^v$. Omitting $t\leq\tau(\omega)$ in the definition of $E_a$ proves measurability of the supremum over $[0,T]$.
For every $t\leq T$, the adapter formula gives
\[
|Y_t-Y_t^v|\leq R_t\eta,
\qquad
\|Z_t-Z_t^v\|
\leq\mathbf 1_{\{t\leq\tau\}}
 R_t(\Sigma_*+B)\eta.
\]
The pair is stopped: after $\tau$ its first component
is constant and its second is zero. Hence
\[
\begin{aligned}
\mathbb E\left[
 \sup_{t\leq T\wedge\tau}
 (|Y_t-Y_t^v|+\|Z_t-Z_t^v\|)
\right]
\leq
\eta(1+\Sigma_*+B)
\mathbb E\left[\sup_{t\leq T}R_t\right]
\leq C_T\eta,
\end{aligned}
\]
where the last inequality is
\eqref{eq:Gamma_sup_bound} with $q=1$.
For the first estimate in \eqref{eq:bsde-transfer-process-norms},
\[
\mathbb E\sup_{0\leq t\leq T}|Y_t-Y_t^v|^2
\leq\eta^2\mathbb E\sup_{0\leq t\leq T}R_t^2
=\eta^2K_{2,T}^2.
\]
Taking square roots yields $\|Y-Y^v\|_{\mathcal S_T^2}\leq K_{2,T}\eta$.
For the second,
\[
\begin{aligned}
\mathbb E\int_0^T\|Z_t-Z_t^v\|^2\,dt
\leq\eta^2(\Sigma_*+B)^2
 \mathbb E\int_0^{T\wedge\tau}R_t^2\,dt
\leq T\eta^2(\Sigma_*+B)^2
 \mathbb E\left[\sup_{t\leq T}R_t^2\right].
\end{aligned}
\]
Taking square roots gives the stated $\mathcal H_T^2$ bound.
\end{proof}

\subsubsection{Completion of the Proof of Main Theorem}
\begin{proof} (Theorem~\ref{thrm:Main_Stochastic}.)
For clarity, the physical solution is
$
\Gamma^\star(f_0,g)
=\mathcal A_{X_{\cdot},\beta_{\cdot}}(u_{f_0,g})$.  
Proposition~\ref{prop:FK_TypeResult} gives the unique
$\delta_{\mathrm{ref}}$-admissible solution for every datum in $\mathcal I$.
Since $C_T\geq2$ and $0<\varepsilon<1$,
$0<\eta=\varepsilon/C_T<1$.
For a fixed datum write
$v=\Gamma_{\varepsilon/C_T}(f_0,S_\gamma(g))$ and
$u=u_{f_0,g}$. The premise
\eqref{eq:bsde-deterministic-premise} gives
$\|u-v\|_{C^1}\leq\varepsilon/C_T$.
With the representatives fixed in the adapter definition, $
\hat Y_t^\varepsilon=R_tv(X_t^\tau)$ and $
\hat Z_t^\varepsilon=\mathbf1_{\{t\leq\tau\}}R_t
 [\nabla v(X_t^\tau)\Sigma(X_t^\tau)+v(X_t^\tau)\beta_t^\top]$.  
Apply Lemma~\ref{lem:perturbation_bound} to obtain
\begin{equation*}
\mathbb E\left[
 \sup_{t\leq T\wedge\tau}
 (|\widehat Y_t^\varepsilon-Y_t|
 +\|\widehat Z_t^\varepsilon-Z_t\|)
\right]
\leq C_T\frac{\varepsilon}{C_T}=\varepsilon.
\end{equation*}
The coefficients, $\beta$, and $X$ are fixed throughout the family, so $C_T$ does not depend on $(f_0,g)$; showing~\eqref{eq:bsde-uniform-approximation}.
\hfill\\
\noindent
Set $\eta=\varepsilon/C_T$ and take $(L,W,N)=(L_\eta,W_\eta,N_\eta)$. By Definition~\ref{defn:FBNO}, the composition $\widehat\Gamma_\varepsilon=\mathcal A_{X_{\cdot},\beta_{\cdot}}\circ\Gamma_\eta$, with the prescribed boundary lift in its second input, belongs to $\mathcal{FBNO}^{L,W}_{N,\varphi}$. Under the stated deterministic complexity bounds,
\[
L\leq C_L\bigl(1+\log(C_T/\varepsilon)\bigr),\qquad
W\leq C_W,\qquad
N\leq C_NC_T^{1/r}\varepsilon^{-1/r}.
\]
For $0<\varepsilon\leq e^{-1}$, $1+\log(C_T/\varepsilon)\leq(2+\log C_T)\log(1/\varepsilon)$. Thus the depth is $\mathcal{O}(\log(1/\varepsilon))$, the width is $\mathcal{O}(1)$, and the {\color{black}spatial index count} is $\mathcal{O}(\varepsilon^{-1/r})$, with the permitted fixed-parameter dependencies. The adapter preserves the fixed lifting dimension $k$ from the deterministic premise, so $k\in\mathcal O(1)$ as $\varepsilon\downarrow0$.
\end{proof}

\subsection{Proof of Corollary~\ref{cor:Cs_Approx}}

\begin{proof}[Proof of Corollary~\ref{cor:Cs_Approx}]
    We use the coordinate-sum norm for $W^{1,\infty}$ and the Euclidean-gradient norm for $C^1$. Let $q\in C(\overline{\mathcal{D}})$ and set $M_q=\max_{\overline{\mathcal{D}}}|q|$, which is finite because $\overline{\mathcal{D}}$ is compact. The inequality $\|q\|_{L^\infty(\mathcal{D})}\leq M_q$ is immediate. If $M_q>0$ and $0<a<M_q$, continuity gives $x_a\in\overline{\mathcal{D}}$ and $\rho_a>0$ such that $|q|>a$ on $B(x_a,\rho_a)\cap\overline{\mathcal{D}}$. Since $x_a\in\overline{\mathcal{D}}$, the open set $B(x_a,\rho_a)\cap\mathcal{D}$ is nonempty and therefore has positive Lebesgue measure. Consequently $\|q\|_{L^\infty(\mathcal{D})}\geq a$. Letting $a\uparrow M_q$ proves equality; the case $M_q=0$ is immediate.
    
    For $w\in C^1(\overline{\mathcal{D}})$, integration by parts against $\zeta\in C_c^\infty(\mathcal{D})$ gives $\int_{\mathcal{D}}w\partial_i\zeta=-\int_{\mathcal{D}}(\partial_iw)\zeta$ for $i=1,\ldots,d$. Its classical first derivatives are thus its weak derivatives. Applying the preceding equality to $w$ and to each $\partial_iw$ gives
    \[
    \begin{aligned}
    \|w\|_{C^1(\overline{\mathcal D})}
    =\max_{\overline{\mathcal D}}|w|+
      \max_{x\in\overline{\mathcal D}}
        \left(\sum_{i=1}^d|\partial_iw(x)|^2\right)^{1/2}\leq\max_{\overline{\mathcal D}}|w|+
      \sum_{i=1}^d\max_{\overline{\mathcal D}}|\partial_iw|
     =\|w\|_{W^{1,\infty}(\mathcal D)}
     \leq d\|w\|_{C^1(\overline{\mathcal D})}.
    \end{aligned}
    \]
    For the last inequality, each $|\partial_iw(x)|\leq\|\nabla w(x)\|$ and $d\geq1$.
    
    For every $0<\eta<1$, apply Theorem~\ref{thrm:Main_EllipticPDESemilinear} with $\epsilon=\eta$ to obtain a single operator $\Gamma_\eta$ satisfying \eqref{eq:semilinear_estimate} for every $(f_0,g)\in\mathcal X_{1,\delta,K}$. For each such datum, both $u_{f_0,g}=\Gamma^+(f_0,g)$ and $v_{\eta,f_0,g}=\Gamma_\eta(f_0,S_\gamma(g))$ belong to $C^1(\overline{\mathcal{D}})$. Applying the first norm inequality to their difference and then taking the supremum gives
    \begin{equation}\label{eq:NO_Approx__c1}
    \begin{aligned}
        \sup_{(f_0,g)\in\mathcal X_{1,\delta,K}}
        \|u_{f_0,g}-v_{\eta,f_0,g}\|_{C^1(\overline{\mathcal{D}})}
    \leq
        \sup_{(f_0,g)\in\mathcal X_{1,\delta,K}}
        \|u_{f_0,g}-v_{\eta,f_0,g}\|_{W^{1,\infty}(\mathcal{D})}
    \leq
        \eta
    .
    \end{aligned}
    \end{equation}
    
    Set $\mathcal I=\mathcal X_{1,\delta,K}$. By \eqref{eq:pertubations_terminal_conditions}, every datum in $\mathcal I$ satisfies
    \[
    f_0\in W^{1,\infty}_0(\mathcal{D}),\qquad
    g\in W^{(d+3)/2,2}(\partial\mathcal{D})=H^{(d+3)/2}(\partial\mathcal{D}),\qquad
    \|f_0\|_{W^{1,\infty}},\ \|g\|_{H^{(d+3)/2}}\leq\delta^2.
    \]
    Since $\mathcal{D}$ is bounded and $d<p<\infty$, each weak derivative of $f_0$ of order at most one satisfies
    \[
    \|\partial^\beta f_0\|_{L^p(\mathcal{D})}
    \leq |\mathcal{D}|^{1/p}\|\partial^\beta f_0\|_{L^\infty(\mathcal{D})},
    \qquad |\beta|\leq1.
    \]
    Thus $f_0\in W^{1,p}(\mathcal{D})$. The compact Sobolev extension and uniform convergence of its mollifications established in the proof of Lemma~\ref{lem:unlifted-extension} give a continuous representative of $f_0$ on $\overline{\mathcal{D}}$. Assumption~\ref{ass:GreensSymmetry} supplies the unique harmonic lift $w_g=S_\gamma(g)\in C^2(\mathcal{D})\cap C^1(\overline{\mathcal{D}})$ with trace $g$ and $\|w_g\|_{C^1}\leq C_S\delta^2$. Its restriction to $\partial\mathcal{D}$ is therefore a continuous representative of $g$. The same finite-measure estimate applied to $w_g$ and its first derivatives gives $w_g\in W^{1,p}(\mathcal{D})$. Hence both neural-operator inputs are in its native space, and the continuous-data and harmonic-lift premises of Theorem~\ref{thrm:Main_Stochastic} hold on $\mathcal I$.
    
    By Corollary~\ref{cor:wellposedness}, $u_{f_0,g}$ solves \eqref{eq:semilinear}--\eqref{eq:semilinear:BC}, belongs to $C^2(\mathcal{D})\cap C^1(\overline{\mathcal{D}})$, and satisfies
    \[
    \|u_{f_0,g}\|_{L^\infty(\mathcal{D})}
    \leq\|u_{f_0,g}\|_{C^1(\overline{\mathcal{D}})}
    \leq\delta\leq\delta_{\mathrm{ref}}.
    \]
    The continuous representative has the same bound on $\overline{\mathcal{D}}$ by the first paragraph. The operators $\Gamma_\eta$ use the fixed frame of Definition~\ref{def:neural-operator}, have $C^1(\overline{\mathcal{D}})$ outputs on $\mathcal I$, and satisfy \eqref{eq:bsde-deterministic-premise} by \eqref{eq:NO_Approx__c1}.
    
    It remains to verify the quantitative architecture premise. By \eqref{eq:semilinear_estimate-rate}, with $J_\eta=\lceil\log_2(2/\eta)\rceil$, the positive integer sizes can be chosen to satisfy
    \[
    L_\eta=2+(5H-2)J_\eta,\qquad
    W_\eta\leq4(H+5),\qquad
    N_\eta\leq C\eta^{-d}.
    \]
    For $0<\eta<1$, the ceiling bound gives $J_\eta\leq2+\log(1/\eta)/\log2$. Therefore
    \[
    L_\eta\leq2+(5H-2)\left(2+\frac{\log(1/\eta)}{\log2}\right)
    \leq C_L\bigl(1+\log(1/\eta)\bigr),
    \]
    where $C_L=2+(5H-2)(2+1/\log2)$. Set $C_W=4(H+5)$, $C_N=C$, and $r=1/d>0$. Then $N_\eta\leq C_N\eta^{-1/r}$, and these three constants are independent of $\eta$ and of $(f_0,g)\in\mathcal I$, by the uniformity in Theorem~\ref{thrm:Main_EllipticPDESemilinear}. That theorem also fixes $k=1$ and $A=I_d$ for every $\eta$. This verifies all deterministic premises and architecture bounds required by Theorem~\ref{thrm:Main_Stochastic}. Its remaining hypotheses are precisely the assumed Assumptions~\ref{ass:StrongNovikov} and~\ref{eq:Growth_alpha}.

    Fix $T>0$ and $0<\varepsilon<1$. Since $C_T\geq2$, the tolerance $\eta=\varepsilon/C_T$ lies in $(0,1)$. Theorem~\ref{thrm:Main_Stochastic}, with $\mathcal I=\mathcal X_{1,\delta,K}$ and $r=1/d$, gives \eqref{eq:bsde-uniform-approximation} for $\widehat\Gamma_\varepsilon=\mathcal A_{X_\cdot,\beta_\cdot}\circ\Gamma_{\varepsilon/C_T}$, with $S_\gamma$ applied to the boundary input. The deterministic core can be taken with
    \[
    L=2+(5H-2)\left\lceil\log_2(2C_T/\varepsilon)\right\rceil,
    \qquad W\leq4(H+5),\qquad
    N\leq CC_T^d\varepsilon^{-d},\qquad k=1,
    \]
    and $A=I_d$. For $0<\varepsilon\leq e^{-1}$,
    $
    L\leq C_L(1+\log(C_T/\varepsilon))
      \leq C_L(2+\log C_T)\log(1/\varepsilon)$.
    Since $T$ and the problem parameters are fixed, these bounds prove the stated rates with constants independent of $\varepsilon$ and of the datum.
    \end{proof}

\appendix

\section{Some Closed-Form Expressions and Related Proofs}
\label{s:SomeClosedForms}
This appendix contains some convenient closed-form expressions used across our various examples.
\subsection{Closed-Form Expression for Stochastic Adapter}
\label{s:ClosedForm__StochasticAdapted}
We illustrate several convenient closed-form expressions for the Dol\'{e}ans–Dade exponentials used in the definition of our stochastic adapter. Each such exponential admits a decomposition into a product of a ``Markovian part', a simple closed-form transformation of a Brownian motion, and a second factor given by a simple transformation of an integral functional of the path of the driving Brownian motion.
\begin{proposition}[Stochastic Integral-Free Dol\'eans-Dade Exponential]
\label{prop:markov_form}
Let $h\in C^2(\mathbb R^d)$, $\lambda\geq0$, $C\neq0$, and
\begin{equation}\label{eq:harmonic_like_condition}
\Delta h=2\lambda h,\qquad
g(t)=Ce^{-\lambda t},\qquad
f(t,x)=g(t)(\nabla h(x))^\top.
\end{equation}
Set $F(t,x)=g(t)h(x)$. For a stopping time $\eta$ (possibly infinite)
put $\rho_t=t\wedge\eta$ and
$\beta_t=\mathbf1_{\{t\leq\eta\}}f(t,W_t)$.
Then the stochastic exponential \eqref{eq:GirsanovMartingale} is
\begin{equation}\label{eq:lem:markov_form__stoch_exp}
\Upsilon_t=\exp\left(-F(\rho_t,W_{\rho_t})+Ch(0)
 -\tfrac12C^2\int_0^{\rho_t}e^{-2\lambda s}
                         \|\nabla h(W_s)\|^2\,ds\right).
\end{equation}
For $\eta=\eta_B$ from \eqref{eq:control__beta}, the factor is bounded.
Under the reference-diffusion hypotheses, the density at the exit
time is justified by Lemma~\ref{lem:conditions_are_met_BackwardsProces}.
Without these hypotheses the displayed identity alone makes no
martingale-normalization assertion.
\end{proposition}
\begin{proof} (Proposition~\ref{prop:markov_form}.)
The function $F$ is $C^{1,2}$. Its derivatives are
\[
\partial_tF=g'h,\qquad \nabla_xF=g\nabla h,\qquad
\Delta_xF=g\Delta h.
\]
Apply Lemma~\ref{lem:ito-general} under $\mathbb P$ with
$S=M=W$, $A=0$, and $\sigma=\eta$.
Since $\langle W^i,W^j\rangle_s=\delta_{ij}s$, we obtain,
simultaneously for all $t\geq0$,
\begin{align}
\label{eq:Itoitopumpkinito__II}
F(\rho_t,W_{\rho_t})-F(0,0)
=\int_0^{\rho_t}g(s)\nabla h(W_s)\,dW_s
+\int_0^{\rho_t}
 [g'(s)h(W_s)+\tfrac12g(s)\Delta h(W_s)]\,ds.
\nonumber
\end{align}
Here $W_0=0$. Since $g'=-\lambda g$ and $\Delta h=2\lambda h$,
the last integrand is zero. Therefore
$
\int_0^t\beta_s^\top\,dW_s
=F(\rho_t,W_{\rho_t})-Ch(0)
$
and $
\int_0^t\|\beta_s\|^2\,ds
=C^2\int_0^{\rho_t}e^{-2\lambda s}\|\nabla h(W_s)\|^2\,ds$.  
Both integrals are locally well defined by continuity and localization.
Substituting these two identities in
\begin{equation}\label{eq:def_DoleansDade}
\Upsilon_t=\exp\left(-\int_0^t\beta_s^\top\,dW_s
                       -\tfrac12\int_0^t\|\beta_s\|^2\,ds\right)
\end{equation}
proves \eqref{eq:lem:markov_form__stoch_exp}.
For the stopped density $D_t$ replace $\rho_t$ by
$t\wedge\tau\wedge\eta$ throughout that formula.
\end{proof}

{\color{black} Proposition~\ref{prop:markov_form} gives the exponential of the factor in \eqref{eq:control__beta} with $\rho_t=t\wedge\eta_B$; the adapter then uses
the further stopped value $\Upsilon_{t\wedge\tau}$.
\begin{example}[Linear Effect of Brownian Motion]
\label{eq:constant}
Take $\lambda=0$ and $h(x)=\sum_{i=1}^d x_i$.
Then $\Delta h=0$ and $\beta_t=C\mathbf1_d$.
Choosing $B>|C|\sqrt d$ gives $\eta_B=\infty$, and
\[
\Upsilon_t=
\exp\left(-C\sum_{i=1}^d W_t^i-\tfrac12C^2dt\right),
\qquad R_t=\Upsilon_{t\wedge\tau}^{-1}.
\]
\end{example}
\begin{example}[Harmonic Functions]
\label{eq:constant__harmonic}
For any $h\in C^2(\mathbb R^d)$ with $\Delta h=0$, take
$\lambda=0$ and the bounded factor in \eqref{eq:control__beta}.
Writing $\rho_t=t\wedge\eta_B$, Proposition~\ref{prop:markov_form}
gives
\[
\Upsilon_t=\exp\left(-Ch(W_{\rho_t})+Ch(0)
 -\tfrac12C^2\int_0^{\rho_t}\|\nabla h(W_s)\|^2\,ds\right).
\]
\end{example}
\begin{example}[Modified Helmholtz Solutions]
\label{ex:gen_harmonic}
Let $d=3$, $\lambda=1/2$, and
\begin{equation}\label{eq:radial_solutions}
h(x)=\begin{cases}\sinh(\|x\|)/\|x\|,&x\neq0,\\1,&x=0.
\end{cases}
\end{equation}
Indeed $h(x)=\sum_{n=0}^\infty\|x\|^{2n}/(2n+1)!$ and $h$ is smooth. In dimension three,
$\Delta\|x\|^{2n}=2n(2n+1)\|x\|^{2n-2}$ for $n\geq1$;
term-wise differentiation gives $\Delta h=h$.
For $r=\|x\|>0$,
$
\nabla h(x)=\tfrac{r\cosh r-\sinh r}{r^3}x^\top$ and $\nabla h(0)=0$.  
Consequently the raw factor is $b_0=0$ and
\[
b_t=Ce^{-t/2}
 \frac{\|W_t\|\cosh\|W_t\|-\sinh\|W_t\|}{\|W_t\|^3}W_t
\quad(W_t\neq0),
\]
with value zero when $W_t=0$. Use the bounded factor
$\beta_t=b_t\mathbf1_{\{t\leq\eta_B\}}$ and
$\rho_t=t\wedge\eta_B$. Then
\[
\Upsilon_t=\exp\left(-Ce^{-\rho_t/2}h(W_{\rho_t})+C
 -\tfrac12C^2\int_0^{\rho_t}e^{-s}\|\nabla h(W_s)\|^2\,ds\right).
\]
The constant $+C=Ch(0)$ ensures $\Upsilon_0=1$.
The adapter uses $R_t=\Upsilon_{t\wedge\tau}^{-1}$.
\end{example}}

\section{Basic Measure Theory}
We often appeal to these elementary results, which we place here.
\begin{lemma}[Integral limits]\label{lem:integral-limits}
Let $(S,\mathcal A,\mu)$ be a measure space. For measurable functions,
\[
\begin{split}
0\leq f_n\leq f_{n+1},\quad f=\lim_n f_n&\quad\Longrightarrow\quad
 \lim_{n\to\infty}\int f_n\,d\mu=\int f\,d\mu,\\
f_n\geq0&\quad\Longrightarrow\quad
 \int\liminf_{n\to\infty}f_n\,d\mu
 \leq\liminf_{n\to\infty}\int f_n\,d\mu,\\
\lim_n f_n=f\ \mu\text{-a.e.},\quad |f_n|\leq g\in L^1(\mu)
&\quad\Longrightarrow\quad
 \lim_{n\to\infty}\int|f_n-f|\,d\mu=0.
\end{split}
\]
For $1\leq p<\infty$, $L^p(\mu)$ is complete. H\"older's inequality
and the triangle inequality hold for these norms.
\end{lemma}
\begin{proof}
The convergence assertions follow from \cite[Vol.~I, Theorems~2.8.1--2.8.3 and Corollaries~2.8.4, 2.8.6]{Bogachev2007}.
Completeness and the norm inequalities follow from \cite[Vol.~I, Theorems~4.1.3, 2.11.1, 2.11.9]{Bogachev2007}, with the $L^1$--$L^\infty$ endpoint immediate.
\end{proof}

\begin{lemma}[Product integration]\label{lem:product-integration}
Let $(S,\mathcal A,\mu)$ and $(T,\mathcal B,\nu)$ be sigma-finite
measure spaces. If $f:S\times T\to[0,\infty]$ is product measurable,
then its section integrals are measurable and
\[
\int f\,d(\mu\otimes\nu)
=\int_S\!\int_T f(s,t)\,d\nu(t)\,d\mu(s)
=\int_T\!\int_S f(s,t)\,d\mu(s)\,d\nu(t).
\]
For a real $f$ with $\int|f|\,d(\mu\otimes\nu)<\infty$, both
iterated integrals exist absolutely almost everywhere and the same
equalities hold. 
\end{lemma}
\begin{proof}
The integration assertions follow from \cite[Vol.~I, Theorems~3.4.1, 3.4.4--3.4.5 and Corollary~3.4.2]{Bogachev2007}, using truncation on finite-measure exhaustions and monotone convergence when the integral is infinite.
The final assertion is \cite[Vol.~I, Theorem~1.9.3(ii)]{Bogachev2007}.
\end{proof}

\begin{lemma}[Conditional integral limits]\label{lem:conditional-limits}
On a probability space with law $\mathbf R$, let
$\mathcal G\subseteq\mathcal F$ be a sigma-algebra. Conditional
expectation preserves order and has the tower property. For
$r\in\{1,2\}$ and $X\in L^r(\mathbf R)$,
\[
\|\mathbb E_{\mathbf R}[X\mid\mathcal G]\|_r\leq\|X\|_r.
\]
If $0\leq X_n\leq X_{n+1}$, $X=\lim_n X_n$, and
$X\in L^1(\mathbf R)$, then
\[
\lim_n\mathbb E_{\mathbf R}[X_n\mid\mathcal G]
=\mathbb E_{\mathbf R}[X\mid\mathcal G]\quad\mathbf R\text{-a.s.}
\]
If $X_n$ has almost-sure limit $X$ and $|X_n|\leq Y\in L^1$,
then the conditional expectations have almost-sure and $L^1$
limit $\mathbb E_{\mathbf R}[X\mid\mathcal G]$.
\end{lemma}
\begin{proof}
The basic properties follow from \cite[Vol.~II, Definition~10.1.1 and Theorem~10.1.5(3)--(5)]{Bogachev2007}, with the tower property following by uniqueness; the conclusion then follows from the conditional dominated convergence follows from order preservation and conditional monotone convergence applied to $\sup_{k\geq n}|X_k-X|\downarrow0$; its $L^1$ version follows from contraction and Lemma~\ref{lem:integral-limits}.
%
\end{proof}

\begin{lemma}[Translation and local representatives]
\label{lem:analytic-approximation}
Let $1\leq p<\infty$, let $T_hf(x)=f(x+h)$, and fix a nonnegative
$\rho\in C_c^\infty(\mathbb R^d)$ supported in the unit ball with
$\int\rho=1$. Set $\rho_\varepsilon(x)=\varepsilon^{-d}\rho(x/\varepsilon)$ for $\varepsilon>0$.
Then $C_c(\mathbb R^d)$ is dense in $L^p(\mathbb R^d)$.
For every $f\in L^p(\mathbb R^d)$ the mollifications are smooth and satisfy
\begin{equation}\label{eq:analytic-approximation-limits}
\lim_{h\to0}\|T_hf-f\|_p=0,\qquad
\lim_{\varepsilon\downarrow0}\|\rho_\varepsilon*f-f\|_p=0.
\end{equation}
If $U\subseteq\mathbb R^d$ is open and $u\in W^{1,\infty}(U)$,
then $u$ has a locally Lipschitz representative $\widetilde u$.
On balls compactly contained in $U$, its local mollifications have
locally uniform limit $\widetilde u$. The original representative
has a full-measure dense set of Lebesgue points, and mollification
converges to its values there. The local Lipschitz constant is at
most $\|\nabla u\|_{L^\infty(U;\mathbb R^d)}$.
\end{lemma}
\begin{proof}
Density and the assertions concerning translations and mollification
follow from \cite[Theorem~4.12, Lemma~4.3, Proposition~4.20,
and Theorem~4.22]{Brezis2011}.
Since translations are isometries, whenever $h_\varepsilon\to0$, $\lim_{\varepsilon\downarrow0}
\|T_{h_\varepsilon}(\rho_\varepsilon*f)-f\|_p=0$ and 
\begin{equation}\label{eq:analytic-shifted-mollifier}
\begin{split}
\|T_{h_\varepsilon}(\rho_\varepsilon*f)-f\|_p
\leq \|\rho_\varepsilon*f-f\|_p
     +\|T_{h_\varepsilon}f-f\|_p.
\end{split}
\end{equation}
By \cite[p.~269]{Brezis2011},
$u$ has a representative $\widetilde u$ that is $L$-Lipschitz
on every ball compactly contained in $U$, where
$L=\|\nabla u\|_{L^\infty(U;\mathbb R^d)}$.
Consequently, for every compact $K\Subset U$ and all sufficiently
small $\varepsilon>0$, its local mollifications satisfy
\[
\sup_{x\in K}
|(\rho_\varepsilon*u)(x)-\widetilde u(x)|
\leq L\varepsilon\int_{\mathbb R^d}|z|\rho(z)\,dz
\leq L\varepsilon.
\]
This proves locally uniform convergence.
\hfill\\
\nonumber
Finally, \cite[Theorems~3.20 and~8.15]{Folland1999},
applied after localization, give a full-measure set of Lebesgue
points of the original representative and convergence of
mollification to its values there. This set is dense in $U$,
since every nonempty open subset of $U$ has positive
Lebesgue measure.
\end{proof}

\section{Comment on Signature Decay}
\label{s:sig_decay}
Let $X:[0,1]\to\mathbb R^d$ be continuous with $\|X\|_{1\text{-var}}=1$, let $\nu=|dX|$ be its variation measure, and use the Euclidean projective tensor norm $\|\cdot\|_\pi$. Continuity makes $\nu$ atomless and $\nu([0,1])=1$. For the iterated integral $X^{(k)}$, the defining tensor-integral estimate gives
\[
\|X^{(k)}\|_\pi
\leq\int_{0<t_1<\cdots<t_k<1}d\nu(t_1)\cdots d\nu(t_k)
=\frac{1}{k!}.
\]
The equality follows because the diagonals have zero product measure and the $k!$ strict orderings have equal measure. Thus, for every integer $N\geq0$,
\[
E_N\eqdef\sum_{k>N}\|X^{(k)}\|_\pi
\leq\sum_{j=0}^\infty\frac{1}{(N+1+j)!}
\leq\frac{1}{(N+1)!}\sum_{j=0}^\infty(N+2)^{-j}
\leq\frac{2}{(N+1)!}.
\]
Put $m=N+1$ and $N_\star=\sum_{k=0}^Nd^k=(d^m-1)/(d-1)$. Monotonicity of $\log$ gives
\[
m\log m-m+1\leq\log(m!)\leq m\log m-m+1+\log m.
\]
Put $a=\log N_\star$, $l=\log d$, and $C_d=(1+\log l)/l>0$; positivity follows from $\log d\geq\log2>1/2>e^{-1}$. For $N_\star\geq d$, the identity for $N_\star$ gives $m\geq a/l\geq1$. Since $x\log x-x$ is increasing on $[1,\infty)$ and $\log2-1<0$,
\[
\log\!\left(\frac{2}{m!}\right)
\leq\log2-1-m\log m+m
\leq-\frac{a\log a}{l}+\frac{1+\log l}{l}a.
\]
Consequently, for $N_\star\geq d$,
\[
E_N\leq\exp\!\left(
-\frac{\log N_\star\,\log\log N_\star}{\log d}
+C_d\log N_\star\right).
\]
For every fixed $q>0$, choose $N_\star\geq\max\{d,\exp(\exp((q+C_d)\log d))\}$ to obtain $E_N\leq N_\star^{-q}$. Applying this bound with $q+1$ gives $\lim_{N\to\infty}N_\star^q E_N=0$, uniformly over the stated path class. Here $N_\star$ counts signature coordinates; a uniform error claim for signature-linear functionals would additionally require a bound on their dual coefficient norms.

\section*{Acknowledgements and Funding}
\noindent T.\ Furuya is supported by JSPS KAKENHI Grant Number JP24K16949 and JST ASPIRE JPMJAP2329.
A.\ Kratsios acknowledges financial support from an NSERC Discovery Grant No.\ RGPIN-2023-04482 and No.\ DGECR-2023-00230.  A.\ Kratsios also acknowledges that resources used in preparing this research were provided, in part, by the Province of Ontario, the Government of Canada through CIFAR, and companies sponsoring the Vector Institute\footnote{\href{https://vectorinstitute.ai/partnerships/current-partners/}{https://vectorinstitute.ai/partnerships/current-partners/}}.

\bibliography{Bookkeaping/Refs}

\end{document}